\documentclass[12pt]{amsart}
\usepackage[english]{babel}
\usepackage{amsmath}
\usepackage{amsfonts}
\usepackage{amssymb}
\usepackage{amsopn}
\usepackage{amsthm}
\numberwithin{equation}{section}
\usepackage{hyperref}
\usepackage{mathtools} 
\mathtoolsset{showonlyrefs=true,showmanualtags=true} 
\usepackage{enumerate}
\usepackage{mathrsfs}
\usepackage{amsmath,amssymb}
\usepackage{dsfont}
\usepackage{bbold}
\usepackage{extarrows}
\usepackage{soul}

\usepackage{wasysym}
\usepackage{stmaryrd}
\usepackage{todonotes}
\usepackage{comment}
\usepackage{algpseudocode}
\usepackage[shortlabels]{enumitem}
\usepackage[symbol]{footmisc}
\usepackage{xspace}
\usepackage[arrow,matrix,curve]{xy} 
\usetikzlibrary{matrix,decorations.pathreplacing,calc}
\usepackage{listings}
\usepackage{bm}
\usepackage{faktor}

\renewcommand{\hat}{\widehat}
\renewcommand{\tilde}{\widetilde}

\newtheorem{environment}{Environment}[section]

\newtheorem{remark}[environment]{Remark}

\newtheorem{lemma}[environment]{Lemma}

\newtheorem{corollary}[environment]{Corollary}
\newtheorem{theorem}[environment]{Theorem}
\newtheorem{proposition}[environment]{Proposition}
\newtheorem{example}[environment]{Example}
\newtheorem{conjecture}[environment]{Conjecture}

\newtheorem{problem}[environment]{Problem}

\newtheorem*{notation*}{Notation}
\newtheorem*{lemma*}{Lemma}
\newtheorem*{remark*}{Remark}
\newtheorem*{corollary*}{Corollary}
\newtheorem*{theorem*}{Theorem}
\newtheorem*{proposition*}{Proposition}
\newtheorem*{claim*}{Claim}
\newtheorem*{exercise*}{Exercise}



\makeatletter
\renewcommand*\env@matrix[1][*\c@MaxMatrixCols c]{%
  \hskip -\arraycolsep
  \let\@ifnextchar\new@ifnextchar
  \array{#1}}
\makeatother

\newcommand{\Vector}[1]{\ensuremath{\boldsymbol{#1}}}

\newcommand{\C}{\mathbb{C}}
\newcommand{\N}{\mathbb{N}}

\newcommand{\kk}{\mathbb{K}}
\newcommand{\R}{\mathbb{R}}

\renewcommand{\P}{\mathbb{P}}

\newcommand{\f}{\mathbf{f}}

\renewcommand{\l}{{\lambda}}
\newcommand{\x}{\mathbf{x}}

\newcommand{\lb}{{\boldsymbol{\lambda}}}

\newcommand{\gb}{{\boldsymbol{g}}}

\newcommand{\Psk}{{\Psi_k}}
\newcommand{\Pskm}{{\Psi_{k-1}}}
\newcommand{\set}[1]{\{ {#1} \}}
\newcommand{\AUX}{\AP{\cl F}}

\newcommand{\EDD}{\mathrm{\normalfont{EDD}}}

\def\disc{\Xi}
\def\sing{\mathrm{sing}}
\newcommand{\scp}[1]{\langle{#1}\rangle}
\newcommand{\norm}[1]{\Vert{#1}\Vert}
\newcommand{\cl}[1]{\mathcal{#1}}
\newcommand{\Hdn}{\H_{d}}
\newcommand{\Hun}{\H_{1}}
\newcommand{\Sn}{{\cl S}}

\newcommand{\nap}{\nabla_{X_0}^{\perp}}
\newcommand{\napp}{\nabla_{X_0}^{\perp,2}}
\newcommand{\Ver}{\mathcal{V}_{d}}
\newcommand{\Tan}{\mathcal{T}_{d}}
\newcommand{\Disc}{\disc_{d}}
\newcommand{\ann}{\mathrm{Ann}}
\newcommand{\rank}{\mathrm{rank}}

\def\H{\cl H}
\def\Ok{{\cl O}^{k}}
\def\Odn{{\cl O}_{d}}
\def\Okdn{\Ok_{d}}
\newcommand{\floor}[1]{\left\lfloor #1 \right\rfloor}

\newcommand{\dist}{\mathrm{dist}}
\newcommand{\vspan}[1]{\langle #1 \rangle}
\newcommand{\BM}[1]{\textcolor{black}{#1}}
\newcommand{\AP}[1]{\textcolor{black}{#1}}
\newcommand{\KK}[1]{\textcolor{black}{#1}}
\newcommand{\smooth}{\mathrm{smooth}}

\newcommand{\chin}[2]{\chi^{#1}_ {#2}}

\makeatletter
\DeclareFontEncoding{LS1}{}{}
\DeclareFontSubstitution{LS1}{stix}{m}{n}
\DeclareMathAlphabet{\mathcal}{LS1}{stixscr}{m}{n}
\makeatother

\DeclareMathOperator{\Sym}{\mathcal{S\mkern-4mu y\mkern-2.5mu m\mkern-1.5mu}}

\def\llcorner{\vartriangleright}
\newcommand{\zariski}[1]{\overline{#1}^{\mathrm{Zariski}}}

\author{Khazhgali Kozhasov}

\author{Bernard Mourrain}

\author{Adam Parusi\'nski}

\address{Khazhgali Kozhasov,
Universit\'e C\^ote d’Azur, CNRS, Labo. J.-A. Dieudonn\'e, UMR CNRS 7351, Parc Valrose, 06108 Nice Cedex 02, France}
\email{khazhgali.kozhasov@univ-cotedazur.fr}

\address{Bernard Mourrain, Inria at Universit\'e C\^ote d'Azur}
\email{bernard.mourrain@inria.fr}

\address{Adam Parusi\'nski,
Universit\'e C\^ote d’Azur, CNRS, Labo. J.-A. Dieudonn\'e, UMR CNRS 7351, Parc Valrose, 06108 Nice Cedex 02, France}
\email{adam.parusinski@univ-cotedazur.fr}


\title[On the distance to discriminants and osculating varieties]{On the distance to discriminants \\and osculating varieties}

\begin{document}

\maketitle

\begin{changemargin}{1.5cm}{1.5cm}
{\bf\noindent Abstract.}
We study the problem of finding the closest singular hypersurface to a given \AP{nonsingular} one from algebro-geometric point of view. The set of singular hypersurfaces $\disc$ of a given degree is a projective variety dual to the Veronese variety. However, as was shown by Raffalli \cite{Raffalli2014}, the Bombieri-Weyl (also known as apolar) distance from a general real hypersurface to $\disc$ \BM{could be minimized at a singular} point of $\disc$. We study the singular locus of $\disc$ \BM{and its irreducible component,}   consisting of hypersurfaces with cuspidal singularities. This variety is projectively dual to the tangential variety of the Veronese variety. We compute its Euclidean Distance Degree using topological tools and deduce a formula for the number of critical points of the distance function to $\disc$ from a general 
\BM{hypersurface}.
We also initiate the study of the analogous problem for higher osculating varieties to the Veronese variety. Furthermore, we establish a new extremal property of Chebyshev polynomials (conjectured in \cite{Raffalli2014}) by characterizing real-rooted binary forms that maximize the distance to $\disc$.
\end{changemargin}
\vspace{1cm}

\section{Introduction}

In many situations in scientific computing, analysing the closest singular configurations helps to develop robust computations with noisy data. The inverse of the distance to these singular configurations, is closely related to the conditioning of the problem, and often controls the numerical quality of the computed solutions of the problem \cite{Demmel1987}.

A typical example is a low-rank approximation of matrices, where a solution is obtained via truncation of the singular value decomposition \cite{EY1936}.
It turns out that singular values of a matrix correspond to critical values of the distance from it to the variety of singular matrices and pairs of singular vector give rise to the associated critical points.

Similarly, in the analysis of real algebraic varieties, the numerical quality of their approximation or some of their numerical  topological invariants are controlled by the inverse distance to the discriminant variety, which coincides, up to a scaling factor, with the so-called condition number.
Complexity bounds and the practical behaviour of many algorithms depend directly on this quantity (see e.g. \cite{Cucker2008,Brgisser2021}).

In Tensor Analysis, the distance from a symmetric tensor $F$ to the Veronese variety of rank-one tensors is given by the maximal absolute value on the unit sphere of the homogeneous form identified with $F$. More generally, critical points of the distance from $F$ to the Veronese variety (known as critical rank-one approximations of $F$) are symmetric rank-one tensors constructed from critical points of $F$ on the unit sphere (also known as eigenvectors of $F$ \cite{LekHengLim,QiLuo}) and their corresponding critical values.

Our aim is to study in detail the Euclidean distance function to discriminant varieties of forms with simple or higher-order singularities. From a dual point of view, this corresponds to considering the Euclidean distance to osculating varieties of the Veronese variety.
This problem fits into the following general framework \cite{DHOST} that lies at the intersection of real algebraic geometry and optimization. 
If $\cl X^{\R}$ is a (semi)algebraic variety of an Euclidean space $(\R^N,\scp{\cdot,\cdot})$ and $\bm u\in \R^N$ is a general data point, one seeks to optimize the (squared) distance $\textrm{dist}_{\bm u}: \cl X^{\R}\rightarrow \R$, $\bm x\mapsto \
\scp{\bm u - \bm x,\bm u - \bm x}$, from this point to the variety.
If $\textrm{dist}_{\bm u}$ is minimized at a smooth point of the Zariski closure $\cl X\subset \C^N$ of $\cl X^{\R}$, a minimizer is among the critical points of the restriction of $\textrm{dist}_{\bm u}$ to the smooth part $\smooth(\cl X)$ of the variety.
\BM{In general, }when the data point $\bm u$ is chosen generically, the distance function has only non-degenerate critical points on $\smooth(\cl X)$ and finitely many of them.
\BM{The number of \emph{complex} non-degenerate critical points for a generic data point $\bm u$ is known as} \emph{the Euclidean Distance Degree} $\EDD(\cl X)$ of the variety \cite{DHOST}.
However, when $\cl X$ has singularities it can happen that $\textrm{dist}_{\bm u}$ is minimized at a singular point of $\cl X$ for data points $\bm u$ \BM{in a ball of $\R^N$}, see Fig. \ref{fig:binary cubics}.
This observation motivates one to study the restriction of the distance function $\dist_{\bm u}$ to the singular locus $\sing(\cl X)$ of $\cl X$ or, more specifically, to particular irreducible components of $\sing(\cl X)$.
In particular, it is then natural to count critical points of $\dist_{\bm u}$ to the smooth part of $\sing(\cl X)$.
If, for a generic $\bm u$, the distance to $\cl X$ is minimized either at a smooth point of $\cl X$ or at a smooth point of $\sing(\cl X)$, the sum $\EDD(\cl X)+\EDD(\sing(\cl X))$ captures in a better way the algebraic complexity of the distance minimization problem to $\cl X$.

In the present work we investigate this phenomenon for the discriminant variety $\Disc$ that consists of \BM{degree-$d$ forms $\in \cl H_d$, which are singular}. We also consider subvarieties of $\Disc$ corresponding to singularities of prescribed types. The ambient space $\Hdn$ of degree-$d$ forms is endowed with the Bombieri-Weyl (also known, as apolar) product \eqref{eq:Bomb}. 

{\bf Previous works.} The case $d=2$ is well-understood.
For any quadratic form (or symmetric matrix) $F\in \H_2$ in $n$ variables, and the discriminant variety $\cl X =\disc_2$ of singular quadratic forms, the distance from $F$ to $\disc_2$ is 
$|\l_0|=\min_{i=0,\dots,n-1}\vert\lambda_i\vert$
where $\l_i$ are the eigenvalues of $F$ ordered by their absolute values. 
The Euclidean distance from $F$ to $\Xi_2$ is then minimized at the
 singular symmetric matrix $F -\l_0 X_0^2 $ where $X_0$ is the linear form given by a unit eigenvector of $F$ associated to $\l_0$.
 In particular, for any non-singular $F$, the distance to the discriminant attains its minimum at a smooth point of $\Xi_2$.
Moreover, for a generic $F\in \cl H_2$, any other critical point of the distance function from $F$ to $\Xi_2^{\smooth}$ is obtained by discarding from $F$ a rank-one quadratic form  $\l_i X_i^2$ associated to an eigenpair of $F$ \BM{(see  e.g. \cite[Thm. 2.2]{HS1995})}. Hence $\EDD(\Xi_2)=n$.
The singular locus of $\sing(\Xi_2)$ is the variety of symmetric matrices of rank $\le n-2$.
More generally, the distance function from a generic $F\in \cl H_2$ to the smooth locus of $\sing(\Xi_2)$ has $\EDD(\sing(\Xi_2))={n\choose 2}$ critical points obtained by discarding a pair of rank-one summands $\l_i X_i^2+ \l_j X_j^2$ from  $F$, (see e.g. \cite[Thm. 2.2]{HS1995}).

In the case $n=2$, the discriminant $\disc_{d}$ consists of binary forms having (at least) double roots. 
More generally, one can consider the variety of binary forms in $\H_d$ having (at least) a $k$-fold root ($k\leq d$).
Lee and Sturmfels \cite{LEE2016499} demonstrated that the Euclidean Distance Degree of each of these varieties equals $d$.
These are particular cases of varieties in $\Hdn$ that are defined by partitions  $\lambda \vdash d$ encoding multiplicities of roots.
The ED degree of these varieties is not known except for $\lambda=(k,1,\dots,1)\vdash d$, $2\leq k\leq d$, some partitions $\lambda \vdash d$ with $d\leq 7$ (see \cite[Table 1]{LEE2016499}) and $\lambda=(2,\dots,2)\vdash d$ for $d=2,4,6,8, 10$ \cite{SefatPanah2026}.

In the case of general $d$ and $n$, the Euclidean Distance Degree \BM{EDD$(\Disc)=\frac{(d-1)^n-1}{d-2}$} was computed in \cite{FS, CS2013}, \BM{in view of the invariance of the EDD under duality} (see \cite[Thm. 5.2]{DHOST}). 

The \BM{distance to the} real discriminant $\Disc^{\R}$ \BM{has also been studied in the past}. \BM{This semi-algebraic set} consists of  real forms with real singular points (cf. Section \ref{sub:disc}). 
The Euclidean distance function to $\Disc^{\R}$ is considered by Cucker et al. in \cite{Cucker2008}, where it is used to control the complexity of algorithms for counting zeros of polynomial systems.
It has been further studied by Raffalli in \cite{Raffalli2014}, where he explored how the topology of a real projective hypersurface affects whether its Bombieri-Weyl distance to $\Disc^\R$ attains minimum at a smooth point of the discriminant.

To \BM{analyse} the critical points of the distance to the (real) discriminant, \BM{we need to take into account its singular locus $\sing(\Xi_d)$.}
\BM{One of its components is known as}
the \emph{caustic component} (see e.g. \cite[p. 21]{Arnold1993}). It consists of forms having a degenerate singular point.
The other irreducible component of $\sing(\Xi_d)$,  known as \emph{the Maxwell stratum} \cite[p. 21]{Arnold1993}, consists of forms having at least two distinct singular points.
\BM{It is shown} in \cite[Thm. 4.5]{Raffalli2014}, that a minimizer of the distance to the real discriminant is never in the Maxwell stratum.
Therefore, from the perspective of distance minimization to $\Disc^{\R}$, only its smooth locus and the caustic component of $\sing(\Disc)$ matter. 

The \BM{dual to the caustic component} is the  tangential variety $\Tan$ of the Veronese variety $\Ver$, that is its first \emph{osculating variety}. 
The $k^{\mathrm{th}}$ osculating variety $\Okdn$ to $\Ver$ is defined as the locus of forms in $\Hdn$ that are divisible by $(d-k)^{\mathrm{th}}$ power of a linear form.
Higher-order osculating varieties are studied in the context of defectivity of secants by Piene \cite{Piene2022}. 
They are also relevant for the analysis of polar classes  of a projective variety $\cl X$ done by Bernardi et al. \cite{Bernardi_Catalisano_Gimigliano_Idà_2007} and their  degrees sum up to the ED degree of $\cl X$ for a sufficiently general scalar product \cite[Thm. 5.4]{DHOST}. 
In \cite{Lerario2021} Lerario and Stecconi considered the distance from $k$-jets of a vector-valued function $f$ to (local) discriminant varieties. They use it to define the condition number of a singularity and to bound invariants such as Betti numbers of the locus of $k$-jets of $f$ with prescribed values or degrees of its approximating polynomial equations.
As explained in \cite[Sec. 2.3]{Lerario2021}, this distance function can be seen as a generalization of the local version of the distance to the $k$-th osculating variety that we investigate in the present work.
Finally, higher-order osculating condition appears in the study of higher-order distance degree and higher-order distance loci recently introduced by Di Rocco et al. \cite{DiRoccoRoseSodomaco2026}.  

The osculating varieties $\Okdn$ play a central role in the generalized additive decomposition of (symmetric) tensors, providing a compact representation of a form $f\in \Hdn$ as a point on their joint variety (see e.g. \cite{Iarrobino1999,BarrilliMourrainTaufer2025}). 
\BM{They were also considered in \cite{Bernardi2007}, where the authors studied dimensions and defectivity of the secant varieties of $\mathcal{O}^k_{d}$.}

The dual of the osculating variety $\Okdn$ is naturally connected to the higher-order discriminant $\Disc^k$.
It is defined as the variety of forms in $\Hdn$ that vanish at some point in the projective space together with all the derivatives of order up to $k$.
Higher discriminants and their multiplicities are studied, e.g. in \cite{AluffiCukierman1993}, generalizing results from \cite{Parusiski1991}.


{\bf Contributions.}
We study the Bombieri-Weyl distance from a homogeneous polynomial $F\in \H_d$ of degree $d\geq 3$ in $n$ variables to the discriminant hypersurface $\Disc$ from algebro-geometric point of view.
It is well known that $\Disc$ is projectively dual to the Veronese variety $\Ver$ consisting of $d$-th powers of linear forms.  
\BM{We prove that} the singular locus of $\Disc$ splits into two irreducible components, $\sing(\Disc)=\cl C_{d}\cup \cl M_d$, both of codimension $2$ in $\Hdn$ (see Proposition \ref{prop:sing_discr}). 
The \emph{caustic component} $\cl C_d$ consists of forms in $\Hdn$ having (at least) $A_2$-singularities. \BM{We show that} its dual variety is the tangential variety $\Tan$ to the Veronese variety. 
The other irreducible component $\cl M_d$ of $\Disc$, known as \emph{the Maxwell stratum}, is obtained as the Zariski closure of forms in $\Hdn$ that are singular at two points in $\P^{n-1}_{\C}$. The projectively dual to $\cl M_d$ is the $2$-secant variety of $\Ver$.
Our main result is a closed formula for the Euclidean Distance Degree of the caustic component $\cl C_d$ of the discriminant.

\begin{theorem}\label{thm:main}
The Euclidean Distance Degree of the tangential variety $\Tan$ to the Veronese variety  or, equivalently, of its projectively dual variety $\cl C_{d}$ equals 
\begin{align}
\EDD(\Tan) = \EDD(\cl C_{d})  & = \frac {d} {2 (d-2)^2} \, \big( (2d-3)^{n} - 2 (d-1)^{n} + 1\big)
\end{align}
\end{theorem} 

As we pointed out above, the Bombieri-Weyl distance from a given real form $F\in \Hdn$ to the real discriminant $\Disc^\R$ was studied by Raffalli \cite{Raffalli2014}.
It is given via the following optimization problem
\begin{align}\label{eq:dist_opt_intro}
\min_{G\in \Disc^\R} \Vert F-G\Vert^2\ =\ \min_{X\in \Sn^\R} \delta_F^1(X),\quad \delta_F^1(X)\ =\ (1-d)F(X^*)^2+\frac{1}{d}\Vert \nabla_X F\Vert^2,
\end{align}
where $\nabla_X F$ is the differential of $F$ computed at the point $X^*$ uniquely associated with $X\in \Hun$ and $F(X^*)$ is the value of $F$ at this point.
In the sequel we denote by $\Sn$ the affine hypersurface in $\Hun$ defined by $\Vert X\Vert^2=1$, its real locus is the unit $(n-1)$-sphere $\Sn^\R$ in the space of real $n$-variate linear forms.
The function $\delta^1_F$, whose minimum over $\Sn^\R$ gives the Bombieri-Weyl distance from $F$ to $\Disc^\R$, depends polynomially both on $X$ and $F\in \Hdn$.
In Proposition \ref{prop:distdisc} we generalize \eqref{eq:dist_opt_intro} and give a formula for the distance to higher-order (real) discriminants defined in \eqref{eq:k_discr}.

We show that for a generic $F\in \Hdn$, critical points of the squared distance function from $F$ on $\Disc^\R$ are in correspondence with pairs of antipodal critical points of $\delta_F^1$ (see \eqref{eq:delta1 value}) on $\Sn^\R$ and that these critical points are non-degenerate (see Proposition \ref{prop:unique_min}). Moreover, by Proposition \ref{prop:critical point}, for a generic $F\in \Hdn$, a \BM{pair of antipodal critical points $\set{X,-X}\subset \Sn^\R$} of \BM{$\delta_F^1$} corresponds to 
\begin{itemize} 
\item either a smooth point of $\Xi_d$ 
\item or a smooth point of $\cl C_d$.
\end{itemize}
To prove this, we use a novel parametrisation of the singular forms, based on homogeneous Taylor expansion,  and the properties of the Bombieri-Weyl product.
In view of this characterization of critical points of the distance to $\Disc^\R$, our Theorem \ref{thm:main} together with the formula \cite{FS, CS2013} for the ED degree of $\Disc$ provide a measure of the algebraic complexity of the Bombieri-Weyl distance minimization problem to the real discriminant.

\begin{theorem}\label{thm:bound}
For a generic $F\in \Hdn$ the number of complex critical points of $\delta_F^1$ on $\Sn$ is 
\begin{align}\label{eq:sum}
2(\EDD(\Disc) + \EDD(\cl C_{d}))\ =\ \frac{d(2d-3)^n - 4(d-1)^n - d + 4}{(d-2)^2}
\end{align}
In particular, this formula gives an upper bound on the number of real critical points of the optimization problem \eqref{eq:dist_opt_intro} for a generic real form $F\in \Hdn$.
\end{theorem}

To prove the main result, we pass to the dual side and compute, instead, the EDD of the tangential variety $\Tan$ of $\Ver$. For the computation of the latter, we follow an approach
based on complex Morse theory on singular varieties, instead of polar classes techniques that require transversality conditions, which are not fulfilled in our case. We exploit a formula due to Maxim et al. \cite{MaximRodriguezWang2021} that expresses the EDD of a projective variety $\cl X$ in terms of the Euler characteristic of the local Euler obstruction of $\cl X$ restricted to some Zariski open subset $\cl U\subset \cl X$ of the variety.
This computation being quite tedious is nevertheless feasible due to the fact that a natural Whitney stratification of $\Tan$ by $\Ver=\Tan^\sing$ and $\Tan\setminus\Ver$ consists of only two strata. Furthermore, the parametrization \eqref{eq:Psi} of $\Tan$ by the product of two projective spaces facilitates the computations.

In the case of a smooth projective variety $\cl X$, the topological formula \cite{MaximRodriguezWang2021} for the ED degree discovered earlier by Aluffi and Harris \cite{AluffiHarris2018} reduces to the computation of the Euler characteristic of $\cl U\subset \cl X$.  We apply this formula to recover $\EDD(\Disc)=\EDD(\Ver)=\frac{(d-1)^n-1}{d-2}$ obtained in \cite{FS, CS2013} (see Proposition \ref{prop:EDD_Sigma}). 

We partially extend this analysis to higher discriminant varieties. In Section \ref{sec:proj dual}, we describe the dual of higher-order osculating varieties. In Section \ref{sec:disc}, we give formulas generalizing \eqref{eq:dist_opt_intro} for the distance to the higher-order osculating and discriminant varieties, leaving many open questions on its critical points for future works.

In \cite{Raffalli2014} the very natural problem of maximizing the distance to $\Disc^\R$ from a real form with a fixed Bombieri-Weyl norm \BM{is studied}.
In particular, \BM{it is conjectured} that for $d-r\in 2\N$, $C_{r,d}=(X_1^2+X_2^2)^{\frac{d-r}{2}} \textrm{Re}(X_1+i X_2)^r$, $r>d/3$, maximizes the distance to $\Disc^\R$ among all binary real forms $F\in \Hdn$ with $\Vert F\Vert_B=\Vert C_{r,d}\Vert_B$ and having $2r$ roots on $\Sn^{\R}$.
In Theorem \ref{thm:extremal} we settle this conjecture in the case of real-rooted forms ($r=d$), see Section \ref{sec:extremal} for details.
In \cite{AKU} it is shown that $C_{d,d}$ maximizes the Bombieri-Weyl distance to the Veronese variety among all binary forms in $\Hdn$ having the same norm as $C_{d,d}$.
Because of its link with Chebyshev polynomial of the first kind ($C_{d,d}(t,\sqrt{1-t^2}) =\cos(d\arccos t)$, $t\in [0,1]$), $C_{d,d}$ is referred to in \cite{AKU} as \emph{Chebyshev form}.
Our result can then be considered as a contribution to the list of extremal properties of Chebyshev polynomials.
\medskip

In Section \ref{sec::prel} we define all necessary notions and prove auxiliary results.
In particular, we describe the two irreducible components of $\Disc^\sing$ and their tangent spaces at smooth points. We also introduce parametrizations of some of the varieties relevant for this work.
In Section \ref{sec:disc} we discuss distance minimization problems to discriminant and osculating varieties, their geometry and critical points.
We provide there several examples explicitly showing the role played by the caustic component $\cl C_d$.
Section \ref{sec::crit_points} is dedicated to the study and counting of critical points of the Bombieri-Weyl distance to $\Disc$ and $\cl C_d$.
In particular, we compute the ED degree of these varieties using the topological approach developed in \cite{AluffiHarris2018, MaximRodriguezWang2021}.
In Section \ref{sec:extremal}, we prove that Chebyshev form of degree $d$ is a unique (up to rotations and scalings) form maximizing the distance to the real discriminant among real-rooted binary forms of degree $d$ with fixed Bombieri-Weyl norm. 
In Section \ref{sec::questions} we collect some remarks, questions and further directions pertinent to this work.
In the Appendices, we prove useful formulas related to the Bombieri-Weyl product and Taylor expansion of a homogeneous polynomial, describe the dimension and singular loci of osculating varieties $\Okdn$ and state Hirzebruch's formula for the Euler characteristic of a complete intersection in the projective space.


\section{Preliminaries}\label{sec::prel}

In this section we collect definitions, basic results and auxiliary facts that we need.\smallskip

Let $\kk$ be a field of characteristic $0$ (for example, $\kk=\R$ or $\kk=\C$). The set of non-zero vectors of a subset $S$ of a vector space is denoted $S^{\times}$. 
Let $\cl E$ be a $\kk$-vector space of dimension $n$, that we identify with $\kk^{n}$
by choosing a (canonical) basis $\bm e_1,\ldots, \bm e_n \in \cl E$. 
We denote by $\Hun=\cl E ^*$ the dual space of $\cl E$, that is the vector space of linear forms on $\cl E$. The basis \AP{of $\Hun$}, dual to the canonical basis \AP{of $\cl E$}, is denoted by $\bm e^{1}, \ldots, \bm e^{n}$. For $X\in \Hun$, let $X^*= \sum_{i=1}^{n} X(\bm e_i) \, \bm e_i\in \cl E$. 
For $\x\in \cl E$, let $\x^* = \sum_{i=1}^n \bm e^{i}(\x)\, \bm e^i \in \Hun$.
We endow $\cl E$ (respectively, $\Hun$) with the (canonical) symmetric bilinear form with respect to which the basis vectors $\bm e_1,\dots, \bm e_n\in \cl E$ (respectively, $\bm e^1,\dots, \bm e^n\in \Hun$) are orthonormal.
In the case of $\kk=\R$, it defines a scalar product on $\cl E$ and $\H_{1}$. We use \AP{$\Sn^\R$} to denote the unit sphere in these spaces.
We will also write $\P^{n-1}_{\kk}$ for the projectivization of $\cl E$ or $\H_{1}$ in the case of any field $\kk$. 

We denote by $\H_{d}^{\kk}= \H_{d}= \Sym^{d}(\cl E^*)$, the $d$-th symmetric power of $\Hun= \cl E^*$, that we identify with the space of $n$-variate forms (homogeneous polynomials) of degree $d$ with coefficients in $\kk$. We use $N={n+d-1\choose d-1}=\dim_{\kk} \H_{d}$ for its dimension.
Similarly, for a vector subspace $L\subset \Hun$, $\H_{d}(L)= \Sym^d(L)$ denotes the $d$-th symmetric power of $L$. Hereafter, forms or elements of $\cl H_m$ for some $m\in \N$ will usually be denoted with upper case symbols.
In what follows we will always assume that $d\geq 3$.

\subsection{Bombieri-Weyl (apolar) product and its properties}\label{sub:Bomb}
We fix a basis $X_1, \ldots, X_n$ of $\Hun$ that is orthonormal with respect to the canonical bilinear form (for example, $\bm e^1,\dots, \bm e^n\in \Hun$). Then $(\Vector{X}^\alpha)_{|\alpha|=d}$ is a (monomial) basis of $\Hdn$. We endow the space $\H_{d}$ with the \emph{Bombieri-Weyl product} (also known as the \emph{apolar product}), 
which is a symmetric bilinear form defined as
\begin{align}\label{eq:Bomb}
  \langle F,G\rangle_B\ =\ \sum_{\vert\bm \alpha\vert=d} {d \choose \bm \alpha}^{-1} F_{\bm \alpha} G_{\bm \alpha}\quad\textup{for}\quad F\, =\, \sum_{\vert\bm \alpha\vert=d} F_{\bm \alpha} \Vector{X}^{\bm \alpha},\ G\, =\, \sum_{\vert \bm \alpha\vert=d} G_{\bm \alpha} \Vector{X}^{\bm \alpha}.
\end{align}
On $\Hun$ it coincides with the canonical symmetric bilinear form that we fix in the beginning. 
One can equivalently define the Bombieri product as 
\begin{align}\label{eq:Bomb2}
    \langle F,G\rangle_B\ =\ \frac{1}{d!} F(\bm D)\, G,
\end{align}
where $F(\bm D)=F(D_1,\dots,D_n)$ is the differential operator obtained from $F$ by replacing variables $X_i$ with respective partial derivatives $D_i=\frac{\partial}{\partial X_i}$.
When $\kk=\R$, these formulas define a scalar product on $\H_{d}$ whose associated norm we denote by $\Vert F\Vert_B=\sqrt{\langle F, F\rangle_B}$. 

For $U\in \Hun$, $k\in \N$, let
\begin{align}\label{eq:nabla}
\nabla_{U}^{k}F = \sum_{|\bm \alpha|= k} {k \choose \bm \alpha} D^{\alpha_1}_1\cdots D^{\alpha_n}_n F(U^*)X_1^{\alpha_1} \cdots X_n^{\alpha_n}  \in \H_{k}
\end{align}
be the degree $k$ form associated to all $k^{\textup{th}}$-order derivatives of $F$ at $U$.
Slightly abusing the notations, we will use  $\nabla_{U}^k F$ to also denote the associated symmetric tensor $\left( D_{i_1}\cdots D_{i_k} F(U^*)\right)_{i_1,\dots, i_k=1}^n$.
Thus, we write $\nabla_{U} F= \sum_{i=1}^{n} D_i F(U^*) X_i\in \Hun$ and $\nabla_{U}^2 F= \sum_{i,j=1}^{n} D_i D_j F(U^*) X_i X_j\in \H_{2}$ for the gradient and the Hessian matrix of $F$ at $U$. \smallskip

For any $F\in \H_{d}$, $U\in \Hun$ with $\scp{U,U}=1$, we define 
the homogeneous Taylor expansion of $F$ at $U$ as
\begin{align}
F = U^d \,\nabla_{U}^{\perp,0}F + U^{d-1}\,  \nabla_{U}^{\perp,1}F  + \frac 1 {2!}  U^{d-2}\, \nabla_{U}^{\perp,2}F+ \cdots + \frac 1 {d!} \nabla_{U}^{\perp,d}F,
\end{align}
where $U^{\perp}= \{X \in \Hun \mid \scp{X, U}=0\}$,  $\nabla_{U}^{\perp,k}F \in \H_{k}(U^{\perp})$, $\nabla_{U}^{\perp, 0}F= F(U^*)\in \kk$ (see \eqref{eq:taylor}). For convenience, we denote $\nabla_{U}^{\perp}F := \nabla_{U}^{\perp, 1}F$.
Observe that $\nabla^{\perp,k}_U F\in \H_k(U^{\perp})$ is the restriction of the degree $k$ form $\nabla^k_U F$ to the hyperplane $U^\perp$.

For $0\le k\le d$, the sum of the first $k+1$ terms $T^k_{U}(F)$ and of the last $d-k$ terms $R^{k+1}_{U}(F)$ of this expansion satisfy
\begin{align}
F = T^{k}_{U}(F) +  R^{k+1}_{U}(F) 
\end{align}
with $\scp{T^{k}_{U}(F),R^{k+1}_{U}(F)}_B=0$
(see Section \ref{sec:taylor}).
Hereafter we will denote by $T^k: (F,X) \in \Hdn \times \Hun \mapsto T^k_{X}(F)\in \Hdn$ (resp. $R^k: (F,X) \in \Hdn \times \Hun \mapsto R^k_{X}(F)\in \Hdn$) the rational maps associated to this Taylor expansion.


For $P\in \H_{k}$ and $F\in \H_{d}$ with $d \ge k$, we define the co-product of $F$ and $P$ as 
\begin{align}\label{eq:co-product}
F \llcorner P = 
\frac{(d-k)!}{d!}P(\bm D)(F) 
 \in   \H_{d-k},
\end{align}
where $P(\bm D)=P(D_1,\dots, D_n)$.
We also denote $\ann_k(F) =\{ P \in \H_{k} \mid  F \llcorner P = 0\}$.

Note that if $F$ is singular at $X^*\in \cl E$, then, $X \in \ann_1(\nabla^2_{X} F)$ since, by Euler's identity \eqref{eq:euler}, we have ${\nabla^2_{X}F}\llcorner X = (d-1) \nabla_{X} F=0$.

By the properties of the product \eqref{eq:Bomb} (see Appendix \ref{sec:prop apolar}), $F\llcorner P$ can be identified with the image of $P$ by the \emph{Hankel operator} or \emph{the catalecticant} associated to $F\in \H_{d}$ defined as follows:
\begin{align}\label{eq:Hankel}
  H^{k,d-k}_{F}:  \H_{d-k}\ &\rightarrow\  (\H_{k})^*\\
  P\ &\mapsto\  \scp{F, P \, \cdot}_{B} = \scp{F\llcorner P, \cdot}_B. 
\end{align}

\BM{In particular, if $F$ is singular at $X^*\in \cl E$, we have $\nabla^2_{X} F\in \H_2(X^{\perp})$} and hence the image of $H^{1,1}_{\nabla^2_{X} F}$ lies in the space $X^\perp\subseteq \H_{1}$ that is canonically identified with the dual of $\faktor{\H_{1}}{\kk\, X}$.  
In this case, we can define the \emph{induced Hankel operator}
\begin{align}\label{eq:induced_Hankel1}
    \overline{H^{1,1}_{\nabla_{\x}^{2} F}}: \faktor{\H_{1}}{\kk \, X}\ &\rightarrow\ \left(\faktor{\H_{1}}{\kk\, X}\right)^*,\quad
    G_1+\kk\, X\ \mapsto\ H^{1,1}_{\nabla^{2}_{\x} F}(G_1),
\end{align}
see \ref{eq:induced_Hankel2} for the generalization of this for \eqref{eq:Hankel} with $k=1$.

We refer to Appendix for some further properties of \eqref{eq:Bomb} and its behaviour with respect to the Taylor expansion of a homogeneous polynomial.

\subsection{Discriminant varieties}\label{sub:disc}

Elements in $\H_{d}$ that have a singular point in $\P_{\kk}(\cl E)$ form the \emph{$\kk$-discriminant}
\begin{align}\label{eq:discr}
\disc_{d}^{\kk}\ =
{\bigcup\limits_{X \in \cl H_1^{\kk, \times}} \disc_{d, X}}
,\quad
\disc_{d, X}\ =\ \left\{F \in \H_{d}: D_i F(X^*)= 0,\ i=1,\ldots, n\right\}.
\end{align}
Note that the set $\disc_{d,X}$ of forms that are singular at a given point $X^*\in \P_\kk(\cl E)$ is a linear space of $\H_{d}$ of codimension $n$. 

In the sequel we abuse the notation and use $\disc_{d} := \disc_{d}^{\C} \subset \H_{d}$ to also denote the projectivization of \eqref{eq:discr} and refer to both as the discriminant.

It is known that $\disc_{d}^{\C}$ is an irreducible hypersurface of $\H_{d}$, defined by a single polynomial (with \BM{integer} coefficients) also called the discriminant \cite[Ch. 13]{GKZ1994}.

In the case $\kk=\R$, the set \eqref{eq:discr} is a semi-algebraic set of codimension $1$, called \emph{the real discriminant} by Raffalli in \cite{Raffalli2014}. 
Note that $\disc_{d}^\R$ is (in general) strictly contained in the real locus $\Disc^\C(\R)$ of the $\C$-discriminant, since the latter contains real forms having complex (conjugate) singular points and no real ones (See Fig. \ref{fig:swallowtail}). Moreover, the difference $\Disc^\C(\R)\setminus \disc_{d}^\R$ has codimension two (cf. Proposition \ref{prop:sing_discr}) and consists of exactly such real forms. 
We refer to Section \ref{sec:disc} for more details, where we also recall the distance minimization problem (with respect to the distance induced by \eqref{eq:Bomb}) to $\disc_{d}^{\R}$ studied in \cite{Raffalli2014}.\smallskip

For $k\geq 0$, define the $k^{\mathrm{th}}$ discriminant (see e.g. \cite{AluffiCukierman1993}) as
\begin{align}\label{eq:k_discr}
\disc^{k, \kk}_{d} = \bigcup\limits_{X \in \Hun^{\times,\kk}} \disc^{k}_{d, X},\quad\ 
\disc^{k}_{d, X} = \left\{F \in \H_{d}: \bm D^{\,\bm \alpha} F(X^*) = 0,\ \vert\bm \alpha\vert\le k\right\}.
\end{align}
In other words, $\Disc^{k} := \Disc^{k, \kk}$ consists of forms $F\in \H_{d}$ having (at least) a $k$-fold singular point in $\P^{n-1}_{\kk}$. 
By definition, $\disc_{d, X}^{0}= \{ G\in \H_{d}\mid \ G(X^*)=0\}$ and $\disc_{d}^0 = \Hdn$, when $\kk$ is algebraically closed.
Trivially, $\disc_{d}^1=\disc_{d}$ and inclusions $\disc_{d}^{k+1}\subseteq \disc_{d}^k $ hold for all $k\geq 0$.

We easily check that for any $G \in \disc_{d,X}^{k-1}$ we have $T^{k-1}_{X}(G)= 0$
and $G = R^{k}_{X}(G)$, and that for any $F\in \Hdn$, $R^{k}_{X}(F) \in \disc_{d,X}^{k-1} \subset \Disc^{k-1}$.  
This leads to the following parametrization of the elements $R^k_{X}(F)$:
\begin{align}\label{eq:varphik}
\varphi_k  : (\kk^{n-1} &\times \hat \H_{k} \times  \cdots \times \hat \H_{d}) \to \Disc^{k-1}\subset \Hdn, \\ 
&\varphi_k (\lb, H_k, \dots, H_d) = \sum_{i=k}^{d} X_{\lb} ^{d-i} H_i^{\lb},
\end{align}
where $\set{X_0, X_1, \ldots, X_{n-1}}$ is a basis of $\Hun$, $X_{\lb}= X_0+ \l_1 X_1 +\cdots + \l_{n-1} X_{n-1}$, $\hat \H_i = \H(\kk^{n-1})$,
$H_i^{\lb}= H_i(X_1-\l_1 X_0, \ldots, X_{n-1}-\l_{n-1} X_0) \in \H_{i}(X_{\lb}^{\perp})$.
By uniqueness of the Taylor expansion at $X_{\lb}$, we have $\frac 1 {i!} \nabla^{i,\perp}_{X_{\lb}}F=H_i^{\lb}$ and $R^k_{X_{\lb}}(F)= F$ for 
$F= \varphi_{k}(\lb, H_k, \ldots, H_d)$, 
This map parametrizes the Taylor expansions $R^k_{X}(F)$ \emph{nearby} $X_0$. It depends on the choice of the basis $\set{X_0, X_1, \ldots, X_{n-1}}$ of $\Hun$. 

When $\kk=\C$, the discriminant $\disc_{d}\subset \P^{N-1}_{\C}$ is a singular hypersurface. We now \BM{prove} a result describing its singular locus, that one can regard as a direct consequence of \cite[Prop. 5.5.2]{Teissier1977HuntingInvariants}\BM{, \cite{Dimca1986MilnorDual}, or \cite[Thm. 8.13]{Tevelev2003ProjectivelyDual}.}

\begin{proposition}\label{prop:sing_discr} Let $\kk= \C$.
The smooth locus of the discriminant hypersurface $\Disc$ is 
$$
\smooth(\Disc)=\{F \in \Hdn\,:\, \nabla_X F =0,\ \textrm{\normalfont{rank}}\, \nabla^2_{X}F = n-1\ \textrm{for a unique}\ X\in \P^{n-1}_{\C}\}.
$$

The singular locus of $\disc_{d}\subset \Hdn$ consists of two irreducible components, $\sing(\disc_{d})=\mathcal{C}_{d}\cup \cl M_{d}$, where 
    {\small\begin{align}
        \mathcal{C}_{d}\ &=\ \left\{F\in \Hdn\,:\, \nabla_{X}F=0,\ \textrm{\normalfont{rank}}\, \nabla_{X}^2 F \le n-2\ \textrm{for some}\ X\in \Hun^{\times}\right\}, \label{eq:cusp} \\
        \quad \cl M_{d}\ &=\ \zariski{\left\{ F\in \Hdn\,:\, \nabla_{X} F=\nabla_{Y} F=0\ \textrm{for linearly independent}\ X, Y\in \Hun^{\times}\right\}}        \label{eq:node}
    \end{align}}both of codimension two in $\Hdn$.
\end{proposition}

\begin{proof}
    Let $U\subset \Xi_{d}$ denote the set of forms $F$ that have a unique singular point $X^*\in \P{\cl{E}}$ so that $\nabla^2_{X} F$ has  rank $n-1$.
Then $U\subseteq \smooth(\Xi_{d})$ by \cite[Thm. 8.12]{Tevelev2003ProjectivelyDual}.
The complement of $U$ is precisely the union of $\mathcal{C}_{d}$ and $\cl M_{d}$ that are defined in \eqref{eq:cusp} and \eqref{eq:node} respectively. \BM{This shows that $\cl C_d \cup \cl M_d\supset \sing (\Xi_{d})$.}


\BM{We first prove} that $\mathcal{C}_{d}$ and $\cl M_{d}$ are \BM{irreducible} of codimension one in $\Xi_{d}$.
Indeed, $\Xi_{d}=\bigcup_{X\in \Hun^{\times}} \Xi_{d,X}$ is a hypersurface in $\H_{d}$ and we have $\mathcal{C}_{d}=\bigcup_{X\in \Hun^{\times}} \mathcal{C}_{d}\cap \Xi_{d,X}$, $\cl M_{d}=\bigcup_{X\in\Hun^{\times}} \cl M_{d}\cap \Xi_{d,X}$.
Inside \AP{every} linear space $\Xi_{d,X}$, the forms $F=R_{X}^2(F)$ satisfying $\textrm{rank} \nabla^2_{X} F\leq n-2$ constitute a hypersurface $\mathcal{C}_{d}\cap \Xi_{d,X}$ (given by the vanishing of the determinant of the induced Hankel operator $\overline{H_{\nabla^2_{X} F}^{1,1}}$ defined in \eqref{eq:induced_Hankel1}).
\BM{Note that $\cl C_d \cap \disc_{d,X}$ is the Zariski closure of the first projection of the incidence variety $\cl K_{X}  = \{(F,W)\in \disc_{d,X} \times \P\H_1(X^{\perp}) : \nabla_{X}^{2,\perp} F\llcorner W = 0\}$. The fiber of the second projection is linear of codimension $n-1$. Since the fiber of the first projection is generically of dimension $0$,  we deduce  by the fiber theorem (see e.g. \cite[Thm. 1.26]{Shafarevich2013}), that $\cl K_X$ and $\cl C_d \cap \disc_{d,X}$ are irreducible of dimension $\dim(\disc_{d,X})-1$, and that $\cl C_d$ is irreducible of codimension $2$ in $\Hdn$.}

Similarly, inside \AP{every} linear space $\Xi_{d,X}$, the forms having another singularity $Y^*\in \P{\cl{E}}$ constitute a hypersurface $\cl M_{d}\cap \Xi_{d,X}$ (obtained as the closure of the projection of the associated incidence variety $\BM{\cl N_{X}  =} \{(F,Y)\in \Xi_{d,X} \times \P \Hun \,:\, X,Y \textup{ linearly independent}, \nabla_{Y} F=0\}$). 
\BM{The fiber of the second projection of $\cl N_X$ on $Y\in \P\Hun$ is a linear subset of codimension $n$ in $\disc_{d,X}$, when $X,Y$ are linearly independent. We deduce as above that $\cl N_{X}$ is irreducible of dimension $\dim(\disc_{d,X})-1$. So is its first projection $\cl M_d\cap \disc_{d,X}$, since its generic fiber is of dimension $0$.
We deduce that $\cl M_d$ is also irreducible of codimension $1$ in $\Disc$, that is, of codimension $2$ in $\Hdn$.
}

\BM{The set $U^{\circ}$ of forms $F\in \cl C_{d}\cup \cl M_{d}$ with only finitely many singular points is dense in $C_{d}\cup \cl M_{d}$, since the first projection of $\cl N_X$ has generic fibers of dimension $0$.}
Then, by \cite{Dimca1986MilnorDual} or \cite[Ch. III, 4.5]{TeissierVanishing1973}, the multiplicity of the discriminant polynomial $\Delta_d$ at such a form $F$ equals $\sum_{X\in \P(\Hun)} \mu_{X} (F)$, the sum of Milnor numbers of singularities of $F$. 
Since $F$ either has (at least) a cuspidal singularity ($F\in \mathcal{C}_{d}$) or two distinct singular points ($F\in \cl M_{d}$), the above sum is at least two. Hence $F\in \sing (\Xi_{d})$.
\BM{We deduce that $\overline{U^{\circ}}= \cl C_{d}\cup \cl M_{d}= \sing (\Xi_{d})$}.
\end{proof}

\begin{example}\label{ex:smooth disc}
Let $X_0, \ldots, X_{n-1}$ be a basis of $\Hun$ and 
$$
G = X_0^{d-2} (\sum_{j=1}^{n-1}X_j^2) + \sum_{j=1}^{n-1} a_j\,X_j^d
$$
where $a_j\in \kk$ are generic (non-zero) coefficients. 

We verify directly that $\nabla_{X_0}G = 0$ so that  $G\in \Disc$ and that $\rank \napp G= n-1$.
Let us show that $G \not \in \cl M_{d}$, i.e. there is no $Y = \sum_{j=0}^{n-1} y_j X_j\in \Hun \setminus \vspan{X_0}$ such that $\nabla_{Y} G = 0$, that is,
\begin{align}
0 & =   (d-2) y_0^{d-3} (\sum_{i=1}^{n-1} 
y_i^2 ) = 0\\
0 & =  2 y_0^{d-2} y_j + d \,a_j\, y_j^{d-1} = 0 \textup{ for } j=1, \ldots, n-1.
\end{align}
Solving the last equations, we deduce that 
$y_j$ is either $0$ or a $(d-2)^{\mathrm{th}}$ root of $- \frac {2 y_0^{d-2}} {a_j d}$. 
Since $a_j$ are (non-zero) generic, the non-trivial solutions of this system are such that $y_1= \cdots = y_{n-1}= 0$ and $y_0\neq 0$.
That is $Y\in \vspan{X_0}$. 
By \AP{\cite[Thm. 8.12]{Tevelev2003ProjectivelyDual}}, see Proposition \ref{prop:sing_discr}, $G\in \smooth(\Disc)$.
\end{example}
\subsection{Veronese variety and its osculating varieties}

The \emph{Veronese variety} $\Ver\subset \H_{d}$ consists of $d^{\mathrm{th}}$ powers of linear forms,
\begin{align}\label{eq:Veronese}
\Ver\ =\bigcup_{X\in \Hun} \cl V_{d,X},\quad  \cl V_{d,X}= X^d\cdot \cl H_{0}\quad \textrm{for}\quad X\in \Hun,
\end{align}
where $\cl H_{0}= \kk$. We also consider the \emph{tangential variety} to $\Ver$ that is defined as 
\begin{align}\label{eq:tangential}
\mathcal{T}_{d}\ = \bigcup_{X\in \Hun} \cl T_{d,X},\quad \mathcal{T}_{d,X}\ =\ X^{d-1} \cdot \Hun \quad \textrm{for}\quad X\in \Hun.
\end{align}
More generally, for any nonnegative integer $k\leq d$, the \emph{$k$th osculating variety} to $\Ver$ is defined as 
\begin{align}\label{eq:osculating_variety}
\mathcal{O}^k_{d}\ =\ \bigcup_{X \in \Hun} \cl O^{k}_{d,X},
\quad
\mathcal{O}^k_{d,X}= X^{d-k}\cdot \H_{k} \quad \textrm{for}\quad X\in \Hun.
\end{align}

The osculating variety $\Okdn$ can be parametrized \emph{locally} as follows:
\begin{align}\label{eq:psik}
\psi_k  : \kk^{n-1} &\times  \H_{k}  \to \Okdn\subset \Hdn, \\ 
&\psi_k (\lb, W) = X_{\lb}^{d-k}\, W,
\end{align}
where $\set{X_0, X_1, \ldots, X_{n-1}}$ is a basis of $\Hun$ and $X_{\lb}= X_0+ \l_1 X_1 +\cdots + \l_{n-1} X_{n-1}$. 
Note that the image of $\psi_k$ is the  open dense subset $\tilde{\mathcal{O}}^k_{d} = \cup_{X \not \in \vspan{X_1, \ldots, X_{n-1}}}  \cl O^{k}_{d,X}$, since 
for any \BM{$F=X^{d-k} W\in \tilde{\mathcal{O}}^k_{d}$ with} $X \not \in \vspan{X_1, \ldots, X_{n-1}}$, there exist a unique $\l_0\neq 0$ and a unique $\lb \in \kk^{n-1}$ such that $X= \l_0 X_{\lb}$, so that $F= X_{\lb} ^{d-k} (\l_0^{d-k} W) = \psi_k(\lb, \l_0^{d-k} W) \in \Okdn$.
The map $\psi_k$ is also injective when $d> 2k$ or when $d\le 2 k$ for $W\not \in \cl O^{2k-d}_{k}$, since 
$X_{\lb_0}^{d-k}\, W_0= X_{\lb_1}^{d-k}\, W_1$ for distinct $X_{\lb_0}, X_{\lb_1} \in \Hun$ and $W_0, W_1 \in \H_{k}$ implies that $W_0 = X_{\lb_1}^{d-k} W_0' \in \cl O^{2k-d}_{k}$.


Since $\Ver$ and, more generally, $\cl O^k_{d}$ are affine cones in $\H_{d}$ we will abuse notations and denote their projectivizations in $\P^{N-1}_{\kk}$ by the same letters.


One has $\cl O^{0}_{d}= \cl V_{d}$,  $\cl O^{1}_{d}= \cl T_{d}$, $\cl O^{d}_{d} = \Hdn$, and $\cl O^{k}_{d} \subseteq \cl O^{k+1}_{d}$ for $0\le k<d$. 
The dimension and singularities of $\Okdn$ are analysed in Appendix \ref{sec:ok sing}.
In particular, we see that the variety $\cl O^{k}_{d}$ is singular along $\cl O^{k-1}_{d}$.
For $k<\frac d 2$, $\sing(\Okdn)= \Odn^{k-1}$.
%
%
We have in particular that for $d>2$, $\dim(\cl T_{d})= 2n -1$ and $\sing(\cl T_{d})= \cl V_{d}$. Also, 
$\dim(\Odn^2)= \frac 1 2 (n^2+3 n -2)$ and $\sing(\Odn^{2})= \cl T_{d}$.

\begin{lemma}\label{lem:tgtOk}
The tangent space of $\Okdn$ at a smooth point $F = X_0^{d-k} W_0 \in \smooth(\cl O_{d}^{k})$ is
$$
T_{F}(\Okdn) = X_0^{d-k-1}W_0 \, \H_{1} + X_0^{d-k} \, \H_{k}.
$$
\end{lemma}
\begin{proof}
\AP{By Appendix \ref{sec:ok sing} $F\notin \cl O_{d}^{k-1}$.  Therefore, we may assume that 
$F= X_0^{d-k} W_0 \in \Okdn$ with $\set{X_0, X_1, \ldots, X_{n-1}}$ a basis of $\Hun$, so that $F=\psi_k(\bm 0,W_0)$ and $W_0 \notin X_0\cl H_{k-1}$. Then the map $\psi_k$ is locally injective at $(0,W_0)$.  As a holomorphic injective map between two smooth complex manifolds of the same dimension it is a local diffeomorphism.  Therefore} 
the tangent space of $\Okdn$ at $F$, obtained by differentiating $\psi_k$ with respect to $\lb$ and $W$, is
\begin{align}
T_F(\Okdn)   = \mathrm{Im} (d\psi_k)_{|(\bm 0, W_0)} 
 & = X_0^{d-k-1} \vspan{X_j}_{j=1, \ldots, n-1} W_0 \oplus X_0^{d-k} \H_{k} \\ & = X_0^{d-k-1} W_0 \Hun + X_0^{d-k} \H_{k},
\end{align}
which proves the lemma.
\end{proof}
It is well-known than $\Ver$ is smooth because it is an orbit of the standard $GL(n)$-action on $\H_{d}$.  
By the above lemma, the tangent space of $\Ver$ at $F = X_0^d W_0\in \Ver$ with $X_0 \in \Hun$, $W_0 \in \kk$ and $F \neq 0$ is 
$T_{F}(\Ver) = \cl T_{d,X_0}=  X_0^{d-1}\, \Hun$.
\subsection{Projective duality}\label{sec:proj dual}

Recall that for a variety $\cl X\subset \P_{\kk}^{N-1}$ its \emph{(projectively) dual} is a variety $\cl X^\vee \subset (\P_{\kk}^{N-1})^{\vee}$ in the dual projective space defined as 
\begin{align}
\cl X^\vee\ =\ \overline{ \left\{H\in (\P_{\kk}^{N-1})^\vee\,:\, H|_{T_{\x} (\cl X)}=0\ \textrm{for some}\ \x\in \smooth(\cl X)\right\}}^{\textrm{Zar}}.
\end{align}
Namely, $\cl X^{\vee}$ is the Zariski closure of the set of hyperplanes containing tangent spaces at smooth points of $\cl X$.
When $\kk^N$ is endowed with a non-degenerate bilinear form (like \eqref{eq:Bomb} for $\kk^N=\H_{d}$), one can identify $\cl X^{\vee}$ with a subvariety of $\P_{\kk}^{N-1}$, so both $\cl X$ and its dual lie in the same projective space.

\begin{proposition}\label{prop:dual}
The projectively dual variety to $\Okdn$ is given as the Zariski closure of 
{\small\begin{align}
\cl C^{k}_{d} = \left\{ F \in \H_{d} \,\big|\, \nabla_{X_0}^{k} F =0, {\nabla_{X_0}^{k+1} F} \llcorner W_0 =0 \textrm{ for some } X_0 \in \Hun^{\times}, W_0\in \H_{k}\setminus X_0\cdot \H_{k-1}\right\},
\end{align}}where $\nabla_{X_0}^k F$ is defined in \eqref{eq:nabla}. It consists of forms $F\in \disc_{d,X_0}^k$, that vanish to order (at least) $k$ at some $X_0\in \Hun$ and the catalecticant of $\nabla_{X_0}^{k+1}F$ has a ``non-trivial kernel'' in degree $k$.
\end{proposition}
\begin{proof}
Let $F= X_0^{d-k} W_0 \in \smooth(\Okdn)$ and let $G \in T_F(\Okdn)^\perp \subset (\Okdn)^{\vee}$. By Lemma \ref{lem:tgtOk},
$T_F(\Okdn) = X_0^{d-k-1}W_0 \, \H_{1} + X_0^{d-k} \, \H_{k}.$
Thus $G$ satisfies the conditions
\begin{align}
    \scp{G,X_0^{d-k-1} W_0 X_j}_{B}\ =\ \frac{(d-k-1)!}{d!} \scp{\nabla_{X_0}^{k+1}(G), W_0 X_j}_{B} = \scp{\nabla_{X}^{k+1}(F) \llcorner W_0, X_j}_{B} =\ 0,
\end{align}
for $j=0, \ldots, n-1$, i.e. $\nabla_{X_0}^{k+1}(F)\llcorner W_0=0$, and 
\begin{align}
    \scp{G,X_0^{d-k} \dot{W}}_{B}= \frac{(d-k)!}{d!} \scp{\nabla_{X}^kF, \dot{W}}_{B} =\ 0.
\end{align}
for all $\dot{W}\in \H_{k}$, i.e. $\nabla^k_{X_0}F =0$. 

Vice versa, if $G\in \H_{d}$ is a form for which there exist $X_0\in \Hun$ and a degree $k$ form $W_0\in \H_{k}$ not divisible by $X_0$ so that $\nabla_{X_0}^k G=0$ and $\nabla_{X_0}^{k+1}G \llcorner W_0=0$,
then, reversing the above arguments, we see that $G$ lies in the space orthogonal to the tangent space of $\Okdn$ at a smooth point $F = X_0^{d-k} W_0\in \Okdn$.
%
\end{proof}

As a particular instance of this result, we recover the following classical facts.
\begin{proposition}\label{prop:dual_ver}
The projectively dual variety to $\Ver$ is the discriminant variety $\disc_{d}$.
The tangent space to $\disc_{d}$ at $F \in \smooth(\disc_{d})$ \BM{singular at $X_0^*\in \P{\cl{E}}$} is 
\begin{align}
T_F(\disc_{d}) =\disc_{d,X_0^*}^0 = \{ H\in \H_{d}\mid \ H(X_0^*)=0\}.
\end{align}
\end{proposition}
 
\begin{proof}
The first part of the proposition is a corollary of Proposition \ref{prop:dual}, 
since $\Ver = \Odn^0$ and $\nabla_{X_0}^0 F= F(X_0^*)$, $\nabla_{X_0}^1 F \llcorner W_0 = 0$ for $W_0 \in \H_{0}^{\times}=\kk^\times$ implies $\nabla_{X_0}^1 F=0$.
This shows that 
$\cl C_{d}^0= \Disc=\Ver^\vee$.
 
Let $F \in \smooth(\Disc)$. By Proposition \ref{prop:sing_discr}, $F$ has a unique singular point $X_0^{*}\in \P{\cl{E}}$ (i.e. $F\not \in \cl M_{d}$), such that $\rank \nabla_{X_0}^{2,\perp}F=n-1$. 

The elements in $\Disc$ around $F$ are parametrized by $\varphi_2$ (defined in \eqref{eq:varphik}): 
$$
(\lb, H_2, \ldots, H_d) \in \kk^{n-1} \times \hat{\H}_2 \times \cdots \times \hat{\H}_d \mapsto \varphi_2(\lb, H_2, \ldots, H_d) \in \Disc.
$$
The image of the differential map at $(0,  H_{0,2}, \ldots, H_{0,d})$ with \KK{$H_{0,i} = \frac{1}{i!}\nabla_{X_0}^{i, \perp}F \in \hat{\H}_{i}$} is
\begin{align}\label{eq:imdphi2}
\lefteqn{\mathrm{Im}\, (d\varphi_2)_{|(0,H_{0,2}, \ldots, H_{0,d})}}\\ &=  
\vspan{ \sum_{i=2}^{d-1}  X_0^{d-i-1} 
(d-i) X_j H_{0,i} - X_0^{d-i+1} \partial_j H_{0,i}}_{j=1, \ldots, n-1} +\sum_{i=2}^{d} X_0^{d-i} \hat{\H}_i \\&
= X_0^{d-1} \vspan{\partial_j H_{0,2}}_{j=1, \ldots, n-1} + \sum_{i=2}^{d} X_0^{d-i} \hat{\H}_i.
\end{align}

If $\dim \vspan{\partial_1 H_{0,2}, \dots, \partial_{n-1} H_{0,2}} = \rank H_{0,2} = \rank \nabla_{X_0}^{2,\perp}F =n-1$, then $\vspan{\partial_1 H_{0,2}, \dots, \partial_{n-1} H_{0,2}} = \hat \H_{1}$ and 
$\mathrm{Im}\, (d\varphi_2)_{|(0,H_2, \ldots, H_{d})}$ is a hyperplane of $\Hdn$.
Since $F\not\in \cl M_{d}$ is not a double point of the parametrization $\varphi_2$, $\Disc$ is smooth at $F$
and
$$
T_{F}(\Disc)= \sum_{i=1}^{d} X_0^{d-i} \hat{\H}_i = \{ G\in \H_{d}\mid \ G(X_0^*)=0\} = \vspan{ X_0^{d}}^{\perp}.
$$
\end{proof}
Notice also, that the biduality theorem \cite[Thm. 1.1]{GKZ1994}, can easily be verified here since $T_{F}(\disc_{d})^{\perp} = (\vspan{ X_0^{d}}^{\perp})^{\perp} = \vspan{X_0^{d}}$, so that $\Disc^{\vee} = \Ver$.

Hereafter, we prove similarly that the dual to the tangential variety $\mathcal{T}_{d}=\Odn^1$ to $\Ver$ is the variety $\mathcal{C}_{d}$ that we encountered in Proposition \ref{prop:sing_discr}:

\begin{proposition}\label{prop:dual_ver_tan}
The projectively dual variety to $\mathcal{T}_{d}$ is the caustic component $\mathcal{C}_{d}$ of the singular locus of the discriminant (cf. Prop. \ref{prop:sing_discr}). 
If $F$ has a unique singularity $X_0^*\in \P{\cl{E}}$, which is of $A_2$-type (i.e. $\rank \nabla_{X_0}^{\perp,2}F = n-2$ and $\nabla_{X_0}^{\perp,3}F \llcorner W_0^3 \neq 0$ for some $W_0 \in \ann_1(\nabla_{X_0}^{\perp,2}F) \cap \H_{1}(X_0^{\perp})$), then $F\in \smooth(\cl C_{d})$.
The tangent space to $\cl C_{d}$ at $F$ is 
{\small\begin{align}
    T_F( \cl C_{d}) = & \left\{G\in \H_{d} \mid G(X_0^*)=0,\, \nabla_{X_0} (G) \llcorner W_0 =0 \ \textup{ for some } W_0\in \ann_1( \nabla^{2}_{X_0} F) \setminus{\kk\cdot \BM{X_0}} \right\}.
\end{align}}
\end{proposition}

\begin{proof}
The first part of the proposition is proved in Proposition \ref{prop:dual}, since $\Tan= \Odn^1$.

Let $F\in \cl C_{d}$ with a unique singularity $X_0^*\in \P{\cl{E}}$ (i.e. $F\not \in \cl M_{d})$ and such that $\rank \nabla_{X_0}^{\perp,2}F = n-2$,  
$\nabla_{X_0}^{\perp,2}F \llcorner X_1=0$ and $\nabla_{X_0}^{\perp,3}F \llcorner X_1^3 \neq 0$, where $X_0, X_1, \ldots, X_{n-1}$ is a basis of $\Hun$.
Then it is of the form $F = \sum_{i=2}^{d} X_0^{d-i} H_{0,i}$ where $H_{0,i} = \nabla^{\perp,i}_{X_0}F \in \hat \H_i$ and $H_{0,2} \in \H_2 (X_2, \ldots, X_{n-1})$ since $\nabla_{X_0}^{\perp,2}F \llcorner X_1=0$.

We consider the map $\tilde \varphi_2$ induced by $\varphi_2$ (defined in \eqref{eq:varphik}): 
$$
\tilde \varphi_2: (\lb, H_2, \ldots, H_d) \in \kk^{n-1} \times {\tilde{\H}_2} \times \hat \H_3 \cdots \times \hat{\H}_d \mapsto \varphi_2(\lb, H_2, H_3 \ldots, H_d) \in \cl C_{d},
$$
where $\tilde{{\H}_2} =\set{H_2 \in \hat \H_2\mid \rank H_2 \le n-2}$.
By Proposition \ref{prop:sing_discr}, $\tilde \varphi_2$ parametrizes $\cl C_{d}$.

By differentiating the equations $H_2\llcorner W =0$ defining the incidence variety in $\hat \H_2 \times \hat \H_{1}$ above $\tilde \H_2$, we verify that the tangent space of $\tilde \H_2$ at $H_{0,2} \in \H_2(X_2, \ldots, X_{n-1})$ is  
$$
T_2 := T_{H_{0,2}}(\tilde \H_2) = \set{ H_2 \in \hat \H_2 \mid H_2 \llcorner X_1 \in \vspan{X_2, \ldots, X_{n-1}}} = 
\set{ H_2 \in \hat \H_2 \mid H_2 \llcorner X_1^2 = 0}.
$$
The image of $d\tilde \varphi_2$ at $F$, computed as in \eqref{eq:imdphi2}, is
\begin{align}
\lefteqn{\mathrm{Im}\, (d\tilde\varphi_2)_{|(0,H_{0,2}, \ldots, H_{0,d})}} \\
& =  
\vspan{\sum_{i=2}^{d-1}  (d-i) X_0^{d-i-1} 
 X_j H_{0,i} - X_0^{d-i+1} \partial_j H_{0,i}}_{j=1, \ldots, n-1} + X_0^{d-2} T_2 
 + \sum_{i=3}^{d} X_0^{d-i} \hat{\H}_i
\\ & = X_0^{d-2} 
\big( \vspan{ X_0\, \partial_j H_{0,2} 
+  \partial_j H_{0,3}}_{j=1, \ldots, n-1}
+ T_2\big)
+ \sum_{i=3}^{d} X_0^{d-i} \hat{\H}_i
\\ &
= X_0^{d-2} 
\big( \vspan{ X_0\, \partial_j H_{0,2} 
+ \partial_j H_{0,3}}_{j=2, \ldots, n-1}
+ \vspan{\partial_1 H_{0,3}}
+ T_2\big)
+ \sum_{i=3}^{d} X_0^{d-i} \hat{\H}_i
\end{align}

As $H_{0,3} \llcorner X_1^3 \neq 0$, $\vspan{\partial_1 H_{0,3}}+ T_2 = \hat \H_2$. We deduce that 
\begin{align}
{\mathrm{Im}\, (d\tilde\varphi_2)_{|(0,H_{0,2}, \ldots, H_{0,d})}} 
= X_0^{d-1} \vspan{ \partial_j H_{0,2}}_{j=2, \ldots, n-1}
+ X_0^{d-2} \hat \H_2
+ \sum_{i=3}^{d} X_0^{d-i} \hat{\H}_i,
\end{align}
which is of dimension $\dim \Hdn - 2$, since $\rank H_{0,2} = \rank \nabla^{\perp,2}_{X_0} F = n-2$. Therefore, ${\mathrm{Im}\, (d\tilde\varphi_2)_{|(0,H_{0,2}, \ldots, H_{0,d})}}$ is the tangent space $T_{F}(\cl C_{d})$ of $\cl C_{d}$ at $F$. 
We deduce that $G\in T_{F}(\cl C_{d})$, if and only if, $\nabla^{\perp,0}_{X_0}G= G(X_0^*) = 0$ and $\nabla^{\perp,1}_{X_0}G \in \vspan{X_2, \ldots, X_{n-1}}$, i.e. $\nabla^{\perp,1}_{X_0}G \llcorner X_1 = 0$, which proves the proposition.
\end{proof}

\begin{remark}
Similarly, one can show that the dual of the Maxwell component $\cl M_d$ of $\sing(\Disc)$ (see Proposition \ref{prop:sing_discr}) is the $2$-secant $\sigma_2(\Ver)$ of $\Ver$.  
\end{remark}

\begin{example}\label{ex:formeA_2}
    For $d\geq 3$ let us define $F=X_0^{d-2}(X_2^2+\dots+X_{n-1}^2)+X_0^{d-3} X_1^3+\sum_{i=1}^{n-1} a_i X_i^d$. For generic real parameters $a_1,\dots, a_{n-1}\in \mathbb{R}$, the form $F$ has a unique singular point $X^*=[1:0:\dots:0]$ which is of type $A_2$ and so $F$ belongs to the smooth locus of $\mathcal{C}_{d}$.
To see this, observe first that $F$ does not have singular points in the hyperplane $X_0=0$. In the affine chart $X_0=1$, singular points \BM{$Y^* = [1:y_1:\cdots: y_{n-1}]$} of $F$ satisfy
\begin{align}\label{eq:sing_equations}
&(d-2)(y_2^2+\dots+y_{n-1}^2) + (d-3) y_1^3=0,\quad y_1^2(3+ d a_1 y_1^{d-3})=0,\\ &y_i(2+d a_i y_i^{d-2})=0,\quad i=2,\dots, n-1. 
\end{align}
For generic choice of $a_1,\dots, a_{n-1}\in \R$, this system of equations has only the origin as a solution.
For example, we can take $a_1=\frac{3}{d}(d-3)^{\frac{d-3}{3}}$ and $a_i=\frac{2}{d}\left(\frac{d-2}{3^{\,i-1}}\right)^{\frac{d-2}{2}}$,
$i=2,\dots,n-1$.
We have $(d-3)\vert y_ 1^3\vert=1$ if $y_ 1\neq 0$ and $(d-2)\vert y_ i^{2}\vert =3^{i-1}$ whenever $y_ i\neq 0$, $i=2,\dots, n-1$.
The \BM{left-hand side of the first equation} in \eqref{eq:sing_equations} is then the sum of complex numbers whose absolute values are powers of $3$ among $3^{i-1}$, $i=1,2,\dots,n-1$.
Let $3^k$, $1\leq k\leq n-2$, be the largest absolute value appearing.
Then 
\begin{align}
    \vert(d-2)(y_ 2^2+\dots+y_ {n-1}^2) + (d-3) y_ 1^3\vert\ \geq\ 3^k-\sum_{i=1}^k 3^{i-1}\ =\ 3^k-\frac{3^k-1}{2}\ >\ 0
\end{align}
and so $y_ 1,y_ 2,\dots, y_ {n-1}$ must all be zero.
\end{example}

The condition for $\x\in \P^{n-1}_{\kk}$ to be a cuspidal singularity means that the Hessian $\nabla_{\x}^2 F$ of $F$ at $\x$ does not have maximal possible rank $n-1$. This condition is equivalent to $H^{1,1}_{\nabla_{\x}^2 F}(G_1)=0$ for some $G_1\in \H_{1}$ non-proportional to $(\x, \cdot)$ and also to $\det \overline{H^{1,1}_{\nabla_{\x}^2 F}} =0$, where $\overline{H^{1,1}_{\nabla_{\x}^2 F}}$ is defined in \eqref{eq:induced_Hankel1}.
For higher $k\geq 1$ this generalizes as follows.

\begin{remark}
    Let us carefully analyze the condition $G\in \ker H^{1,k}_{\nabla^{k+1}_{X}F}$ with $G\in \H_{k}$ being not divisible by $X \in \Hun$.
First, because of $\nabla_{X}^k F=0$, any form in $X\, \H_{k-1}$ lies in the kernel of $H^{1,k}_{\nabla_{X}^{k+1} F}$ by \eqref{eq:eval} and \eqref{eq:euler}. Also, \eqref{eq:euler} implies that the image of $H^{1,k}_{\nabla^{k+1}_{X} F}$ lies in the space $X^{\perp} \subseteq (\H_{1})^*$ that is canonically identified with the dual space to $\faktor{\H_{1}}{\kk(X  ,\cdot)}$. Therefore, \emph{the induced Hankel operator} 
\begin{align}\label{eq:induced_Hankel2}
    \overline{H^{1,k}_{\nabla_{X}^{k+1} F}}: \faktor{\H_{k}}{\BM{X} \,\H_{k-1}}\ &\rightarrow\ \left(\faktor{\H_{1}}{\kk\, X}\right)^*,\\
    \overline{G}\,=\,G+X\,\H_{k-1}\ &\mapsto\ H^{1,k}_{\nabla^{k+1}_{X} F}(G)
\end{align}
is well-defined.
The above condition then becomes $\ker \overline{H^{1,k}_{\nabla^{k+1}_{X}F}}\neq \{\boldsymbol{0}\}$.
\end{remark}

\begin{remark}
[Parametrizations and the \BM{orthogonal group} $O(n)$ action]
 The action of $O(n)$ on $\kk^n$ induces an action on $\H_{d}$ by 
$\gb \cdot F (X )=  F ( \gb \cdot X )$,  that preserves the Bombieri-Weyl form. The discriminant and the osculating varieties are preserved by this action but the action is not transitive.  The trace of this action appears in the formulas  \eqref{eq:varphik}, \eqref{eq:psik} and in the formula \eqref{eq:param-oscul2} and play an important role in the computation of tangent spaces, as in  Lemma \ref{lem:tgtOk} and Proposition \ref{prop:dual_ver_tan}. 

For instance, consider the osculating variety parametrized locally  by 
 \eqref{eq:param-oscul2} and let 
$F= X_0^{d-k} G \in \cl{O}^k_{d}$, where 
$H= \sum_{i=0}^k X_0^{k-i} G_i$ with $G_i \in \hat  \H_i$.  The derivatives of $\Psi_k$ with respect to the 
variables $H_i$ show that 
\begin{align}
T_{F} \cl {O}^k_{d} \supset X_0^{d-k} \, \H_{k}.  
\end{align}
If we denote by $Orb(F)$  the $O(n)$ orbit of $F$, then for $F= X_0^{d-k} (\sum_{i=0}^k X_0^{k-i} G_i)$ 
\begin{align}\label{eq:tangent_orbit1}
T_{F} Orb(F) =  X_0^{d-k-1}G_k \, \hat \H_{1} \oplus  (T_{F} Orb(F) \cap X_0^{d-k} \, \H_{k}) .
\end{align}
 Hence 
\begin{align}\label{eq:tangentO_d}
T_{F} \cl {O}^k_{d} =T_{F} Orb(F) + X_0^{d-k} \, \H_{k}  =  X_0^{d-k-1} G_k \, \hat \H_{1} \oplus  X_0^{d-k} \, \H_{k} .
\end{align}

 Similarly, if $F= \sum_{i=k+1}^{d} X_0 ^{d-i} G_i \in \disc^{k}_{d} $, then 
\begin{align}\label{eq:tangent_orbit2}
T_{F} Orb(F)   = X_0^{d-k} \vspan{\partial_1 G_{k+1}, \dots, \partial_{n-1} G_{k+1}} 
\oplus (T_{F} Orb(F) \cap \sum_{i=k+1}^{d} X_0 ^{d-i}  \, \hat \H_{i} ). 
\end{align}
  Therefore 
\begin{align}
T_{F} \disc^{k}_{d}  =  T_{F} Orb(F) +\sum_{i=k+1}^{d} X_0 ^{d-i}  \, \hat \H_{i} 
 = X_0^{d-k} \vspan{\partial_1 G_{k+1}, \dots, \partial_{n-1} G_{k+1}}  \oplus  
\sum_{i=k+1}^{d} X_0 ^{d-i}  \, \hat \H_{i}   .
\end{align}  
\end{remark}

\section{Distance to discriminants and osculating varieties}\label{sec:disc}

To analyze the distance to discriminant and  osculating varieties, we define the function
\begin{align}\label{eq:delta}
\delta_F^{k}(X)\ :=\ \sum_{i=0}^{k}  { {d \choose i}^{-1} \big(i!\big)^{-2} } 
 \norm{\nabla_{X}^{\perp,i}F}^2_B,\quad X\in \Sn^\R.
\end{align}
In particular, we have $\delta_F^0(X) = F(X^*)^2$ and 
\begin{equation}\label{eq:delta1 value} 
\delta_F^1(X) 
= F(X^*)^2 + \frac{1}{d} \norm{\nabla^{\perp}_{X}F}^2 
= (1-d) F(X^*)^2 + \frac{1}{d} \norm{\nabla_{X} F}^2,
\end{equation}
(see Lemma \ref{lem:norm_Delta_K} and Lemma  \ref{lem:T1}), which coincides with the distance function studied by Raffalli \cite{Raffalli2014}.

\begin{lemma}\label{lem:delta1 crit}
For any $F\in \Hdn$, $X\in \Hun$ is a critical point of $\delta_F^1$ on $\Sn$, if and only if,
$$
\nabla^{\perp,2}_{X}F \llcorner \nabla^{\perp}_{X} F = 0.
$$
\end{lemma}
\begin{proof}
Let $X$ be a critical point of $\delta^1_F$ on $\Sn$ with $X, X_1, \ldots, X_{n-1}$ an orthogonal basis of $\Hun$. Then by  \eqref{eq:delta1 value}, we have 
\begin{equation}\label{eq:discd1}
\nabla_X(\delta_F^1) = (1-d) F(X^*)\nabla_{X}F + \frac{1} {d} \, \nabla_{X}^2 F \llcorner \nabla_{X} F = \lambda\, X.
\end{equation}
By \eqref{eq:nabla perp}, $\nabla_{X}F =  d \,F(X^*)\, X + \nabla^\perp_{X} F$ with $\nabla^{\perp}_{X}F \in  \vspan{X}^{\perp}=\vspan{X_1, \ldots,  X_{n-1}}$,
and by Euler identity, $\nabla_{X}^2 F \llcorner X = (d-1) \nabla_X F$. 
Using \eqref{eq:nabla2 dec}, we deduce that
\begin{align}
    \nabla_{X}^2 F \llcorner \nabla_{X} F & = d (d-1) F(X^*) \nabla^2_X F \llcorner X + \nabla^2_X F \llcorner \nabla^\perp_X F\\
    &  = d (d-1) F(X^*) \nabla_X F  
    +2 X (\nabla_{X} (D_X F) \llcorner \nabla^\perp_X F )
    + \nabla^{\perp,2}_X F \llcorner \nabla^\perp_X F
\end{align}
By projection of \eqref{eq:discd1} on $X^{\perp}= \vspan{X_1, \ldots, X_{n-1}}$, we get 
\begin{align}
    0 & = (1-d) F(X^*)\nabla_{X}^{\perp} F + \frac{1}{d} \big ( (d-1) \, d \, F(X^*)\nabla_{X}^{\perp} F + \nabla^{\perp, 2}_{X}F \llcorner \nabla^{\perp}_{X} F \big)\\
    &= \nabla^{\perp,2}_{X}F \llcorner \nabla^{\perp}_{X} F.
\end{align}
\end{proof}


We consider first the distance to the discriminants of order $k$.
\begin{proposition}\label{prop:distdisc} 
For any real form $F\in \Hdn$ we have 
\begin{align}
\dist_B(F, \disc^{k,\R}_{d})^2\ =\ \min_{X\in \Sn^{\R}} \delta^k_{F}(X)
\end{align}
and a closest point to $F$ on $\disc^{k,\R}_{d}$ is $R^{k+1}_{X_{\min}}(f)$ where $X_{\min}$ is a minimizer of $\delta^k_{F}$ on $\Sn^\R$.
\end{proposition}
\begin{proof}
For $X \in \Sn^{\R}$, $G \in \disc^k_{d,X^*}$, we have 
$$
\norm{F-G}_B^2= \norm{T^{k}_{X}(F-G)}_B^2 + \norm{R^{k+1}_{X}(F-G)}_B^2= \norm{T^{k}_{X}(F)}_B^2 + \norm{R^{k+1}_{X}(F)-G}_B^2,
$$
since $R^{k+1}_{X}(G)=G$ and $T^{k}_{X}(G)=0$.
We deduce that
\begin{align*}
\dist_B(F, \disc^k_{d, X^*})^2 &= \norm{T^{k}_{X}(F)}_B^2 = \delta^k_F(X),
\quad \dist_B(F, \disc^{k,\R}_{d})^2 = \min_{X \in \Sn^\R} \delta^k_{F}(X),
\end{align*}
and a closest point to $F$ on $\disc^{k,\R}_{d}$ is $R^{k+1}_{X_{\min}}(F)$, where $X_{\min}$ is a minimizer of $\delta^k_{F}$ on $\Sn^\R$.
\end{proof}
A particular case concerns the distance to the discriminant hypersurface $\disc^\R_d=\disc_d^{1,\R}$. Equation 
\eqref{eq:delta1 value} shows that 
for any real form $F\in \H_{d}$, 
{\small\begin{equation}\label{eq:distdisc1}
\dist_B(F, \disc_{d}^\R)^2 
    = \min_{X \in \Sn^\R} \left(F(X^*)^2 + \frac{1}{d} \norm{\nabla^{\perp}_{X}F}^2\right) 
    = \min_{X \in \Sn^\R} \left((1-d) F(X^*)^2 + \frac{1}{d} \norm{\nabla_{X} F}^2\right).
\end{equation}}
Hence computing  the distance to $\disc^{\R}_{d}$ reduces to solving the  polynomial minimization problem \eqref{eq:distdisc1}. 
We illustrate it in Example \ref{ex:binary cubic}.

Next, we consider the distance to the osculating varieties.
\begin{proposition}\label{prop:dist_osc} For $F\in \H_{d}$, 
\begin{align}
\dist(F,\cl O^{k}_{d})^2 = \norm{F}_B^2 - \max_{X\in \Sn^\R} \delta^k_F(X)
\end{align}
and a closest point to $F$ on $\Ok_{d}$ is $T^k_{X_{\max}}(F)$, where $X_{\max}\ \in \Sn^\R$ maximizes $\delta^k_{F}$ on $\Sn$.
\end{proposition}
\begin{proof}
For $X\in \Sn^\R$ and $G \in \cl O^{k}_{d,X}$,
\begin{align}
\norm{F-G}_B^2= \norm{T^{k}_{X}(F-G)}_B^2 + \norm{R^{k+1}_{X}(F-G)}_B^2= \norm{T^{k}_{X}(F)-G}_B^2 + \norm{R^{k+1}_{X}(F)}_B^2,
\end{align}
since $R^{k+1}_{X}(G)=0$, $T^{k}_{X}(G)=G$ and for any $H\in \H_{d}, \scp{T^{k}_{X}(H), R^{k+1}_{X}(H)}_B=0$.
We deduce by Lemma \ref{lem:norm_Delta_K} that 
\begin{align*}
\dist_B(F, \Ok_{d,X})^2 &= \norm{R^{k+1}_{X}(F)}_B^2 =  \norm{F}_B^2 -\delta^k_F(X).
\end{align*}
Consequently,
$\dist_B(F, \Ok_{d})^2= \norm{F}_B^2 - \max_{X\in \Sn^\R} \delta^k_F(X)$.

We verify that $G= T^{k}_{X_{\max}}(F)$, for $X_{\max}$ a maximizer of $\delta^k_F$ on $\Sn^\R$, is a closest point to $F$ on $\Ok_{d}$, since $G\in \Ok_{d}$, and
$$
\norm{F-G}_B^2 = \norm{R^{k+1}_{X_{\max}}(F)}_B^2 = \norm{F}_B^2 - \delta^k_F(X_{\max})= \dist(F, \Ok_{d}).
$$\end{proof}
\begin{remark}
As in Proposition \ref{prop:distdisc}, computing the distance to $\Ok_{d}$  
 reduces to solving the (polynomial) \emph{maximization} problem on $\Sn^\R$ with the objective function $\delta^k_F(X)$.  The minimal distance is reached at a point
$T^k_{X_{\max}}(F)$, where $X\in \Sn^\R$ is a \emph{maximizer} of $\delta^k_F$.
For $k=0$ we have $\delta_F^0(X) = F(X^*)^2$ and so the distance minimization to the Veronese variety $\Ver=\Odn^0$ is equivalent to maximization of $F(X^*)$ over the sphere. This is well-known in the context of best rank-one approximations of symmetric tensors, see \cite[Thm. 2.19]{QiLuo}.
\end{remark}

We can notice, that for any $F\in \H_{d}$, 
\begin{align}
 F = T_{X}^k(F) + R_{X}^{k+1}(F),
\end{align}
with $R_{X}^{k+1}(F) \in \disc_{d,X^*}^{k}$.
Since $\scp{T^{k}_{X}(F), R^{k+1}_{X}(F)}_B=0$, 
$R_{X}^{k+1}(F)$ is the orthogonal projection (with respect to the product \eqref{eq:Bomb}) of $F$ on the linear space $\disc_{d}^{k}(X)$ and
$\dist(F, \disc_{d,X^*}^{k})^2= \norm{F - R_{X}^{k+1}(F)}_B^2 = \norm{T_{X}^{k}(F)}_B^2 = \delta^k_{F}(X)$. 
As $\Disc^{k,\R} = \cup_{X\in \Sn^\R} \disc_{d,X^*}^{k}$, we recover, in another way that $\dist_B(F, \disc^{k,\R}_{d})^2 = \min_{X\in \Sn^\R} \delta^k_{F}(X)$.
This shows that the minimal distance to  $\disc^{k,\R}_{d}$ is reached at a point
 $R^{k+1}_{X}(F)$, 
where $X\in \Sn^\R$ minimizes the function 
$\delta^k_F$. Thus, computing a closest point on $\disc^{k,\R}_{d}$ reduces to solving a polynomial optimization problem on  $\Sn^\R$.



\begin{example}\label{ex:binary cubic}
We consider the space $\H_{3,2}$ of binary cubics. In the affine hyperplane consisting of $F=X_1^3+aX_0X_1^2+bX_0^2X_1+cX_0^3$, the discriminant $\Xi_{3,2}$ is defined by
\begin{align}
   \Delta\ =\ a^2 b^2 - 4b^3 - 4a^3 c - 27c^2 + 18 a b c
\end{align}
and its real part is depicted in Fig. \ref{fig:binary cubics}.  
\begin{figure}
    \centering
    \includegraphics[width=0.7\linewidth]{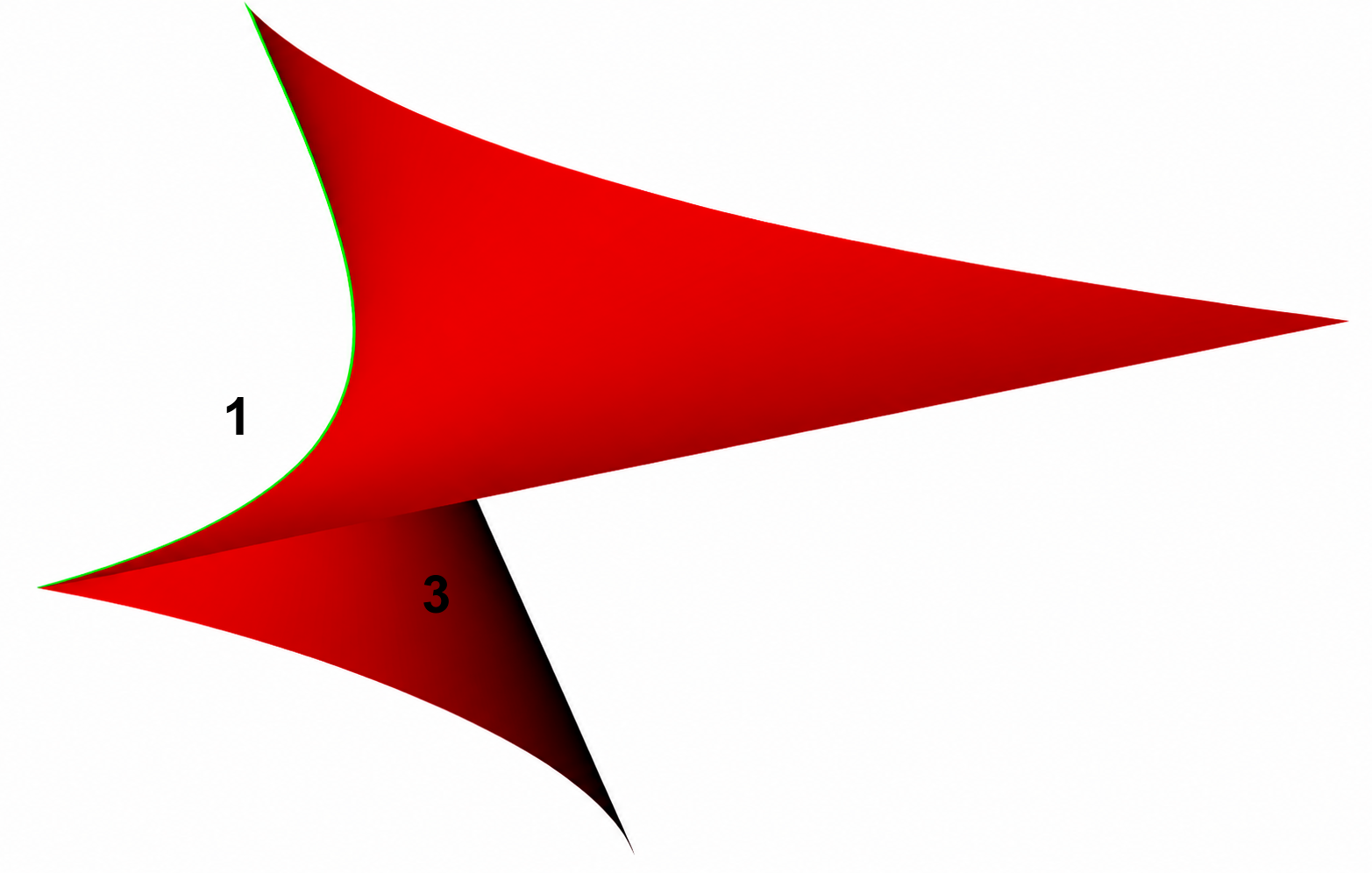}
    \caption{\centering The discriminant $\Xi^\R_3$ of real binary cubics divides the space into regions of cubics with $1$ or $3$ real zeros. Its singular locus (green) consisting of cubics with a triple zero coincides with the Veronese curve.}
    \label{fig:binary cubics}
\end{figure}  
The binary cubics that belong to this surface have (at least) a double root. By Proposition \ref{prop:sing_discr}, the cubics $F$ on the singular locus of  $\disc_{3,2}$ have a triple root (since a cubic cannot have two distinct double roots), that is, $F=X^3$
for some $X\in \H_{1,2}$. Therefore, the singular locus of $\disc_{3,2}$ is the Veronese variety,
$\sing(\disc_{3,2}) = \cl V_{3,2}$. It is a very special situation, where the dual of a variety coincides with its singular locus. The red surface on Fig. \ref{fig:binary cubics} separates the space of real cubics into two connected components, in one of which the cubics have one real zero, whereas in the other - three real zeros.

Given $F = X_1^3 + X_1^2X_0+ 2\, X_1X_0^2+ 3\, X_0^3 \in \H_{3,2}$, we want to compute a form with a real singular point, which is the closest to $F$ with respect to the apolar/Bombieri-Weyl norm. This reduces to solving the following polynomial optimization problem:
\begin{align}\label{eq:opt cubic}
\inf\  & -2 F^2(X^*) + \frac 1 3 \big( (D_{X_0} F(X^*)^2 + (D_{X_1} F(X^*))^2 \big)\\
\textup{s.t.}\  & \norm{X}^2=1
\end{align}
Using Moment-Sum of Squares relaxation techniques \cite{lasserre2010moments} and the package \textup{MomentPolynomialOpt.jl}\footnote{\url{https://github.com/AlgebraicGeometricModeling/MomentPolynomialOpt.jl/}}, we obtain two opposite minimizers, for a relaxation of order $5$. One of them is $X_{\min}^*\approx (-0.262683, 0.964882)$, which yields the closest real singular cubic $F^*$:
{\small\begin{align}
 F^* & :=  R^2_{X_{\min}}(F) =  F- T^1_{X_{\min}}(F) \\
 & = F - F(X_{\min}^*)\, (X_{\min})^3 - (-X^*_{\min,1}\, D_{X_0} F(X_{\min}^*)+  X^*_{\min,0}\, D_{X_1} F(X_{\min}^*) ) X_{\min}^{\perp} (X_{\min})^2 \\
& \approx 0.059248\, X_1^3 + 0.652888\, X_1^2 X_0 + 2.398174\, X_1 \, X_0^2 + 2.936309\, X_0^3
\end{align}}%
We notice that $T^1_{X_{\min}}(F) = T^1_{-X_{\min}}(F)$, so that $F$ has a unique closest singular cubic $F_{\min}$.
We verify that $F_{\min}$ is singular at $X_{\min}^*$,
and that $\nabla^{\perp,2}_{X_{\min}}F_{\min} = 0$, so that $F_{\min} \in \cl C_{3,2}$ is a singular point in $\disc_{3,2}^{\mathrm{sing}} = \cl V_{3,2}$ and decomposes
as $F_{\min} \approx 3.268729\, (X_{\min}^\perp)^3$.

We consider now the binary cubic $F = X_1^3 - X_1^2X_0 -X_1X_0^2 +1.02\,X_0^3 = (X_1-X_0)^2 (X_1 +X_0) + 0.02 X_0^3$.
The corresponding optimization problem \eqref{eq:opt cubic} yields the minimizer $X_{\min}^* \approx (0.70446, 0.709744)$ and its opposite.
The closest singular form to $F$ is
$$
F^* = R^2_{X_{\min}}(F) = 
0.997486\,X_1^3 - 1.007486\,X_1^2 X_0² - 1.00743\,X_1 X_0^2 + 1.017542\,X_0^3,
$$
which is singular at $X_{\min}^*$. It is a smooth point of 
$\disc_{3,2}$, since $\nabla_{\x^*}^{\perp,2}F^* \approx 5.699161 (X_{\min}^\perp)^2$ is a non-degenerate quadric on $X_{\min}^\perp$.

These examples show that a minimizer of the distance to the discriminant $\disc_{3,2}$ can be either a smooth point of $\disc_{3,2}$ or on $\cl C_{d,2}$. 
In Proposition \ref{prop:critical point}, we will see that this is also the case for the critical points of $\delta_F^1(\x)$. 
\end{example}

Next, we show that the distance minimization problem to $\Disc$ from a generic $F\in \Hdn$ has a unique solution.

\begin{proposition}\label{prop:unique_min}
    Let $F\in \Hdn$ be a generic form. Then $\delta_F^1: \Sn^\R\rightarrow \R$ has only nondegenerate critical points and on each critical level there is only one pair of antipodal critical points. In particular, $\delta_F^1$ has a unique pair of minimizers $X_{\min}, -X_{\min} \in \Sn^\R$ that are nondegenerate.
\end{proposition}

\begin{proof}
By Lemma \ref{lem:delta1 crit}, $X\in \Sn$ is a critical point of $\delta_F^1$ if and only if $\nabla^{2,\perp}_{X} F\llcorner \nabla^{\perp}_X F=0$.
Let us consider a map $L:\Hdn\times \Sn^\R\rightarrow T^*\Sn^{\R}$, $L(F,X)=(X,\nabla^{2,\perp}_{X} F\llcorner \nabla^{\perp}_X F)$, where $T\BM{^*}\Sn^{\R}$ is the cotangent bundle of the sphere.
We will show that $L$ is transverse to the zero section $\Sn^{\R}=\{(X,0): X\in \Sn^{\R}\}\subset T^*\Sn^{\R}$ and then, (semialgebraic) parametric transversality lemma \cite[Thm. 2]{LeLoi2006Density} (see also \cite[Thm 6.35]{Lee2013} and \cite[Thm. 9.6.2]{BCR:98}) imply that for a generic $F\in \Hdn$, $L_F:\Sn^{\R}\rightarrow T^*\Sn^{\R}$, $L_F(X)=L(F,X)$, is transverse to $\Sn^{\R}\subset T^*\Sn^{\R}$, which is equivalent for critical points of $\delta_F^1$ to be nondegenerate.
To apply the parametric transversality lemma, we compute the differential of $L$ at $F\in \Hdn$ in the direction of $H\in \Hdn$:
\begin{align}
    D_H L(F,X)\ =\ 
 \nabla^{2,\perp}_X H \llcorner \nabla^{\perp}_X F+ \nabla^{2, \perp}_X F \llcorner \nabla^{\perp}_X H
\end{align}
Note that for a generic $F\in \Hdn$, the quadratic form $\nabla^{2,\perp}_X F$ is of rank $n-1$ at every its critical point $X^*$.
Indeed, for this, consider the incidence variety
\begin{align}
    \mathcal{IV}\ =\ \{(F,X)\in \Hdn\times \Sn\,:\, \nabla^{\perp}_X F =0,\ \textrm{rank} \nabla^{2,\perp}_X F< n-1\}
\end{align}
The fiber of the second projection over $X\in \Sn$ is identified with a Zariski closed subset of the vector space $\{F\in \Hdn: \nabla^\perp_X F=0\}$ (which is of codimension $n-1$).
It is proper, as it does not contain $F = X_1^{d-2} (X_2^2+\dots+X_{n}^2)$ (with $X, X_2,\dots, X_n$ forming an orthonormal basis of $\Hun$) having $\BM{\frac 1 2}\,  \nabla^{2,\perp}_X F=X_2^2+\dots+X_n^2$ of rank $n-1$ \BM{(by Taylor expansion \eqref{eq:taylor})}.
As all fibers are isomorphic (they are transformed into each other by the action on $\Hdn$ induced by a rotation of $\Sn$), the incidence variety $\mathcal{IV}$ has dimension not greater than $\dim \Sn+ (\dim \Hdn -\dim \Sn - 1)=\dim \Hdn -1$. Hence, its projection on the first factor is a semialgebraic subset of $\Hdn$ of dimension at most $\dim \Hdn -1$. 
It consists of $F\in \Hdn$ having critical points $X\in \Sn$ with $\nabla^{2,\perp}_X F$ of rank less than $n-1$.
So, for a generic $F\in \Hdn$ at every critical point $X\in \Sn$ the quadratic form $\nabla^{2,\perp}_X F$ is of rank $n-1$.
\medskip

If $X$ is a critical point of $\delta_F^1$ such that $\nabla_X^{\perp} F=0$, then $D_H L(F,X)=\nabla_X^{2,\perp} F \llcorner \nabla^{\perp}_X H$ is a surjective linear map (given by the quadratic form $\nabla^{2,\perp}_X F$ that is of rank $n-1$) from $\Hdn$ to $T_X\Sn$ (actually, it only depends on the linear part $\nabla^\perp_X H$ of $H\in \Hdn$).
If $X$ is a critical point of $\delta_F^1$ such that $\nabla_X^{\perp} F\neq 0$, we restrict $D_H L(F,X)$ to the subspace $H\in \Hdn$ such that $\nabla^{\perp}_X H=0$.
However, even on this smaller subspace the map $H\in \Hdn\mapsto D_H L(F,X)=\nabla^{2,\perp}_X H\llcorner \nabla_X^\perp F\in T_X \Sn$ is surjective, as different terms of the Taylor expansion \eqref{eq:taylor} are independent from each other and so quadratic forms $\nabla^{2,\perp}_X H$ can be chosen arbitrarily.
Then \cite[Thm. 2]{LeLoi2006Density} yields the first claim of the proposition.

Consider now an open subset  $\Delta=\{(X,Y)\in \Sn\times \Sn: X\neq \pm Y\}$ of pairs of linear forms in the unit sphere and consider a map $K: (\Hdn\setminus \Disc)\times \Delta\rightarrow \R$ defined as
\begin{align}
    K(F,X,Y)\ =\ \delta_F^1(X)-\delta_F^1(Y).
\end{align}
Note that the second claim of the proposition is equivalent to the regularity of the equation $K(F,X,Y)=0$, that is, $\nabla_X K=\nabla_Y K=0$ does not have a solution $(X,Y)\in \Delta$ such that $\delta_F^1(X)=\delta_F^1(Y)$.  
To show that the above equation is regular, we again exploit parametric transversality argument.
For this let us compute the differential of $K$ with respect to $F\in \Hdn$,
{\small \begin{align}
    D_H K(F,X,Y)\ =\ \left(2F(X^*)H(X^*)+\frac{2}{d}\langle \nabla^{\perp}_X F, \nabla^{\perp}_X H\rangle\right) -
    \left(2F(Y^*)H(Y^*)+\frac{2}{d}\langle \nabla^{\perp}_Y F, \nabla^{\perp}_Y H\rangle\right).
\end{align}}For any $(F,X,Y)\in (\Hdn\setminus \Disc)\times \Delta$, we construct $H\in \Hdn$ such that $D_H K(F,X,Y)\neq 0$ which implies that $0$ is a regular value for the map $K$.
Assume first that $F(X^*)\neq 0$.
Since $s:=\langle X,Y\rangle\in (-1,1)$ (cf. $(X,Y)\in \Delta$), we can define a degree $d$ form
\begin{align*}
    H\ =\ X^{d-3}(X-s Y)^2 ((1-3s^2)X+2s Y)
\end{align*}
that satisfies $H(X^*)=(1-s^2)^3>0$, $H(Y^*)=0$, $\nabla^{\perp}_X H=0$, $\nabla_Y^{\perp} H=0$ and hence $D_H K(F,X,Y)=2 F(X^*) H(X^*)\neq 0$.
If $F(X^*)=0$, we must have $\nabla^{\perp}_X F\neq 0$ (since $F\notin \Disc$ is non-singular).
Then a degree $d$ form
\begin{align*}
    H\ =\ X^{d-3}(X-sY)^2 \nabla^{\perp}_X F 
\end{align*}
satisfies $H(X^*)=H(Y^*)=0$, $\nabla^\perp_X H=(1-s^2)^2 \nabla^\perp_X F$, $\nabla^\perp_Y H=0$ and hence $D_H K(F,X,Y)=\frac{2}{d}(1-s^2)^2\Vert \nabla^\perp_X F\Vert^2\neq 0$.

\end{proof}

We illustrate this proposition by the following example.

\begin{example}\label{ex:cusp_min}
\BM{Let $\kk=\R$.}
We consider a perturbation $F_t=t X_0^{d-1}X_1+F$, $t>0$, of the degree $d$ form $F=X_0^{d-2}(X_2^2+\dots+X_{n-1}^2)+X_0^{d-3} X_1^3+\sum_{i=1}^{n-1} a_i X_i^d$ constructed in Example \ref{ex:formeA_2}.
We know that $X_0^*=[1:0:\dots:0]$ is the only singular point of $F_0=F$, it is of type $A_2$ and so $F$ is in the smooth locus of $\mathcal{C}_{d}$. 
Recall from Proposition \ref{prop:distdisc} 
that the squared distance from $G\in \Hdn$ to the space $\disc_{d,Y^*}$ of forms singular at $Y^*\in \BM{\P{\cl{E}}}$ is given by $\Vert T_Y(G)\Vert^2_B=G(Y^*)^2(1-d)+\frac{1}{d}\Vert \nabla_Y G\Vert^2$.
In particular, as the monomial $X_0^{d-1}X_1$ is orthogonal to the space $\disc_{d,X_0^*}$, we have $T_{X_0}(F_t)=t X_0^{d-1}X_1$ and 
\begin{align}\label{eq:local_ineq}
\dist_B(F_t,\disc_{d})\ \leq\ \dist_B(F_t,\disc_{d,X_0^*})\ =\ \Vert t X_0^{d-1}X_1\Vert_B=\frac{t}{\sqrt{d}}.
\end{align}
Also, since $F$ has no other singularities in $\BM{\P{\cl{E}}}$, for $Y\in \BM{\Sn^\R}\setminus U$ outside any Euclidean ball $U$ of $X_0$ we have (by compactness) 
$\dist_B(F,\disc_{d,Y^*})\geq c$ for some constant $c>0$.
Therefore, $\dist_B(F_t,\disc_{d,Y^*})\geq \dist_B(F,\disc_{d,Y^*})-t/\sqrt{d}\geq c/2$ for sufficiently small $t>0$ and by \eqref{eq:local_ineq} the closest form in $\Xi_{d}^\R$ to $F_t$ can have a real singularity only at a point in $U$.
On the other hand, the partial derivatives of $\frac{\partial F_t}{\partial X_1}, \frac{\partial F_t}{\partial X_i}$, $i=2,\dots, n-1$, written in the local chart on $\Sn$ near \BM{$X_0$} 
\BM{given by $X= y_0 X_0+ \cdots + y_{n-1} X_{n-1}$ with $y_0=\sqrt{1-y_1^2-\dots-y_{n-1}^2}$}, are
\begin{align}
    \frac{\partial F_t}{\partial X_1}\ &=\ t y_0^{d-1}+3y_0^{d-3}y_1^2+d a_1 y_1^{d-1}\ =\ t+3y_1^2 -t \frac{(d-1)}{2}\Vert \bm y\Vert^2+\mathcal{O}(\Vert \bm y\Vert^3)\\
    \frac{\partial F_t}{\partial X_i}\ &=\ 2y_0^{d-2}y_i+ d a_i y_i^{d-1}\ =\ 2y_i+\mathcal{O}(\Vert \bm y\Vert^3),\quad i=2,\dots, n-1.
\end{align}
Therefore, if $0<t<\min\{\frac{6}{d-1},\frac{2}{\sqrt{d-1}}\}$ we have
\begin{align}
    \Vert \nabla_X F_t\Vert^2\ =\ t^2\left(1-\frac{d-1}{2} \Vert \bm y\Vert^2\right)^2+6t X_1^2+4\sum_{i=2}^{n-1}  X_i^2+\mathcal{O}(\Vert \bm \Vert^3)\ \geq\ t^2
\end{align}
in a neighborhood $U$ of \BM{$X_0$} (written in the local coordinates) with equality if only if $y_1=y_2=\dots=y_{n-1}=0$.
It follows that the distance from $F_t$ to $\disc_{d,Y^*}$ for $Y\in U$ cannot be smaller than $t/\sqrt{d}$ with equality only for $X=X_0$.


\end{example}

\begin{example}
In this example, we consider the discriminant in the space $\H_{4,2}$ of binary quartics.
To visualize this $3$-fold, we consider a hyperplane $H$ containing quartics of the form $F=X_1^4+a X_0^2X_1^2+ b  X_0^3X_1+ c X_0^4$, where $a, b, c$ are arbitrary coefficients.
Among such forms, singular ones are defined \BM{by the polynomial equations} $F(1,X)=D_1F(1,X)=0$.
Eliminating the variable $X$, we obtain that (in an affine chart) the hyperplane section $H\cap \disc_{4}$ is defined by
\begin{align}
   \Delta\ =\ 16a^4c - 4a^3b^2 - 128a^2c^2 + 144ab^2c - 27b^4 + 256c^3.
\end{align}
The zero set $\disc'=\Xi'^{\, \R}\cup \cl M'$ of $\Delta$ in $\R^3$ is the union of the \emph{swallowtail surface} \cite[p. 47]{ArnoldGuseinZadeVarchenko1985} with the parabola $\cl M'$ given by $b=0$, $4c=a^2$, see Figure \ref{fig:swallowtail}.
Denoting by $\BM{X^*}$ the double zero of $F(1,X)=X^4+aX^2+bX+c$, one can parametrize $\disc_4$ via $b=-4(X^*)^3-2a(X^*)$, $c=3(X^*)^4+a\,(X^*)^2$, see \cite[p.48]{ArnoldGuseinZadeVarchenko1985}. 
The surface ${\disc'^{\R}}$ is the semialgebraic set obtained as the image of this parametrization for real $a$ and $X$.
Points on the curve $\cl M'$ correspond to polynomials $F(1,X)$ having two double zeros \BM{(i.e. they belongs to the Maxwell component $\cl M_{4,2}$ of $\sing(\disc_{4,2})$, see Proposition \ref{prop:sing_discr})}.
Indeed, since $F(1,X)$ does not have cubic term, its double zeros must be $X^*$ and $-X^*$ and so $F(1,X)=(X-X^*)^2(X+X^*)^2=X^4-2(X^*)^2 X^2+ (X^*)^4$.
We then have $a=-2(X^*)^2$, $b=0$ and $c=(X^*)^4$.
We distinguish two cases. If $X^*$ is real, $F(X^*,1)$ trace the self-intersection locus of $\disc'$, it is one branch of the parabola $\cl M'$ given by $a \leq 0$.
If $X^*$ is a non-real zero, so is its conjugate $\overline{X^*}=-X^*$. 
Then $X^* = \BM{\sqrt{-1}\, \zeta}$ for some real $\zeta\in \R$ and $a=2 \zeta^2$, $b=0$ and $c=\zeta^4$ trace the other branch of $\cl M'$ given by $a>0$. Note that this second branch is disconnected from the swallowtail surface. 
Smooth points of $\disc'$ correspond to polynomials with one double zero.

Points on the edge curve $\cl C'$ correspond to polynomials with one triple zero \BM{(i.e. points on the caustic component $\cl C_{4,2}$ of $\disc_{4,2}$, see Proposition \ref{prop:sing_discr})}. 
The swallowtail surface $\disc'^{\, \R}$ is then (an affine part of) the hyperplane section $\disc'^{\, \R}\cap H$ of the real discriminant,
it separates $\R^3$ into regions corresponding to quartic polynomials $F(1,X)$ with different number of real zeros (see Fig. \ref{fig:swallowtail}).
\end{example}

\begin{figure}
    \centering
    \includegraphics[width=0.7\linewidth]{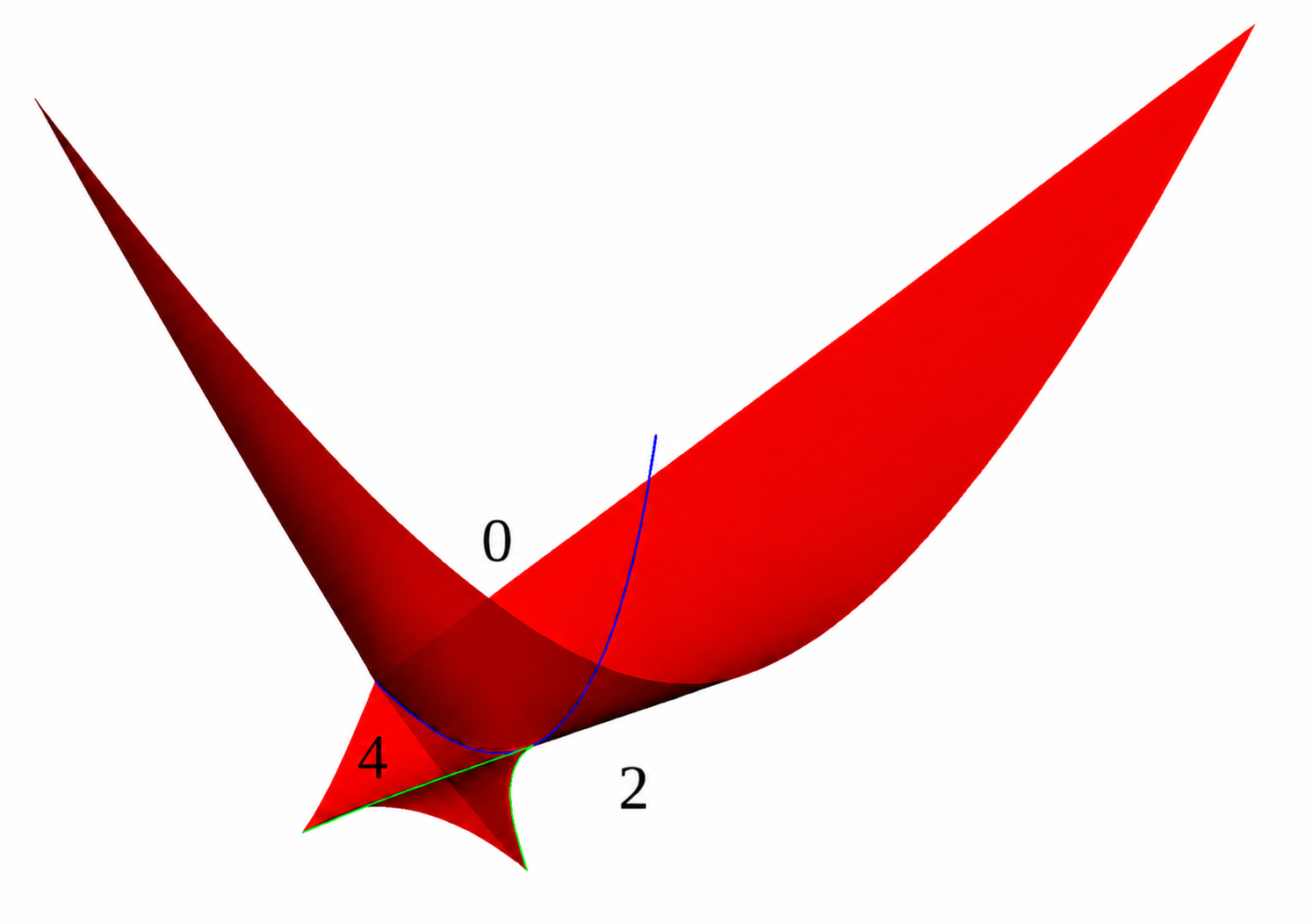}
    \caption{The swallowtail surface $\disc'^{\,\R}$ (red) together with the parabola $\cl M'$ (blue) constitute the real part of a hyperplane section of $\disc_{4,2}^{\C}(\R)$. It divides the space into regions of real quartics with $4$, $2$ or $0$ real zeros respectively.}
    \label{fig:swallowtail}
\end{figure}

\section{The critical points of the distance functions}\label{sec::crit_points}
In this section, we analyse the critical points of $\delta^{k}_F(X)$
and relate them to the critical points of the distance to $\disc_{d}, \cl C_{d}$ and their duals $\Ver$, $\Tan$, for a generic $F\in \H_{d}$.

\subsection{The geometry of the  critical points}
We analyse first the critical points of the distance functions $\delta_F^k(X)$ on the sphere $\Sn$.

\begin{proposition}\label{prop:critical point d0}
For $F\in \Hdn$, $X$ is a critical point of $\delta^0_F$ on $\Sn$, if and only if, either $F(X^*) = 0$ or $\nabla^{\perp}_{X}F= 0$.

For $F$ generic and $X$ a critical point of $\delta^0_F$ such that $F(X^*)\neq 0$, we have $R^{1}_{X}(F)= R^{2}_{X}(F) \in \smooth(\Disc)$ and $R^{1}_{X}(F)$ is a critical point of $\dist(F, \Disc)^2$, and  $T_{X}^0(F) \in \Ver^{\times}$. 

\end{proposition}
\begin{proof}
The critical points $X$ of  $\delta^0_{F}(X)=F(X^*)^2$ on $\Sn$
satisfy $\nabla^{\perp}_{X}(F^2) = 2 F(X^*) \nabla^{\perp}_{X}F =0$, which proves the first part of the proposition.


Let us prove that for $F$ generic and $X \in \Hun$ such that $\nabla_{X}^{\perp} F= 0$, we have $\rank \nabla_{X}^{2,\perp} F = n-1$. 
Let $X_0, \ldots, X_{n-1} \in \Hun$ be a basis of $\Hun$ so that $X= X_{\lb}$ for some $\lb \in \kk^n$. Then 
$F= \varphi_0(\lb, H_0, H_1, \ldots, H_d)$ with  $H_i\in \hat \H_i$, $H_i^{\lb}= \nabla^{\perp,i}_{X_{\lb}} F$ and $H_1^{\lb}= 0$ or equivalently $H_1=0$. 
We consider the irreducible incidence variety 
\begin{align}
    \cl G_1 &= {\{ (\lb, H_0, H_1, \ldots, H_d) \in \kk^{n-1} \times \hat \H_0\times \hat \H_1 \cdots \times \hat \H_d \mid H_1 = 0 \}} \\
    &= \kk^{n-1}\times \hat \H_0 \times \set{0} \times \hat \H_2 \times \cdots \times \hat \H_d.
\end{align}
As $X_\lb$ is a critical point of $\delta_F^0$ for a generic $F\in \Hdn$, we have \BM{$\nabla^{\perp}_{X}F=0$ and by Taylor expansion \ref{eq:taylor}}, $F\in \varphi_0(\cl G_1)$ and $\zariski{\cl G_1}= \Hdn$.

Let us prove that 
$${\cl G}_1'=\{ (\lb, H_0, 0, H_2, \ldots, H_d) \in \cl G_1 \mid \rank H_2 = n-1, \varphi_{2}(\lb, H_2, \ldots, H_d) \not\in \cl M_{d} \}.
$$
is a non-empty open subset of $\cl G_1$.
The Example \ref{ex:smooth disc} shows that  
\begin{align}
F & = X_0^d + X_0^{d-2} (\sum_{i=1}^{n-1}X_i^2) + \sum_{i=1}^{n-1} a_i\,X_i^d
  = \varphi_{0}(\bm 0, 1, 0, \sum_{i=1}^{n-1}X_i^2, 0, \ldots, 0, \sum_{i=1}^{n-1} a_i\,X_i^d)
\end{align}
with $a_i\in \kk$ non-zero generic coefficients, is such that $R_{X_0}^2(F)\in \smooth(\Disc)$, $\nap F=0$ and $F(X_0^*)=1 \neq 0$.

This implies that $\cl G_1'\neq\emptyset$ is dense in $\cl G_1$ and that $\zariski{\varphi_0(\cl G'_1)}= \H_d$.
Thus
 for $F$ generic, that is, $F=\varphi_0(\lb, H_0, 0, H_2, \ldots, H_d)\in \varphi_0(\cl G'_1)$,  $X_\lb^*$ is a critical point of $F$, $\rank \nabla_{X_\lb}^{\perp, 2} F= \rank H_2^{\lb} = n-1$ and $R^2_{X_\lb}(F)= \varphi_{2}(\lb, H_2, \ldots, H_d) \not \in \cl M_{d}$ has a single singularity $X_\lb^*$.
By Proposition \ref{prop:sing_discr}, this implies that $R_{X_0}^2(F)$ is a smooth point of $\Disc$ and concludes the proof.
\end{proof}

\begin{proposition}\label{prop:critical point}
For $X\in \Hun^{\times}$ a critical point of $\delta^1_F$ on $\Sn$, 
\begin{itemize}
    \item either $\nabla^{\perp}_{X}F=0$;  then ${R^{2}_{X}(F)} \in \Disc$ and $T_{X}^1(F)\in \Tan^{\sing} = \Ver$ (for $d>2$). For $F$ generic,  ${R^{2}_{X}(F)} \in \smooth(\Disc)$ and is a critical point of $\dist(F,\disc_{d})^2$;
    \item or $\nabla^{\perp}_{X}F \neq 0$; then  ${R^{2}_{X}(F)} \in \cl C_{d} \subset \Disc^{\sing}$ and $T_{X}^1(F) \in \Tan$. For $F$ generic, ${R^{2}_{X}(F)} \in \smooth(\cl C_{d})$ and is a critical point of 
    $\dist(F, \cl C_{d})^2$, and $T_{X}^1(F) \in \smooth(\Tan)$.
\end{itemize} 
\end{proposition}
\begin{proof}
By Lemma \ref{lem:delta1 crit}, $X$ is a critical point of $\delta^1_F$, if and only if, 
$\nabla^{\perp,2}_{X}F \llcorner \nabla^{\perp}_{X} F= 0$.
For $F\in \Hdn$ 
generic, the system of equations $\nabla^{\perp,2}_{X}F \llcorner \nabla^{\perp}_{X} F= 0$ has solutions $X \in \Hun^{\times}$ such that $\nabla^{\perp}_{X} F\neq 0$ (as in the example \ref{ex:formeA_2}) and solutions such that $\nabla^{\perp}_{X} F=0$. We analyse separately the two cases:

$\bullet$
If $\nabla_{X}^{\perp}F=0$, then $T_{X}^1(F) = F(X^*) X^d \in \Ver = \sing(\Tan)$ (see Proposition \ref{prop:osc dim sing}). By construction, $R_{X}^2(F) \in \Disc$.
For $F$ generic and $X$ such that $\nabla_{X}^{\perp} F=0$, it is proved 
in Proposition  \ref{prop:critical point d0}, that $R^2_{X}(F)$ is a smooth point of $\disc_{d}$.

As $\delta_F^1(X) = \norm{T^1_{X}(F)}^2 = \norm{F}^2 - \norm{R_{X}^2(F)}^2$ and $X$ is a critical point of $\delta_F^1$, $R_{X}^2(F) \in \Disc$
is a critical point of $\dist^2(F,\Disc)$.

$\bullet$ In the case $\nabla_{X}^\perp F\neq 0$, then  $\nabla_{X}^{\perp, 2} F \llcorner \nabla_X^\perp F = 0$ and   
$R^2_{X}(F) \in \cl C_{d}$ (see Proposition \ref{prop:sing_discr}).

Let us prove that for $F$ generic and $X \in \Hun$ such that $\nabla_{X}^{\perp} F\neq 0$  and $X$ a critical point of $\delta^1_F$,  $R^2_{X}(F) \in \smooth(\cl C_{d})$.
Let $\set{X_0, \ldots, X_{n-1}}$ be a basis of $\Hun$ such that $X= X_\lb$ for some $\lb\in \kk^{n-1}$.
The form $F\in \H_d$ with the critical point $X_{\lb}$  of $\delta^1_F$  and with $\nabla^{\perp}_{X_{\lb}}F \neq 0$ is such that $F= \varphi_0(\lb, H_0, H_1, \ldots, H_d)$ for $H_i\in \hat \H_i$, $H_i^{\lb}= \nabla^{\perp,i}_{X_{\lb}} F$ and 
$H_1^\lb \in \ann_1(H_2^\lb)^{\times}$, i.e. $H_2^{\lb} \llcorner H_1^{\lb}= 0$ with $H_1^{\lb}\neq 0$.
We consider the incidence variety 
$$
\cl G_2 = {\{ (\lb, H_0, H_1, \ldots, H_d) \in \kk^{n-1} \times \hat \H_0\times \cdots \times \hat \H_d \mid 
H_1 \in \ann_1(H_2)^{\times}
\}}.
$$ 
Since $H_2^{\lb} \llcorner H_1^{\lb} = 0$, if and only if,  $H_2 \llcorner H_1 = 0$ and $H_1^\lb\neq 0$, if and only if, $H_1\neq 0$, we deduce that for $F\in \Hdn$ generic, $F \in \varphi_{0}(\cl G_2)$, so that $\zariski{\varphi_{0}(\cl G_2)} = \Hdn$.

As for $H_1\neq 0$, 
$\set{H_2 \in \hat \H_2\mid H_2 \llcorner H_1}$ is a linear space of codimension $n-1$, the fibers of the projection of $\cl G_2$ on $\hat \H_1$ are linear spaces of codimension $n-1$. Thus, $\zariski{\cl G_2}$ is an irreducible variety.
Consider the open subset
$$
\cl G_2'= \{ (\lb,  H_0, \ldots, H_d) \in \cl G_2 \mid \rank H_2 = n-2, 
H_{3} \llcorner H_1^3 \neq 0, \varphi_2(\lb, H_2, \ldots, H_d) \not\in \cl M_{d} \}. 
$$
By Proposition \ref{prop:dual_ver_tan}, 
 $R^2_{X_\lb}(\varphi_0(\lb,H_0, \ldots, H_d))$ is a smooth point of $\cl C_{d}$, for 
 any $(\lb,H_0, \ldots, H_d)\in \cl G_2'$.
 
Example \ref{ex:formeA_2} shows that 
\begin{align}
F&= X_0^d + X_0^{d-1} X_1 + X_0^{d-2} \sum_{i=2}^{n-1} X_i^2  + X_0^{d-3} X_1^3 + \sum_{i=1}^{n-1}a_i X_i^d \\
&=
\varphi_0(\bm 0, 1, X_1, \sum_i^{n-1} X_i^2 , X_1^3, 0, \ldots, 0, \sum_{i=1}^{n-1}a_i X_i^d) 
\end{align}
with $a_i\in \kk$ generic non-zero coefficients is such that $\rank \napp F= n-2$ and $R_{X_0}^2(F) \in \cl C_{d}\setminus \cl M_{d}$.
This implies that $\cl G_2'$ is not empty and thus dense in ${\cl G_2}$ and $\varphi_0(\cl G_2')$ is dense in $\Hdn$. 
Consequently, for $F\in \varphi_0(\cl G_2')$ generic with a critical point $X_\lb$ of $\delta^1_F$ such that $\nabla_{X_\lb}^\perp F \neq 0$, $R^2_{X_\lb}(F) \in \smooth(\cl C_{d})$ is a smooth point of $\cl C_d$.

By Proposition \ref{prop:osc dim sing}, $T_{X_\lb}^1(F) \in \smooth(\Tan)$ since $\nabla^{\perp}_{X_\lb}F \neq 0$ and $T_{X_\lb}^1(F) \not\in \Ver$.

As $\delta_F^1(X_\lb) = \norm{T^1_{X_\lb}(F)}^2 = \norm{F}^2 - \norm{R_{X_\lb}^2(F)}^2$ and $R_{X_\lb}^2(F) \in \cl C_d$  for $X_{\lb}$ a critical point of $\delta_F^1$ such that $\nabla_{X_\lb}^\perp F \neq 0$, we deduce that $R_{X_\lb}^2(F) \in \cl C_d$ is a critical point of $\dist(F,\cl C_d)^2$. This concludes the proof of the proposition.
\end{proof}

\subsection{Counting critical points}
Hereafter, we analyse the number of critical points of the square distance function to the discriminant and 
relate it to the Euclidean Distance Degree of $\Ver$ and $\Tan$.

Let $\C^N$ be endowed with a real (that is, defined over $\R$) positive definite symmetric bilinear form $\langle\cdot,\cdot\rangle$ (thus when restricted to $\R^N$ it defines a scalar product) and denote by $\Vert \cdot \Vert^2$ the square of the associated (real) quadratic form.
For a real variety \BM{$\cl X$}, consider the Euclidean distance minimization problem from $U\in \R^N$,
\begin{align}\label{eq:opt}
\inf\  & \Vert U - \x\Vert^2\\
      & \x \in \cl X(\R)
\end{align}
where $\cl X(\R)\subset \R^N$ is the real locus of $\cl X$.
\emph{The Euclidean Distance Degree} of $\cl X$ is defined as the number 
 \begin{align}\label{eq:EDD}
    \EDD(\BM{\cl X})\ =\ \#\left\{\x\in \smooth(\BM{\cl X})\,:\, \langle\mathbf{u}-\x,\dot\x\rangle=0\ \forall\ \dot\x\in T_{\x} \BM{\cl X}\right\}
 \end{align}
of complex critical points of the optimization problem \eqref{eq:opt} for a generic $\mathbf{u}\in \R^N$.
Such $\x\in \smooth(\BM{\cl X})$ are usually referred to as \emph{ED critical points} (relative to $\mathbf{u}$).

    This invariant has been introduced and systematically studied in \cite{DHOST}.
It has also appeared in \cite{Josse_Pene} as \emph{the normal cycle} of $\BM{\cl X}$.
One can interpret \eqref{eq:EDD} as \emph{the algebraic complexity} of the optimization problem \eqref{eq:opt}, that is, the number of (complex) solutions to the equations defining critical points of the distance function from a general input $\mathbf{u}\in \R^N$ to $\BM{\cl X}(\R)$.

It is shown in \cite{DHOST}[Thm. 5.2], that $\EDD(\BM{\cl X})=\EDD(\BM{\cl X}^{\vee})$, where $\BM{\cl X}^{\vee}$ is the dual variety of $\BM{\cl X}$.

By its definition, $\EDD(\BM{\cl X})$ accounts only for points in the smooth locus of the variety.
It can happen though (and does indeed happen for the discriminant $\disc_{d}$ as we have seen in Example \ref{ex:binary cubic} that the closest point in $\BM{\cl X}(\R)$ to $\mathbf{u}\in \R^N$ from a full-dimensional set of inputs is a singular point of $\BM{\cl X}$.


For $\BM{\cl X}=\disc_{d}$, we have the following result:
\begin{proposition}
For a generic $F\in \Hdn$, the number $\nu_{\delta_F^1}$ of critical points $\delta^1_F$ on $\Sn$ is 
\begin{align}
 \nu_{\delta^1_F}\ = \ 2(\EDD(\Disc) + \EDD(\cl C_{d}))
 = 2(\EDD(\Ver) + \EDD(\Tan))
\end{align}
\end{proposition}
\begin{proof}
By Proposition \ref{prop:critical point}, for a generic $F$ if $X$ is a critical point of $\delta^1_F$ on $\Sn$, then $R^2_X(F)$ is either critical point of $\dist^2$ to $\smooth(\Disc)$ or a critical point of $\dist^2$ to $\smooth(\cl C_{d})$.
Since $R^2_X(F)=F-T^1_X(F)$ and $T^1_X(F)=F(X^*)X^d+\nabla^\perp_X F X^{d-1}=X^{d-1}(F(X^*)X+\nabla^\perp_X F)$, 
we have that $R^2_X(F)=R^2_Y(F)$ for $X, Y\in \Sn$ if and only if $T^1_F(X)=T^1_F(Y)$ which holds if and only if $X=\pm Y$.
Therefore, $\nu_{\delta^1_F} =2( \EDD(\Disc) + \EDD(\cl C_{d})) = 2(\EDD(\Ver) + \EDD(\Tan))$, where the second equality follows by the projective duality.
\end{proof}
A naive approach for counting critical points of $\delta_F^1$ would be to apply Bézout theorem to the system
\begin{align}\label{eq:parallel}
    (1-d)F(X^*)\nabla_X F+\frac{1}{d}\Vert X\Vert^2 \nabla^2_X F\llcorner \nabla_X F=\lambda^{2d-2} X,\quad [X:\lambda]\in \P^n_{\C}.
 \end{align}
Since critical points of $\delta_F^1$ are isolated (for a generic $F\in \Hdn$), we have
\begin{align}\label{eq:bad}
    \nu_{\delta_F^1}\ =\ \#\left\{X\in \Sn\,:\, \nabla_X\delta_F^1=\lambda X\ \textrm{for some}\ \lambda\in \C\right\}\ \leq\ 2\frac{(2d-1)^n-1}{2d-2},
\end{align}
for their number, where the right-hand side is essentially Bézout theorem (see \cite{FS, CS2013}).
However, \eqref{eq:parallel} has infinitely many solutions given by some points on the isotropic quadric $\cl Q=\{X\in \BM{\P_{\C}\Hun}: \Vert X\Vert^2=0\}$.
This implies that the bound \eqref{eq:bad} is not optimal.
In the next subsection we show that the count \eqref{eq:sum} of the number of critical points of $\delta_F^1$ on $\Sn$ that we obtain via the topological approach is much smaller than the naive bound \eqref{eq:bad} (cf. 
\eqref{eq:CS} and \eqref{eq:final_formula_simpl}).

\subsection{Euclidean Distance Degree of the Veronese variety}

For the discriminant variety $\disc_{d}$ (equivalently, for its projectively dual $\mathcal{V}_{d}=(\disc_{d})^\vee$), the ED degree had been computed by several authors \cite{FS, CS2013, ASS2017}. Hereafter, we recall this classical result and provide yet another proof, involving Euler characteristic techniques that will be used to compute $\EDD(\Tan)$.
It uses the following theorem of \cite{AluffiHarris2018}, see also 
\cite[Theorem 1.2]{MaximRodriguezWang2021}: 
\begin{theorem}[\cite{AluffiHarris2018}]\label{thm:Aluffi}
Let $\BM{\cl X}$ be a smooth irreducible subvariety of $\P_{\C}^{N-1}$, and assume that $\BM{\cl X}\not \subset \cl Q$, where $\cl Q=\{[x_1:\ldots : x_N]\in \P_{\C}^{N-1}; x_1^2+ \cdots + x_N^2 = 0 \}$. Then the Euclidean Distance Degree of the affine cone $C(\BM{\cl X}) \subset \C^{N}$ over $\BM{\cl X}$ is given by  
\begin{align}
\EDD (C(\BM{\cl X}) ) = (-1)^{\dim \BM{\cl X}} \chi (\BM{\cl X}\setminus (\cl Q\cup H)),  
\end{align}
where $H$ is a generic hyperplane of $\P^{N-1}_{\C}$.  
\end{theorem}
This theorem yields the Euclidean Distance Degree of $\Ver$ or $\Disc$:
\begin{proposition}\label{prop:EDD_Sigma}
\begin{align}\label{eq:EDD_Sigma}
    \EDD(\disc_{d})\ =\ \EDD(\Ver )\ =\ \frac{(d-1)^n-1}{d-2}.
\end{align}
\end{proposition}

\begin{proof}

We apply Theorem \ref{thm:Aluffi} to $\BM{\cl X}$ being the projectivisation of $\Ver$ and $\cl Q$ the isotropic quadratic subvariety ${\Vert f\Vert }_B^2 ={\langle f, f\rangle}_B=0$ of the Bombieri scalar product.  Its pullback to $\C^n$ is given by the equation $(\sum_{i=1}^n x_i^2)^d=0$ so as a set it is the isotropic quadratic subvariety of the standard dot product.  The pullback to $\C^n$ of a generic linear form on  $\C^N$ is a generic form of degree $d$. 

By definition \eqref{eq:Veronese} the projective Veronese variety can be identified with $\P_{\C}^{n-1}= \P\Ver \subset \P^{N-1}$.   
Therefore, by Theorem \ref{thm:Aluffi} and \eqref{eq:Hirzebruchdiff} of Proposition \ref{prop:fromHirzebruch}, see Appendix,  
\begin{align}\label{eq:CS}
\EDD (\Ver) = (-1)^{n-1} \chi(\P^{n-1}  \setminus (X^{n-2}_{2} \cup X^{n-2}_{d}))  = 
\frac {(d-1)^{n} - 1} {d-2}.   
\end{align} 
Here by $X^{n-2}_{i}$ we denote the generic hypersurface of $\P^{n-1}$ of degree $i$.  
\end{proof}

By the biduality theorem (see \cite[Thm. 5.2]{DHOST}) and Proposition \ref{prop:distdisc}, the critical points of the square distance from $f\in \H_{d}$ to $\Ver$ are such that 
$f= T_{\x}^0(f)+ R^1_{\x}(f)$ with $T_{\x}^0(f)= f(\x) (\x, \cdot)^d\in \Ver$ and $\x \in \Sn$ is a critical point of $\delta^0_{f}(\x)= f^2(\x)$ on $\Sn$.
In this case, $\nabla_{\x}^{\perp}f=0$ and $R^1_{\x}(f)= R^2_{\x}(f) = f - T_{\x}^0(f) \in \Disc$ is a smooth point of $\Disc$ for $f$ generic.

In \cite{Kozhasov2018} it is shown that there exists a generic real form $f\in \H_{d}$ relative to which all ED critical points of the distance minimization problem to $\Ver$ are real. 

\medskip

\subsection{Euclidean Distance Degree of the Tangential variety}

To compute the Euclidean Distance Degree of the Tangential, we follow the same approach as for the Veronese variety. We replace Theorem \ref{thm:Aluffi} by Theorem 1.3 of \cite{MaximRodriguezWang2021}:

\begin{theorem}[{\cite[Theorem 1.3]{MaximRodriguezWang2021}}]\label{thm:Maxim2}
Let $\cl X$ be an irreducible subvariety of $\P^{N-1}$, and assume that $\cl X\not \subset \cl Q$, where $\cl Q=\{[x_1:\ldots : x_N]\in \P^N; x_1^2+ \cdots x_N^2 = 0 \}$. Then the Euclidean Distance Degree of the affine cone $C(\cl X) \subset \C^{N}$ over $\cl X$ is given by 
\begin{align}\label{eq:Max2}
\EDD (C(\cl X) ) = (-1)^{\dim \cl X} \chi (Eu_{\cl X}|_{\cl X\setminus (\cl Q\cup H)}),  
\end{align}
where $H$ is a generic hyperplane and $Eu_{\cl X}$ is the local Euler obstruction of $\cl X$.    
\end{theorem}

The local Euler obstruction is equal to $1$ on $\BM{\smooth(\cl X)}$ and is constant on strata of any Whitney stratification of $\cl X$. 

\begin{theorem} \label{thm:EDDtangential}
The Euclidean Distance Degree of the tangential variety \BM{$\Tan$} to the Veronese variety $\Ver$ and of its dual $\cl C_{d}$ is  
\begin{align}\label{eq:final_formula_simpl}
\EDD(\Tan) = \EDD(\cl C_{d})  & = \frac {d} {2 (d-2)^2} \, \big( (2d-3)^{n} - 2 (d-1)^{n} + 1\big)
\end{align}
\end{theorem}

It gives the following formulae for  $n\le 6$:
\begin{center}
\begin{tabular}{c|l}
         n & $EDD(\Tan)$ \\ \hline
         $1$ & $d$ \\
         $2$ & $3\, d^2 -3 d$\\
         $3$ & $7\,d^3 - 16\,d^2 + 10\,d$\\
         $4$ & $15\,d^4 - 55\,d^3 + 70\,d^2 - 30\,d$\\
         $5$ &$31\,d^5 - 158\,d^4 + 309\,d^3 - 272\,d^2 + 91\,d$\\
         $6$ & $63\,d^6 - 413\,d^5 + 1099\,d^4 - 1477\,d^3 + 1001\,d^2 - 273\,d$
\end{tabular}
\end{center}


\begin{proof}
The main differences from the proof of Proposition \ref{eq:EDD_Sigma} are that  the projectivisation $\P \Tan$ of $\Tan$ is singular and it can  be parameterised by $\P_{\C}^{n-1} \times \P_{\C}^{n-1}$, rather than  by $\P_{\C}^{n-1}$ as in the case of the Veronese variety. Consequently,  the inverse image of the isotropic quadric $\cl Q$ has a much more complicated structure than that in the case of the Veronese variety.  Furthermore, the  parameterisation 
\begin{align}\label{eq:Psi}
\Psi : \P_{\C}^{n-1} \times \P_{\C}^{n-1} \to \P \Tan, \quad \Psi(X, Y) = X^{d-1} Y
\end{align} 
of the projectivisation $\P \Tan$,  is bijective if $d\ge 3$, but not an isomorphism.
 
The singular set of $\P \Tan$ is the projective Veronese variety, the image of the diagonal $\Delta \subset \P_{\C}^{n-1} \times \P_{\C}^{n-1}$.    
Thus $\P \Tan$ admits a stratification $\BM{( \smooth(\P \Tan),   \sing(\P \Tan))} = (\P \Tan \setminus \P \Ver, \P \Ver)$. 
The group $G=\P \rm{Gl}(n)$ acts on $\P_{\C}^{n-1} \times \P^{n-1}$ preserving 
$\Delta$ and this action commutes with the action of $G$ on $\P_{\C}^{N-1}$ if we identify $\C^{N}$ with $\Sym^d (\C^{n})$. Moreover this action is transitive on $\Delta$. Therefore for any two points $p_i \in \P \Ver, i=1,2$,  the germs $(\P_{\C} ^{N-1}, \P  \Tan, \P \Ver, p_i)$ are isomorphic. In particular, this stratification is Whitney.

The Euler obstruction of $\P \Tan$ at the points of $\P \Ver$ is by\eqref{rem:Eu_obstruction} equal to $1+(-1)^n$.    Therefore, the formula  \eqref{eq:Max2} reads 
\begin{align}\label{eq:Max3} \notag 
\EDD (\Tan ) & = (-1)^{\dim  \P \Tan} \big ( \chi ( \smooth(\P \Tan)\setminus (\cl Q\cup H)) + (1+(-1)^n) \chi ( \sing(\P \Tan)\setminus (\cl Q\cup H))
\big ) \\ 
& =   \big ( \chi ( (\P \Tan)\setminus (\cl Q\cup H))  + (-1)^n \chi (\P \Ver\setminus (\cl Q\cup H)), \\ \notag 
& =   n^2 - \chi ((\P \Tan) \cap (\cl Q\cup H)) - \EDD (\Ver),  
\end{align}
where $H$ is a generic hyperplane of $\P_{\C}^{N-1}$ and $\cl Q$ is the isotropic quadric.
Because $\Psi$ of \eqref{eq:Psi} is a homeomorphism 
\begin{align}
\chi ((\P \Tan) \cap (\cl Q\cup H)) = \chi (\Psi^{-1}(\cl Q \cup H).
\end{align}
The Euler characteristic $\chi ((\P \Tan) \cap (\cl Q\cup H))$ will be computed in 
Lemma \ref{lem:chiQcupH}.  Using this lemma we have 
\begin{align}
\EDD (\Tan ) &= n^2 - \big (2n \chi^{n-2}_{2} - (\chi^{n-2}_{2})^2  \big ) 
 +(-1)^{n} 
\big (\chi^{n-2}_{2d} - \chi^{n-4}_{2,d,2d-2} + \chi^{n-5}_{2,d,2d-2, 2d} - \chi^{n-3}_{2,d} \big ) \\ \notag 
 &  \quad + (-1)^ {n} \bigl ( n - \chi^{n-2}_{2} -\chi^{n-2}_{d} +  \chi^{n-3}_{2,d})  \\ \notag 
 & = (n - \chi^{n-2}_{2})^2    
 +(-1)^{n} 
\big (  n + \chi^{n-2}_{2d} - \chi^{n-2}_{2} -\chi^{n-2}_{d} - \chi^{n-4}_{2,d,2d-2} + \chi^{n-5}_{2,d,2d-2, 2d} \big ) 
\\ \notag 
& = (-1)^{n} 
\big ( \chi^{n-2}_{2d}  -\chi^{n-2}_{d} - \chi^{n-4}_{2,d,2d-2} + \chi^{n-5}_{2,d,2d-2, 2d} \big ) ,
\end{align}
where, as in Appendix, $\chin{n}{d_1, \ldots , d_k}$ denotes the Euler characteristic of a smooth complete intersection $X^n_{d_1, . . . , d_k} \subset \P^{n+k}$ of multidegree  $d_1, \ldots , d_k$.  
The last equality follows from Corollary \ref{cor:d=2}. Simplifying the above formula (using a computer algebra system), we get \eqref{eq:final_formula_simpl}. 
\end{proof}

Now we complete the details of the proof of Theorem \ref{thm:EDDtangential}. 
Because $H$ is chosen as a generic hyperplane its inverse image by $\Psi$, denoted by $H_{d-1,1}$, is the zero set of a generic form of bidegree $(d-1,1)$:
\begin{align}\label{eq:H1}
\sum _{i=1}^n f_i(x) v_i = 0 .
\end{align}
Then the set $\{x\in \P_{\C}^{n-1}; f_i(x)=0, i=1, \ldots, n\}$ is empty.
The pull-back of the Bombieri form to $\P_{\C}^{n-1}\times \P_{\C}^{n-1}$ is   
\begin{align*}
& \frac 1 {d!} \big ( \sum x_i \partial_{i} \big)^{d-1} \big ( \sum v_i \partial_i \big) \, 
\big ( \sum x_i t_i \big)^{d-1} \big ( \sum v_i t_i \big) = \\ 
& = \frac {1}{d} \big (\sum x_i^2 \big )^{d-2}\Big ( \big (\sum x_i^2 \big )\big (\sum v_i^2 \big ) + (d-1)\big (\sum x_i v_i \big )^2 \Big ) = \frac 1 d F_1(x) F_2(x,v),
\end{align*}
where $ F_1(x) = (\sum x_i^2 \big )^{d-2}$ and we may express $F_2(x,v)$ as 
\begin{align}
F_2(x,v) & = \Big ( \big (\sum x_i^2 \big )\big (\sum v_i^2 \big ) + (d-1)\big (\sum x_i v_i \big )^2 \Big )  = v^T A(x) v, \\ & \quad \text{ where }
A(x) = \sum x_i^2I + (d-1) x x^T . 
\end{align}

Denote $Q_1 = F_1^{-1} (0)$, $Q_2 = F_2^{-1} (0)$ considered as subsets of 
$\P_{\C}^{n-1}\times \P_{\C}^{n-1}$. Note that $Q_1 = \tilde Q_1 \times \P_{\C}^{n-1}$ , where $\tilde Q_1$ is the isotropic quadric of the first factor $\P_{\C}^{n-1}$.  The following lemma follows from the form of the matrix $A(x)$. 

\begin{lemma}\label{lem:singQ2}
The fibres of the projection of $Q_2$ onto the first factor are either nonsingular quadrics, if 
$\sum x_i^2\ne 0$, or double hyperplanes otherwise.  
\end{lemma}

Let us now study the intersection of $Q_2 \cap H_{d-1,1}$ over the set $\sum x_i^2\ne 0$.  For $x$ fixed this is the intersection of a nonsingular quadric and a hyperplane.  Such an itersection is either nonsingular or has one isolated non-degenerate singularity (i.e. with the Milnor number equal to $1$).  

\begin{lemma}\label{lem:singQ2capH}
The set of $x \in \P_{\C}^{n-1}\setminus \tilde Q_1$ such that $Q_2 \cap H_{d-1,1}$ is singular over $x$ is given by the equation 
\begin{align}
\f(x)^T B(x) \f(x) =0, 
\end{align}
where $B(x):= \sum x_i^2I + \frac {1-d} d x x^T $, provided  $\f=(f_1, \ldots, f_n)$ is chosen sufficiently generic.  
\end{lemma}

\begin{proof}
Consider the fibres of projection of $Q_2 \cap H_{d-1,1} \subset \P_{\C}^{n-1} \times \P_{\C}^{n-1}$ onto the first factor $\P_{\C}^{n-1}$. 
Such a fibre is singular at $(x,v)\in \P_{\C}^{n-1} \times \P_{\C}^{n-1} $ if the derivatives 
$\partial_v F_2 = 2 A(x) v$ and $\partial_v (\sum_i f_iv_i) = \f$ are proportional, i.e. if there is a non-zero $\lambda$ such that $\f(x) = \lambda A(x) v$. 

If  $\sum x_i^2\ne 0$ then $A(x)$ is invertible and $(\sum x_i^2)^{-2} B(x)$ is its inverse.  In this case we have to verify whether the unique  possible candidate for $v$, given by: 
 \begin{align}
v = A(x)^{-1} \f(x) = (\sum x_i^2)^{-2} B(x) \f(x)    
\end{align}
satisfies $(x,v)\in H_{d-1,1}$ that is  $\sum_i f_iv_i = \f(x)^T B(x) \f(x) =0$. (Equivalently we may check whether  
$(x,v)\in Q_2$ that is $v^TA(x) v = \f(x)^T B(x)^T A(x) B(x) \f(x) = (\sum x_i^2 ) \f(x)^T B(x) \f(x) =0$.)  
\end{proof}

Let $\AUX :=\{F=0\} \subset \P^{n-1}$, where 
\begin{align}\label{eq:functionF}
F(x):= \f(x)^T B(x) \f(x) = (\sum x_i^2)  (\sum f_i^2) + \frac {1-d} d  (\sum x_i f_i)^2 .  
\end{align}
 If $F(x)\ne 0 $  (and $\sum x_i^2 \ne 0 $) then, by the above, the fiber of the projection of $Q_2 \cap H_{d-1,1}$ on the factor factor  $\P^{n-1}$ is nonsingular, that is a nondegenerate quadric $\BM{Q^{n-3}_2}$ of Euler characteristic $\chi^{n-3}_2 = \chi (\BM{Q}^{n-3}_2)$.  If the fiber is singular, this is a singular quadric with single single, isolated, and forcibly nondegenerate singularity, i.e. with Milnor number $1$.  Therefore, by \cite{Parusinski1988}, see also \cite{PP}, the Euler characteristic of this fiber equals $\chi^{n-3}_2 - (-1)^{n-3}$.  Therefore 
\begin{align}\label{eq:formula1}
\chi ((Q_2 \setminus Q_1) \cap H_{d-1,1}) = \chi (\P^{n-1} \setminus \tilde Q_1) \chi^{n-3}_2 - (-1)^{n-3} \chi(\AUX \setminus 
\tilde Q_1 ).  
\end{align}

The hypersurface $\AUX = F^{-1}(0)$ could be singular even if $\f=(f_1, \ldots , f_n)$ is chosen generic. Our goal is to compute $\chi(\AUX\setminus \tilde Q_2)$ for $\f$ generic.  
Write 
$f_i= \sum_{|\alpha|=d-1} f_{i,\alpha} x^\alpha$ and consider $F $ 
as a function of $x \in \P^{n-1}$ and $(f_{i,\alpha}; i=1, \ldots n; |\alpha|=d-1)\in \P^M $, where $M=n \binom{n-d-2}{d-2}-1$. We compute the singular part of its zero set $F^{-1} (0)=\AUX \subset \P_{\C}^{n-1}\times \P^M$ by considering  
\begin{align}
& \frac {\partial F} {\partial {f_{j,\alpha}}} = 2\big ((\sum x_i^2)  f_j  + \frac {1-d} d  (\sum x_i f_i)x_j  \big ) x^\alpha ,\\
& \frac {\partial F} {\partial {x_j}} = \sum_{i,\alpha}  \frac {\partial F} {\partial {f_{j,\alpha}}} f_{j,\alpha} \alpha_j x^{ \alpha - e_j} +  2\big ( (\sum f_i^2)x_j + \frac {1-d} d  (\sum x_i f_i)f_j \big ) .  
\end{align}
If all $\frac {\partial F} {\partial {f_{j,\alpha}}}$ are zero and $F=0$ then $\sum x_i^2= \sum x_i f_i=0$.  Indeed, suppose this is not the case.  Then $\f$ and $x$ are proportional: $\f=\lambda x$, and $F(x,\f)= \lambda ^2 (\sum x_i^2)^2 ( 1 + \frac {1-d} d  )$ that implies first $\sum x_i^2=0$ and then $ \sum x_i f_i=0$, that contradicts the assumption.  Moreover, if we suppose that all partial derivatives of $F$ with respect to $x_j$ are zero, we can conclude that $\sum f_i^2=0$.  Denote 
\begin{align}
h_1:= \sum x_i^2, h_2:= \sum f_i^2(x), h_3 := \sum x_i f_i(x) .
\end{align} 
By the above computation the singular part of $\AUX$ equals $\{ h_1(x)= h_2(x) = h_3(x) = 0\}$.  
Then by Sard's Theorem, for $\f$ generic,  
\begin{align}
\sing(\AUX) = \{ h_1(x)= h_2(x) = h_3(x) = 0\} . 
\end{align}

\begin{lemma}\label{lem:tranversality}
If $\f$ is sufficiently generic then the set of points $x\ne 0$ such that $x$ is parallel to $\f(x)$ and $\sum x_i^2 =0$ is empty. 
\end{lemma}

\begin{proof}
 Over the open subset of $ \P^{n-1}$ given by $x_k\ne 0$ the intersection of the varieties  $f_i x_k-f_k x_i =0$, $i\ne k$ and $\sum x_i^2 =0$ is empty because by Sard  these varieties are smooth and transverse, provided $\f$ is generic.  Now we apply this argument for every $k$.  
\end{proof} 

\begin{corollary}\label{cor:h_i}
If $\f$ is sufficiently generic then the hypersurfaces $h_i=0$ of $\P^{n-1}$ are smooth and transverse. 
\end{corollary}

\begin{proof}
We again consider $h_i(x,\f)$ as the functions of $x$ and $f_{i,\alpha}$ as parameters.  
Then 
\begin{align}\label{eq:partialh}
 \frac {\partial h_1} {\partial {f_{j,\alpha}}} =0, \quad \frac {\partial h_2} {\partial {f_{j,\alpha}}} = 2  f_j   x^\alpha , \quad \frac {\partial h_3} {\partial {f_{j,\alpha}}} = 2  x_j   x^\alpha .  
\end{align} 
Therefore the hypersurfaces $h_i=0, i=2,3$, are smooth and they are tangent exactly at the points where $x$ is parallel to $\f$, that is an empty set for generic $\f$ by 
Lemma \ref{lem:tranversality}.  

For $h_1$, $\partial_x h_1 = 2x $ is nonzero as a vector. The gradient of $h_1$ cannot be combination of the gradients of $h_2$ and $h_3$ because for these functions the partial gradients with respect 
to all ${f_{j,\alpha}}$ are independent.  
\end{proof}

Let $p\in \sing(\AUX)$.  Then by Corollary \ref{cor:h_i} we may complete $h_i$ to a system of local coordinates at $p$ and then,  $\AUX$ is given by the equation $h_1 h_2 + \frac {1-d} d h_3^2$.  This means that $\sing(\AUX) $ is nonsingular, $(\smooth (\AUX), \sing(\AUX))$ is a Whitney stratification of $\AUX$ with the ordinary double point as the normal slice singularity.  
This allows us to use the formula of  \cite{PP} to compute $\chi(\AUX)$.  By this formula 
\begin{align}\label{eq:chi(X)}
\chi(\AUX) = \chi^{n-2}_{2d} - \chi (\sing(\AUX) \setminus X^{n-2}_{2d}) 
 = \chi^{n-2}_{2d} - \chi ^{n-4}_{2,d,2d-2} + \chi^{n-5}_{2,d,2d-2, 2d}  . 
\end{align}

Note that $\AUX \cap \tilde Q_1= \{  \sum x_i^2=  \sum x_i f_i(x) = 0 \}$ and therefore 
$\chi (\AUX\cap \tilde Q_1)= \chi^{n-3}_{2,d} $.  
By \eqref{eq:formula1} this gives 
\begin{align}\label{eq:formula2}
 & \qquad   \chi ((Q_2 \setminus Q_1) \cap H_{d-1,1})  = \chi (\P_{\C}^{n-1} \setminus \tilde Q_1) \chi^{n-3}_2 - (-1)^{n-3} \chi( \setminus \tilde Q_1 ) \\ 
 \notag 
= & \chi (\P_{\C}^{n-1} \setminus \tilde Q_1) \chi^{n-3}_2  
- (-1)^{n-3} \big (\chi^{n-2}_{2d} - \chi^{n-4}_{2,d,2d-2} + \chi^{n-5}_{2,d,2d-2, 2d} - \chi (\cap \tilde Q_1) \big )\\ \notag 
 = & (n - \chi^{n-2}_{2}) \chi^{n-3}_{2} + (-1)^{n-2} \big (\chi^{n-2}_{2d} - \chi^{n-4}_{2,d,2d-2} + \chi^{n-5}_{2,d,2d-2, 2d} - \chi^{n-3}_{2,d} \big ).  
\end{align}

\begin{lemma}\label{lem:chiQcupH}
$\chi (\Psi^{-1} (Q\cup H) )= 2n \chi^{n-2}_{2} - (\chi^{n-2}_{2})^2   + (-1)^{n-1} \big (\chi^{n-2}_{2d} - \chi^{n-4}{2,d,2d-2} + \chi^{n-5}_{2,d,2d-2, 2d} - \chi^{n-3}_{2,d} \big )$.  
\end{lemma}

\begin{proof}   
 \begin{align}\label{eq:computation}
&\chi (\Psi^{-1} (Q\cup H) )=  \chi (Q_1\cup Q_2) + \chi (H_{d-1,1} ) - \chi ((Q_1\cup Q_2)\cap H_{d-1,1} ) 
 \\ \notag 
& = \chi (Q_1)  +  \chi (Q_2\setminus Q_1) +  \chi (H_{d-1,1} ) - \chi (Q_1\cap H_{d-1,1})
 - \chi ((Q_2\setminus Q_1)\cap H_{d-1,1} )  \\ \notag
& = n \chi^{n-2}_{2} + (n - \chi^{n-2}_{2}) \chi^{n-2}_{2} + n(n-1) - (n-1)\chi^{n-2}_{2} - 
\chi ((Q_2\setminus Q_1)\cap H_{d-1,1} )  \\ \notag
& = (n+1) \chi^{n-2}_{2} - (\chi^{n-2}_{2})^2 + n(n-1) 
\\ \notag 
& \quad - \bigl( (n - \chi^{n-2}_{2}) \chi^{n-3}_{2}  + (-1)^{n-2} \big (\chi^{n-2}_{2d} 
- \chi^{n-4}_{2,d,2d-2} + \chi^{n-5}_{2,d,2d-2, 2d} -   \chi^{n-3}_{2,d} \big ) \bigr ) \\ \notag
& = 2n \chi^{n-2}_{2} - (\chi^{n-2}_{2})^2   + (-1)^{n-1} \big (\chi^{n-2}_{2d} - \chi^{n-4}{2,d,2d-2} + \chi^{n-5}_{2,d,2d-2, 2d} - \chi^{n-3}_{2,d} \big ) . 
\end{align}
For this computation we have used besides \eqref{eq:formula2}:
\begin{enumerate}
\item 
$\chi (Q_1)=  n \chi^{n-2}_{2} $;
\item
$\chi (Q_2\setminus Q_1)=  (n - \chi^{n-2}_{2}) \chi^{n-2}_{2}$ that follows from Lemma \ref {lem:singQ2}; 
\item 
$\chi(H_{d-1,1}) = n(n-1)$  and $\chi(Q_1 \cap H_{d-1,1}) = (n-1)\chi^{n-2}_{2}$;
\item  
$(n - \chi^{n-2}_{2}) (n-1 - \chi^{n-3}_{2}) = 0$ by Corollary \ref{cor:d=2} of Appendix. 
\end{enumerate}
\end{proof}



\section{On forms maximizing the Bombieri-Weyl distance to the discriminant}\label{sec:extremal}

In Section 6, 8 and 9 of \cite{Raffalli2014} Rafalli considers the problem of characterizing forms $F\in \Hdn$ that maximize the distance to the real discriminant $\Disc^\R$ among all real forms in $\Hdn$ with the same Bombieri-Weyl norm as $F$. For the Veronese variety $\Ver$, the analogous problem has been studied in \cite{AKU}.
Thus, for example, \cite[Thm. 6.1]{Raffalli2014} yields that if $F\in \Hdn$ maximizes the distance to $\Disc^\R$ among all real forms in $\Hdn$ with unit Bombieri-Weyl norm, it lies in the span of $T^1_X(F)\in \Tan$, where $X\in \Sn$ ranges over the set of minimizers of the function $\delta_F^1$ (restricted to the $(n-1)$-sphere in the space of real linear forms) defined in \eqref{eq:delta1 value}.
A similar result for $\Ver$ is obtained in \cite{AKU}, where instead of minimizers of $\delta_F^1$ one takes maximizers of $\delta_F^0(X)=F(X^*)^2$, $X\in \Sn$.
Another interesting result from \cite{Raffalli2014} concerns \emph{positive forms}.
A real form is called \emph{positive} if the projective hypersurface in $\P^{n-1}_{\C}$ defined by it has empty real locus.
It follows from \cite[Thm. 8.4]{Raffalli2014} that, for an even $d$, the orthogonally invariant form $P_d=\Vert X\Vert^{d}$ maximizes the distance to $\Disc^\R$ with $\dist(P_d,\Disc^\R)=1$ among all positive $F\in \Hdn$ with the same Bombieri-Weyl norm as $P_d$.

\begin{remark}
From the results of  \cite{Raffalli2014} it is possible to show that $P_d=\Vert X\Vert^d$ is the unique such distance-maximizing positive form.
 Indeed, let $F\in \Hdn$ be a positive form with $\Vert F\Vert=\Vert P_d\Vert$ and write $F=c P_d+Q$ for some constant $c\in \R$ and $Q\in \Hdn$ orthogonal to $P_d$, that is, $\langle Q,P_d\rangle_B=0$.
 Positivity of $H$ and \cite[Lemma 8.1]{Raffalli2014} yield $c>0$. If $Q\neq 0$, we have $c<1$ and by \cite[Cor. 8.2]{Raffalli2014} gives $Q(X^*)<0$ at some point $X^*\in \Sn^\R$. 
Therefore, by \cite[Cor. 5.2]{Raffalli2014} we have \begin{align}
    \dist_B(F, \Disc^\R)\ \leq\ F(X^*)=c\,P_d(X^*) + Q(X^*)=c+Q(X^*)< c <1=\dist_B(P_d,\Disc^\R)
    \end{align}
  \end{remark}

In the case of binary forms ($n=2$) Raffalli posed the following conjecture.
To state it, for an integer $r=0,\dots, d$ with $d-r\in 2\mathbb{N}$, define
{\small\begin{align}\label{eq:Crd}
    C_{r,d}\ =\ (X_1^2+X_2^2)^{\frac{d-r}{2}} \textrm{Re}(X_1+i X_2)^r\ =\ (X_1^2+X_2^2)^{\frac{d-r}{2}} \sum_{k=0}^{\floor{r/2}} (-1)^k {r \choose 2k} X_1^{r-2k} X_2^{2k}
\end{align}}
On the unit circle $\Sn^{\R}$, one has $X_1=\cos t$, $X_2=\sin t$ and $C_{r,d}(\cos t, \sin t)=\cos(r \arccos t)$ coincides with the $r$-th Chebyshev polynomial of the first kind.
Because of that, in \cite{AKU} the form $C_{d,d}$ is called \emph{Chebyshev form}, it has $2d$ equally spaced zeros on $\Sn^{\R}$.

\begin{conjecture}[Conjecture 9.4 in \cite{Raffalli2014}]
Let $d/3<r\leq d$ be such that $d-r$ is even. 
Then $C_{r,d}$ maximizes the distance to the real discriminant $\Disc^{\R}$ among real binary forms $F\in \Hdn$ with $2r$ real zeros in $\Sn$ and same Bombieri-Weyl norm as $C_{r,d}$: 
\[    \dist_B(F,\Disc^{\R})\ \leq\ \frac{\Vert F\Vert_B}{\Vert C_{r,d}\Vert_B}  \]
\end{conjecture}

In the next result we prove this conjecture for $r=d$. 
A real binary form is said to be \emph{real-rooted} if it has only real zeros. 

\begin{theorem}\label{thm:extremal}
    $C_{d,d}$ maximizes the distance to the real discriminant $\Disc^\R$ among real-rooted binary forms in $\Hdn$ having the same Bombieri-Weyl norm as $C_{d,d}$. Moreover, any other maximizer is of the form $C_{d,d}\circ g$ for some orthogonal transformation $g\in O(2)$. 
\end{theorem}

\begin{proof}
The proof is based on harmonic (also known as Fischer) decomposition of the space $\Hdn$, which we need only in the case of binary forms.
For this, along with $C_{r,d}$ let us define 
\begin{align}
    S_{r,d}\ =\ (X_1^2+X_2^2)^{\frac{d-r}{2}} \sum_{k=0}^{\floor{(r-1)/2}} (-1)^k {r \choose 2k+1} X_1^{r-2k-1} X_2^{2k+1}.
\end{align}
It is known from \cite[Ch. IV, \S 2]{SteinWeiss1971} that $C_{r,d}$, $S_{r,d}$ with $d-r\in 2\mathbb{N}$ form an orthogonal basis of the space $\Hdn$ of binary forms of degree $d$ (one discards $S_{0,d}=0$ in case of even $d$).
Let us write a real-rooted form $F\in \Hdn$ in this basis, $F=a C_{d,d} + b S_{d,d} + Q$,
where $a C_{d,d} + b S_{d,d}$ is \emph{the harmonic part} of $F$ and $Q$ belongs to the orthogonal complement of $C_{d,d}$ and $S_{d,d}$.
We can assume without loss of generality that $F$ has $d$ distinct pairs of zeros $\pm X^*_1,\dots, \pm X^*_d\in \Sn^\R$, since otherwise $F\in \Disc^\R$.
By Rolle's theorem, the restriction of $F$ to the unit circle $\Sn^\R$ has $d$ pairs of critical points $\pm Y^*_1,\dots, \pm Y^*_d\in \Sn^\R$ and the values of $F$ at these points alternate in sign.
We claim that $\min_{i=1,\dots, d} \vert F(Y_i^*)\vert \leq \rho$, where $\rho:=\sqrt{a^2+b^2}$.
Indeed, if $\vert F(Y_i^*)\vert > \rho$ for all $i=1,\dots, d$, we have that $Q=F-(aC_{d,d}+b S_{d,d})$ has the same sign as $F$ at each point $Y_i^*$, $i=1,\dots, d$, since 
\begin{align}
    \max_{X^*\in \Sn^\R}\vert a \,C_{d,d}(X^*)+ b \,S_{d,d}(X^*)\vert\ =\ \max_{t\in [0,2\pi]} \vert a \cos(d t) + b \sin (d t)\vert\ =\ \sqrt{a^2+b^2}\ =\ \rho.
\end{align}
This implies that values of $Q$ at these $2d$ points alternate in sign and hence $Q$ has $d$ pairs of zeros on $\Sn^\R$.
However, this is impossible as $Q$ is divisible by $(X_1^2+X_2^2)$ and so it can have at most $2d-4$ zeros on $\Sn^\R$.

Now, by \cite[Cor. 5.2]{Raffalli2014} we have for the distance to the real discriminant from $F$:
\begin{align}
    \dist_B(F,\Disc^\R)^2\ =\ \min_{X\in \Sn^\R} \left(F(X^*)^2+ \frac{1}{d}\Vert \nabla^{\perp}_X F\Vert^2\right)\ =\ \min_{i=1,\dots, d} F(Y_i^*)^2\ \leq\ \rho^2.
\end{align}
Furthermore, we have $\Vert a\, C_{d,d}+ b\, S_{d,d}\Vert_B^2 = \rho^2 2^{d-1}$ by \cite[Fact 9.1]{Raffalli2014} and the fact that $a\, C_{d,d} + b\, S_{d,d} = \rho\, C_{d,d}\circ g$ for some orthogonal transformation $g\in O(2)$, see \cite[Sec. $1.1$]{AKU}.
It follows then $\Vert F\Vert_B^2=\Vert a\, C_{d,d}+ b\,S_{d,d}\Vert_B^2+\Vert Q\Vert_B^2 \geq \rho^2 2^{d-1}$ and
\begin{align}\label{eq:ineq_distance}
    \frac{\dist_B(F,\Disc^\R)}{\Vert F\Vert_B}\ \leq\ \frac{1}{\sqrt{2^{d-1}}}\ =\ \frac{\dist_B(C_{d,d},\Disc^\R)}{\Vert C_{d,d}\Vert_B},
\end{align}
where the equality follows from Facts 9.1 and 9.2 in \cite{Raffalli2014}.
Summarizing we see that any real-rooted form $F\in \Hdn$ with $\Vert F\Vert_B=\Vert C_{d,d}\Vert$ must be at the distance at most one from the real discriminant.
Also, we have equality in \eqref{eq:ineq_distance} if and only if $Q=0$ in which case $F=a\, C_{d,d}+b\, S_{d,d}$ with $\rho=\sqrt{a^2+b^2}=1$ (that is, $a=\cos t_0$, $b=\sin t_0$ for some $t_0\in [0,2\pi]$) can be written as $F(t,\sqrt{1-t^2})=\cos t_0 \cos(dt)+\sin t_0 \sin(dt)= \cos(dt-t_0)$ on the unit circle $\Sn^\R$ parametrized by $t\in [0,2\pi]$. This yields the second claim.
\end{proof}

\section{Some open questions and further directions}\label{sec::questions}

Recall that for a real algebraic variety $X\subset \C^N$, the Euclidean Distance Degree $\EDD(X)$ gives an upper bound on the number of real critical points of the Euclidean distance to $X^\R$ from a general point $\bm u\in \R^N$.
A priori, it is unknown whether this bound is achieved by some $\bm u$.
This is true for the Veronese variety $\Ver$ (equivalently, the discriminant $\Disc=\Ver^\vee$) and the Bombieri-Weyl product \eqref{eq:Bomb} as shown in \cite{Kozhasov2018}.
This result and \cite[Conjecture 6.5]{ASS2017} motivates the following conjecture for the tangential variety of $\Ver$ (equivalently, the caustic component $\cl C_d=\Tan^\vee$ of the discriminant).

\begin{conjecture}
There exists a generic real $F\in \Hdn$ so that $\dist_F: \Tan^\R\rightarrow \R$ (equivalently, $\dist_F: \cl C_d^\R \rightarrow \R$) has $\EDD(\Tan)=\EDD(\cl C_d)$ (cf. \eqref{eq:final_formula_simpl}) real (smooth) critical points.
\end{conjecture}

In the case of binary forms we have $\EDD(\Tan)=\EDD(\cl C_d)=d$ by \cite[Thm. 5.1]{LEE2016499} and also by Theorem \ref{thm:main}.
The distance function to $\cl C_d$ from the Chebyshev form $C_{d,d}$ defined in \eqref{eq:Crd} has $d$ real smooth critical points, see the proof of \cite[Fact 9.2]{Raffalli2014}. Hence the above conjecture holds for $n=2$. Also, validity of the conjecture for quadratic forms ($d=2$) follows from Eckart-Young theorem \cite{EY1936}. 
An even more interesting (and seemingly more difficult) question inspired by Theorem \ref{thm:bound} is whether the upper bound  on the number of real critical points of the function $\delta_F^1$ given by \eqref{eq:sum} is sharp.

\begin{problem}
 Determine whether there exists a generic real form $F\in \Hdn$ such that all  \eqref{eq:sum} critical points of $\delta_F^1$ are real.
\end{problem}

Computing $\EDD(\Tan)$ (see Theorem \ref{thm:main}) can be seen as the first step towards understanding the Euclidean Distance Degree of the $2$-nd secant variety (or, more generally, of the $r$-th secant, $r\geq 2$) of the Veronese variety.

The latter problem is largely open with the answer known only in some special cases (including the matrix case ($d=2$) due to \cite{EY1936} and for the $r$-th secant of the rational normal curve $\cl V_{2r}$ ($n=2$) of degree $2r$ for $r=1,2,3,4, 5$ due to Sefat Panah \cite{SefatPanah2026}).

\begin{problem}
    Compute the Euclidean Distance Degree of the $r$-th secant of the Veronese variety $\Ver$ with respect to \eqref{eq:Bomb}.
\end{problem}

The count given by this problem would give the measure of algebraic complexity of the low (symmetric border) rank approximation problem for symmetric tensors with respect to the Bombieri-Weyl product. 
Another problem motivated by our Theorem \ref{thm:main} is the computation of the Euclidean Distance Degree of higher-order osculating varieties \eqref{eq:osculating_variety} and higher-order discriminant varieities \eqref{eq:k_discr}.

\begin{problem}
    Compute $\EDD(\Okdn)$ and $\EDD(\Disc^k)$ with respect to the Bombieri-Weyl product \eqref{eq:Bomb}.
\end{problem}

Note that $\EDD(\mathcal{O}_d^0)=\EDD(\Ver)$ is computed in \cite{FS} and \cite{CS2013} (see also Propostion \ref{prop:EDD_Sigma}) and $\EDD(\mathcal{O}_d^1)=\EDD(\Tan)$ is given by Theorem \ref{thm:main}.

As can be seen from Examples \ref{ex:binary cubic} and \ref{ex:cusp_min}, the 
distance to the discriminant from some in $F\in\Hdn$ attains its minimum at a singular point of $\Disc$. 
One can show that for $d>2$ this situation occurs for a positive-dimensional semialgebraic set of $F\in \Hdn$.

\begin{problem}
Give a characterization of the set of $F\in \Hdn$ for which the distance to the real discriminant is attained at a singular point of $\Disc$. 
\end{problem}

On the other hand, if the real hypersurface in $\P^{n-1}_{\R}$ given by a generic real $F\in \Hdn$ is \emph{locally extremal} in the sense defined by Raffalli in \cite{Raffalli2014}, the distance to $\Disc^\R$ is minimized at a smooth point of $\Disc$ (see Proposition \ref{prop:sing_discr}).
In particular, this holds for $F\in \Hdn$ defining real hypersurfaces with maximal topology and for $F\in \Hdn$ that are strictly positive on $\P^{n-1}_{\R}$, see \cite{Raffalli2014}.

\appendix
\section{}

\subsection{Properties of the Bombieri-Weyl (apolar) product}\label{sec:prop apolar}

We recall some properties of the apolar or Bombieri-Weil product.
The first one is the well known reproducing property of \eqref{eq:Bomb} that states that Bombieri product with a power of  a linear form is just the evaluation in the respective point, that is,
\begin{align}\label{eq:eval}
\scp{F, U^d}_B\ &=\ F(U^{*}),\quad U \in \Hun,
\end{align}
for all $F\in \H_{d}$. As a consequence, we have $\Vert U^d\Vert_B=\Vert U^{*}\Vert^d$ for $U \in \R^n$ and Cauchy-Schwartz inequality yields
\begin{align}\label{eq:max}
\quad \Vert F\Vert_\infty\ \leq\ \Vert F\Vert_B\quad \forall F\in \H_{d}.
\end{align}
Furthermore, if $G\in \H_{d-1}$ is a form of degree $d-1$, formula \eqref{eq:Bomb2} yields
\begin{align}\label{eq:eval2}
\scp{F, G\, U}_B\ &
= \frac 1 d \scp{D_{U} F, G},\quad U\in \Hun,
\end{align}
where $D_{U}F = \sum_{i=1}^{n}   U_i  D_i(f)$ is the derivative of  $F$ along $U^*= (U_1, \ldots, U_n)\in \kk^n$.

In particular, taking $G=V^{d-1}$ where $V \in \Hun$, we see that Bombieri product with $U V^{d-1}$ is computed as
\begin{align}\label{eq:eval3}
    \scp{F, U \, V^{d-1}}_B\ =\ \frac{1}{d} (D_{U} F)(V^*),\quad U, V \in \Hun.
\end{align}

\subsection{Taylor expansion of a homogeneous polynomial}\label{sec:taylor}

Here we derive some basic properties of Bombieri product \eqref{eq:Bomb} in relation to the Taylor expansion of $F\in \H_{d}$.
Formula \eqref{eq:eval2} implies that for $U_0, U_1, \ldots, U_{n-1} \in \Hun$, $\alpha_1+ \cdots + \alpha_{n-1}= k$, we have
\begin{align}\label{eq:derivs}
    \scp{F, U_0^{d-k}U_1^{\alpha_1}\cdots U_{n-1}^{\alpha_{n-1}}}_B\ &=\ \frac{(d-k)!}{d!}  D_{U_1}^{\alpha_1}\cdots D_{U_{n-1}}^{\alpha_{n-1}} F(U_0^*).
\end{align}
Recall that for $U\in \Hun$, $k\in \N$ we denote by $\nabla^k_{U} F$
the degree $k$ form associated to all $k^{\textup{th}}$-order derivatives of $f$ at $U$, see \eqref{eq:nabla}.
We verify, by Euler's identity, that
\begin{align}\label{eq:euler}
    D_{U}\nabla_{U}^{k}(f) = (d-k+1) k  \nabla_{U}^{k-1} (f).
\end{align}

We have $D_{V}F(U^*) = \scp{\nabla_{U}(f), V}_{B}$
and, more generally,  for any $U:= U_0, U_1,\ldots, U_{n-1}\in \Hun$ and $\bm \alpha=(\alpha_1,\dots,\alpha_{n-1})$, $\vert\bm \alpha\vert=k$, we have
\begin{align*}
D_{U_1}^{\alpha_1}\cdots D_{U_{n-1}}^{\alpha_{n-1}} F(U^*)\ =\
\scp{\nabla_{U}^{k}F, U_1^{\alpha_1}\cdots U_{n-1}^{\alpha_{n-1}}}_{B}.
\end{align*}

Let $U:=U_0, U_1, ..., U_{n-1}$ be an \emph{orthonormal} basis of $\Hun$. Then $\big(\prod_{i=0}^{n-1} U_i^{\alpha_i}\big)_{\vert\bm\alpha\vert=d}$  is an orthogonal basis of $\H_{d}$ with respect to \eqref{eq:Bomb}.
The Taylor expansion of $F\in \Hdn$ at $U^*= U_0^*\in \kk^n$ gives

\begin{align}
F(t_0 U_0^*+ t_1 U_1^* + \cdots + t_{n-1} U_{n-1}^*)\ & =\ \sum_{|\bm \alpha|= d} \frac{1} {\alpha_1! \cdots \alpha_{n-1}!} \bm t^{\alpha}
    D_{U_1}^{\alpha_1} \cdots D_{U_{n-1}}^{\alpha_{n-1}}F(U^*)
\end{align}
for any $t_0, t_1 \ldots, t_{n-1}\in \kk$.
Since $U= U_0, U_1, \ldots, U_{n-1}$ is an orthogonal basis, for any $\bm p\in \kk^n$, we have $\sum U_i(\bm p) \, U_i^*= \bm p$ and 
\begin{align}
F(\bm p) & = \sum_{|\bm \alpha|=d} \frac{1} {\alpha_1! \cdots \alpha_{n-1}!}  U(\bm p)^{\alpha_1} U_1(\bm p)^{\alpha_1} \cdots U_{n-1}(\bm p)^{\alpha_{n-1}}
    D_{U_1}^{\alpha_1} \cdots D_{U_{n-1}}^{\alpha_{n-1}}F(U^*) 
\end{align}
As $\kk$ is of characteristic $0$, this gives the following expansion of $F$ :
\begin{equation}\label{eq:taylor}
\begin{aligned}
F & = \sum_{|\bm \alpha|=d} \frac{1} {\alpha_1! \cdots \alpha_{n-1}!} U^{\alpha_1} U_1^{\alpha_1}\cdots U_{n-1}^{\alpha_{n-1}}
    D_{U_1}^{\alpha_1} \cdots D_{U_{n-1}}^{\alpha_{n-1}}F(U^*) \\
    &  = U^d \,\nabla_{U}^{\perp,0}(F) + U^{d-1}\,  \nabla_{U}^{\perp,1}(F)  + \frac 1 {2!}  U^{d-2}\, \nabla_{U}^{\perp,2}(F)+ \cdots + \frac 1 {d!} \nabla_{U}^{\perp,d}(F),
\end{aligned}
\end{equation}
where 
\begin{align}
\nabla^{\perp,k}_{U}F = \sum_{\alpha_1 + \cdots + \alpha_{n-1}=k} {k \choose \alpha_1, \ldots, \alpha_{n-1}} U_1^{\alpha_1} \cdots U_{n-1}^{\alpha_{n-1}}
    D_{U_1}^{\alpha_1} \cdots D_{U_{n-1}}^{\alpha_{n-1}}F(U^*) \in \H_{k}.
\end{align}
We have  $\nabla_{U}^{\perp, 0}(F)= F(U)$, $\nabla_{U}^{\perp}F := \nabla_{U}^{\perp, 1}(F) =  U_1 \, D_{U_1} F(U^*) + \cdots + U_{n-1}\, D_{U_{n-1}}F(U^*)  \in \H_{1,n-1}$.
Since $U_1,\ldots, U_{n-1}$ are orthonormal with respect to the Bombieri product \eqref{eq:Bomb}, 
\begin{align}
\norm{\nabla_{U}^{\perp}F}_B= \sqrt{D_{U_1} F(U^*)^2+ \cdots + D_{U_{n-1}}F(U^*)^2}
\end{align}
By Euler identity $\scp{\nabla_{U} F, U} = d F(U^*)$ and we have 
\begin{equation}\label{eq:nabla perp}
\nabla^{\perp}_{U}F = \nabla_{U} F - \frac{1}{\norm{U}^2} \scp{\nabla_{U} F, U}\, U= \nabla_{U} F - \frac{d}{\norm{U}^2} F(U^*)\, U
\end{equation}
is the projection of $\nabla_{U} (F)$ onto $U^{\perp}$.

Similarly, applying \eqref{eq:nabla perp} to $D_UF$, we get 
\begin{align}\label{eq:nabla2 dec}
    \nabla_U^{\perp,2} F & = \sum_{1\le i,j\le n-1} U_i U_j D_{U_i} D_{U_j} F(U^*) \\&= \nabla_U^{\perp,2} F - 2 U_1 \sum_{1\leq j \le n-1} U_j D_{U_j} D_{U_1}F (U^*)  - U_1^2 D_{U_1}^2 F (U^*)\\
    & = \nabla^{2}_U F 
    - 2 U \, \nabla^{\perp}_U (D_U F)
    - U^2 D_U^2F(U^*) \\
 & = \nabla^{2}_U F
-2 U \left(\nabla_{U} (D_U F) - \frac{d-1}{\norm{U}^2} U (D_U F)(U^*)\right)
- U^2 D_U^2F(U^*) \\
&  = \nabla^{2}_U F
- 2 U \nabla_{U} (D_U F) + U^2 \left( \frac{2(d-1)}{\norm{U}^2} (D_U F)(U^*) - D_U^2F(U^*)\right)
\end{align}
We define 
\begin{align*}
T^{k}_{U}(F) & := U^d\, \nabla_{U}^{\perp, 0}(F) + \cdots + \frac 1 {k!} U^{d-k}\,  \nabla_{U}^{\perp,k}(F), \\ 
R^{k+1}_{U}(F) & := F- T^{k}_{U}(F)= \frac 1 {(k+1)!} U^{d-k-1}\, \nabla_{U}^{\perp,k+1} (F) + \cdots + \frac 1 {d!}\nabla_{U}^{\perp,d} (F).
\end{align*}


Since $\big(\prod_{i=1}^{n} U_i^{\alpha_i}\big)_{|\bm\alpha|=d}$ is an orthogonal basis of $\H_{d, n}$,  $\scp{T^{k}_{U}(F), R^{k+1}_{U}(F)}_B=0$ holds. 
Thus, $T^{k}_{U}(F)$ is the projection of $F$ on $\Ok_{d,U}$ and 
is independent of the choice of $U_1, \ldots, U_{n-1}$ provided $U, U_1, \ldots, U_{n-1}$ are orthonormal with respect to the standard pairing of $\kk^n$.
Also, $\nabla_{U}^{\perp, k}(F)$ is independent of the choice of $U_1, \ldots, U_{n-1}$, as $T^{k}_{U}(F)- T^{k-1}_{U}(F)= \frac 1 {k!} U^{d-k}\, \nabla_{U}^{\perp, k}(F)$.

For any $U\in \Hun^{\times}$, the expansion \eqref{eq:taylor} obtained by completing $U$ to an orthonormal basis $(U, U_1, \ldots, U_{n-1})$ is called the 
\emph{Taylor expansion} of $F$ at $U$. 


    
\begin{lemma} \label{lem:norm_Delta_K}
For $0\le k\le d$ we have
\begin{align} 
\norm{T^k_{U}(F)}_B^2 
= \sum_{i=0}^{k} {{d \choose i}}^{-1} \big(i!\big)^{-2}  \norm{\nabla_{U}^{\perp,i}(F)}^2
\end{align}
\end{lemma}
\begin{proof}    
Since $U_1, \ldots, U_{n-1}$ are orthogonal to $U= U_1$, we have $D_{U} (\nabla_{U}^{\perp,k}(F))=0$ and \eqref{eq:eval} then yields
\begin{align} 
    \norm{U^{d-i} \nabla_{U}^{\perp,i}(F)}_B^{2}\ &=\ 
\scp{U^{d-i} \nabla_{U}^{\perp,i}(F),U^{d-i} \nabla_{U}^{\perp,i}(F)}_{B}\\
& =\ \frac{(d-i)!\, i!}{d!} \norm{\nabla_{U}^{\perp,i}(F)}_B^2\ =\
{d \choose i}^{-1}\norm{\nabla_{U}^{\perp, i}(F)}_B^2.
\end{align}
We deduce that $\norm{T^k_{U}(F)}_B^2$ is given by
\begin{align}
& \sum_{i=0}^{k}  {(i!)^{-2}} \norm{U ^{d-i}\, \nabla^{\perp, i}_{U}(F)}_B^2\ =\ \sum_{i=0}^{k} {d \choose i}^{-1} (i!)^{-2}  \norm{\nabla^{\perp,i}_{U}(F)}_B^2\\
=\ &\sum_{i=0}^{k} {{d \choose i}}^{-1} \big(i!\big)^{-2} 
 \left\Vert\sum_{\alpha_1 + \cdots + \alpha_{n-1}=i} { i \choose \alpha_1, \ldots, \alpha_{n-1}}U_1^{\alpha_1} \cdots U_{n-1}^{\alpha_{n-1}}
    D_{U_1}^{\alpha_1} \cdots D_{U_{n-1}}^{\alpha_{n-1}}F(U^*)\right\Vert_B^2\\
=\ &\sum_{i=0}^{k} {{d \choose i}}^{-1} \big(i!\big)^{-2} 
\sum_{\alpha_1 + \cdots + \alpha_{n-1}=i} { i \choose \alpha_1, \ldots, \alpha_{n-1}} 
    \big(D_{U_1}^{\alpha_1} \cdots D_{U_{n-1}}^{\alpha_{n-1}}F(U^*)\big)^2\\ 
    =\ &\sum_{i=0}^{k} {{d \choose i}}^{-1} \big(i!\big)^{-2} \norm{\nabla^{i,\perp}(F)(U^*)}_B^2,
\end{align}
since $\left(\sqrt{{ i \choose \alpha_1, \ldots, \alpha_{n-1}}}U_1^{\alpha_1} \cdots U_{n-1}^{\alpha_{n-1}}\right)_{\alpha_1 + \cdots + \alpha_{n-1}=i}$ is an orthonormal basis of $\H_{i,n-1}$.
\end{proof}

\begin{lemma}\label{lem:T1}  For $U\in \Sn$,
$$
\norm{T^1_{U}(F)}^2 = (1- {d}) F(U^*)^2 + \frac{1}{d} \norm{\nabla_{U} F}^2
$$
\end{lemma}
\begin{proof}
Since $\scp{\nabla_{U} F, U}= d\, F(U^*)$, $\nabla_{U} F = d \, F(U^*)\, U + \nabla_{U}^{\perp}F$ and  $\norm{U}=1$ on $\Sn$, we deduce from Lemma \ref{lem:norm_Delta_K}, that
\begin{align}
 \norm{T^1_{U}(F)}^2  
&= F(U^*)^2 + \frac{1}{d} \norm{\nabla_{U}^\perp F}^2
     = F(U)^2 + \frac{1}{d} \norm{\nabla_{U}F - d \, F(U^*) \, U }^2 \\
& = F(U)^2 + \frac{1}{d} \big( \norm{\nabla_{U} F}^2 - 2 \, d \, F(U^*)\, (\nabla_{U} F, U) + d^2\,  F(U^*)^2\big)\\
& = F(U)^2 + \frac{1}{d} \big( \norm{\nabla_{U} F}^2 -  {d^2} F(U^*)^2 \big)
= (1- {d}) F(U^*)^2 + \frac{1}{d} \norm{\nabla_{U} F}^2.
\end{align}
\end{proof}

Notice that $T^{k}_{U}(F)=F$ if $F \in \Ok_{d,U}$.
Notice also that $H=R^{k+1}_{U}(F) = F - T^k_{U}(F)$ is singular at order $k$ at $U$, since $T^{k}_{U}(H)=0$ so that $\nabla^{j}_{U} (H)= 0$ for $j\le k$, and $D_{U_1}^{\alpha_1} \cdots D_{U_{n-1}}^{\alpha_{n-1}} H(U^*)=0$ for $\alpha_1+ \cdots + \alpha_{n-1}\le k$ and for any orthogonal basis $(U, U_1, \ldots, U_{n-1})$ of $\kk^n$.
Thus $R^{k+1}_{U}(F)\in \disc^{k}_{d, U}$.

\subsection{Singularity loci of osculating varieties}\label{sec:ok sing}

In this section, we analyse the singularities of the osculating variety $\Okdn$.
For $\frac{d}{2}\le k \le d$, we define
$$
\cl D_{d}^{k}=\{ X_0^{d-k} X_1^{d-k} G,\  X_0 , X_1\in \Hun \text { $X_0, X_1$ coprime} , G \in \H_{2k-d, n} \}.
$$

\begin{proposition}\label{prop:osc dim sing} 
Let $\kk=\C$. The variety $\Okdn$ is of dimension ${k+n-1 \choose k} + n-1$. Its singular locus is $\sing(\Okdn)= \cl O^{k-1}_{d}$ for $k< \frac{d}{2}$ and $\sing(\Okdn)= \cl O^{k-1}_{d} \cup \cl D_{d}^{k}$ for $\frac{d}{2}\le k \le d$.

Moreover, the multiplicity of $\Okdn$ at a smooth point of \AP{$\cl O^{k-1}_{d}\setminus \cl D_{d}^{k}$} is $2^{n-1}$. 
\end{proposition}

\begin{proof}
As we show a singular point $F_0 \in \sing(\Okdn)$, with $F_0\neq 0$, is either a \AP{multiple} point of the parametrization $\psi_k$ defined in \eqref{eq:psik}, in the sense that $F_0\in \cl D_{d}^{k}$, or the image of a point where $d \psi_k$ is not of maximal rank. 
For the computation we replace it by an equivalent parameterisation $\Psk$, see \eqref{eq:param-oscul2} below.  Note that  $\Psk$ is injective over a neighborhood of $F_0 = X_0^{d-k} G$ if $F_0  $ is not in the closure of $\cl D_{d}^k$.  Let us first decompose $G = \sum_{j=0}^{k} X_0^{k-j} G_{j}$ and define  
\begin{align}\label{eq:param-oscul2}
& \Psk  : X_0^\perp \times \hat \H_{0} \times  \cdots \times \hat \H_{k} \to 
\cl{O}^k_{d}, \\ 
&\Psk (K, H_0, \dots, H_k) = \sum_{i=0}^{k} (X_0 + K) ^{d-i} (G_i+H_i ),
\end{align}
where $\hat \H_i = \H_i(X_0^{\perp})$.  Then, 
\begin{align}\label{eq:imag_deriv}
\textrm{Im}\, \textrm{d}\Psk (0,0)= \vspan{ X_0^{d-k} H : H\in \H_{k}} 
+ \vspan{X_0^{d-k-1} G_k K: K \in X_0^\perp},
\end{align}
and it equals the tangent space $T_{F_0}\Okdn$ if $F_0 \in \Okdn \setminus (\cl O^{k-1}_{d} \cup \cl D_{d}^{k})$.  

From now on, let us assume that $F_0=X_0^{d-k}G \ne 0$.  Then 
either $\dim(\textrm{Im}\, \textrm{d}\psi_k (X_0,G)) = {k+n -1 \choose k} + n - 1$ if $X_0$ does not divide $G$, or, otherwise,  it is of dimension ${k+n -1 \choose k}$.  In the first case the image of a neighborhood of $(X_0, G_0)$ is a smooth subvariety and in the second one, as we will show, it is singular.  
Clearly the dimension of $\Okdn$ equals the maximal rank of $\textrm{d}\psi_k$ that is ${k+n -1 \choose k} + n - 1$.  

If $X_0$ divides $G$ then $F_0\in \cl O^{k-1}_{d}$.  Therefore, if $F_0 \in \Okdn \setminus (\cl O^{k-1}_{d} \cup \cl D_{d}^{k})$ then $X_0$ does not divide $G$ and $F_0$ is a nonsingular point of $\Okdn$.  Then the space \eqref{eq:imag_deriv} equals the tangent space $T_{F_0} \Okdn $. 

If $F_0$ is a \AP{multiple point of $\psi_k$, then there are $X_0$ and $X_1$ coprime such that }
$$
F_0= \psi_k(X_0, G_0) = \psi_k(X_1, G_1) = X_0^{d-k} G_0 =  X_1^{d-k} G_1 
$$ 
and then  
$F_0= X_0^{d-k} X_1^{d-k} G'$ with $G'\in \H_{2k-d}$. This may only happen if 
$\frac d 2 \le k \le d$, and then $F_0\in \cl D_{d}^k$. If, moreover, $d-k\ge 2$ then $F_0$ is a singular point of $\Okdn$.  Indeed, 
suppose otherwise.  We may assume that the tangent space $T_{F_0} \Okdn $ is the image \eqref{eq:imag_deriv} computed for $X_1$ and therefore all the vectors in  
$T_{F_0} \Okdn $ is divisible by $X_1$.  But then the image \eqref{eq:imag_deriv} computed for $X_0$ is divisible by $X_0$ that is impossible.    
Note also that if $d-k\ge 2$ then $\dim \cl D_{d}^k < \dim \Okdn$.

\smallskip
Now suppose that $F_0$ is a nonsingular point of $\cl O^{k-1}_{d}$ and that \AP{$F_0\in\cl O^{k-1}_{d}\setminus (\cl O^{k-2}_{d}\cup \cl D^{k-1}_d)$} (we may argue by induction on $k$ and then the latter is a consequence of the former). In particular,  $G_k=0$ and $G_{k-1}\ne 0$. The restriction of $\Psk$ to $H_k=0$, that we denote by $\Pskm$, parameterises $\cl O^{k-1}_{d}$ near $F_0$.  
The image of the differential of $\Pskm$ at $(0,0)$ is then  
\begin{align*}
 T_{F_0} \cl O^{k-1}_{d} =  \textrm{Im}\, \textrm{d}\Pskm (0,0) = X_0^{d-k} \vspan{ G_{k-1}K + X_0 H'; K \in X_0^\perp, H'\in \H_{k-1}} . 
\end{align*}  
Because $G_k=0$ the image of the differential of $\Psk$ at $(0,0)$ is the vector space 
\begin{align*}
\hat T  = \textrm{Im}\, \textrm{d}\Psk (0,0) = X_0^{d-k} \H_{k} = \sum_{l=0}^{k} X_0^{d-l} \H_{l}(X_0^{\perp}) .
\end{align*}

Let $N$ be the normal slice to $T_{F_0} \cl O^{k-1}_{d}$ at $F_0$.  
Our goal is now to parameterise $\Okdn \cap N$ in order to describe its singularities.  It is convenient to consider first 
 $N_T= F_0+ \hat T^\perp =  \sum_{l=k+1}^{d} X_0^{d-l} \H_{l}(X_0^{\perp})$, the normal affine complement of $\hat T$ at $F_0$. Clearly $N_T $ is an affine subspace of $N$.

Write for short 
$H = \sum_{j=0}^{k} X_0^{k-j} H_{j}$ and 
 $\Psk (K, H)= \Psk (K, H_0, \dots, H_k)$, where $H_j \in \H_{j}(X_0^{\perp})$ .  
By Taylor formula at $(0,0)$ we have 
\begin{align}\label{eq:TaylorCase2-2}
\Psk (K, H ) - F_0 & =  \sum_{j=0}^{k}  X_0^ {d-j}H_j  + 
\sum_{i=1}^{d-k} \binom{d-k} i X_0^ {d-k-i} H_k K^i 
 \\
& + \sum_{j=0}^{k-1} \sum_{i=1}^{d-j} 
\binom{d-j} i X_0^ {d-j-i}(G_j+ H_j) K^i .  
\end{align}
We  divide $\Psk (K, H ) - F_0 $ into two parts: the first one, denoted by $\Psi_T$, with values in  $\hat T $ and the other one, denoted $\Psi_R$, with values in $\hat T^{\perp}$:  
\begin{align}\label{eq:psiT+psiR}
& \Psk(K, H )  = F_0 + \Psi_T (K, H ) + \Psi_R(K, H )\\
\Psi_T (K, H )&  =  \sum_{j=0}^{k}  X_0^ {d-j}H_j  + 
\sum_{j=0}^{k-1} \sum_{i=1}^{k-j} 
\binom{d-j} i X_0^ {d-j-i}(G_j+ H_j) K^i  \\ 
&  =  \sum_{s=0}^{k}  X_0^ {d-s} \Big ( H_s  + 
\sum_{i+j=s, i\ge 1} 
\binom{d-j} i (G_j+ H_j) K^i \Big )  \\ 
\Psi_R (K, H )&  = \sum_{i=1}^{d-k} \binom{d-k} i X_0^ {d-k-i} H_k K^i  + 
\sum_{j=0}^{k-1} \sum_{i=k-j+1}^{d-j} 
\binom{d-j} i X_0^ {d-j-i}(G_j+ H_j) K^i . 
\end{align}

Recall that $N_T$ is the affine orthogonal complement to $F_0 + \hat T$ at $F_0$.  
Thus $\Psk (K, H ) \in N_T$ if and only if 
$\Psi_T(K, H )=0$, i.e. all the coefficients at $X_0^{d-s} $ of $\Psi_T$ vanish.  This gives a triangular  equations for the $H_i$ in terms of $G_i$ and $K$.  It has a unique solution of the form  
\begin{align}\label{eq:systemtriangE}
& H_{0} = 0, \quad  H_{1} = - d G_{0}K ,  \\ 
& H_{2} = - (d-1) G_{1}K +  c_{2,1}  G_{0}K^{2} ,\\
& \cdots \\
&  H_{s } = -(d-s+1) G_{s-1}K + \sum_{i=1}^{s-1}  (c_{si} G_{s-i-1}K^{i+1}) ,\\
& \cdots \\
&  H_{k} = -(d-k+1) G_{k-1}K   +\sum_{i=1}^{k-1}  (c_{ki}  G_{k-i-1}K^{i+1}) ,
\end{align}
for some constants $c_{si}$. 
Thus we may compute $H$ in terms of $K$ (and $G$ that is considered fixed):
\begin{align}
H(K) & = -(d-k+1) G_{k-1}K - \sum_{i=2}^k (d-k+i) G_{k-i}X_0^{ i-1} K + (K^2),   
\end{align}
where the notation $(K^2)$  means the sum of terms having  $K^2$  as a factor. 
Note that all $H_{s}$ are divisible by $K$.   After \eqref{eq:psiT+psiR} we get 
\begin{align}\label{eq:develop2A}
\quad \Psk (K, H(K)) & - F_0   = \Psi_R(K,H(K))   = c \, G_{k-1} X_0^ {d-k-1} K^2  +(K^3)  , 
\end{align}
where $c= \textstyle{\binom {d-k+1} {2} }-(d-k+1)(d-k) \ne 0$. 

Given an orthonormal basis $U_0 = X_0, U_1, \ldots, U_{n-1}$ and writing $K = \sum_{j=1}^{n-1} t_i U_i $ 
we obtain a parametrisation 
of $\Okdn\cap N_T$ by $t=(t_1, \ldots, t_{n-1})$.
\begin{align*}
\Psk (K, H(K))  - F_0   = \Psi_R(K,H(K))   
& =  c \,  G_{k-1}X_0^ {d-k-1} \bigl  (\sum_{j=1}^{n-1} t_i U_i \bigr  )^2  + \cdots .
\end{align*}

Let us compute the Euler characteristic 
of the complex link $\C lk ((\Okdn \cap N_T ,F_0)$ of $\Okdn\cap N_T $ at $F_0$.  The complex link is the Milnor fibre of the generic linear form  at $F_0$ restricted to $\Okdn \cap N_T$.  Its pull-back to $\Psk^{-1} (\Okdn \cap N_T)$ is a generic linear combination of the products $t_i t_j$, 
$2\le i\le j\le n$, plus higher-order terms in $t$.  It is therefore the Milnor fibre of an $A_1$ singularity.  Its Milnor number is equal to $1$ and its Euler characteristic is $1+ (-1)^{n-2}$.  Therefore, $\Okdn$ is singular at $F_0$.  Otherwise this complex link would be contractible. 

\begin{remark}\label{rem:Eu_obstruction}
In the case of tangential variety $\cl O^{1}_{d}= \Tan$,  $\cl O^{}_{d}=\Ver$, we have $\hat T = T_{F_0}(\Ver)$ and therefore $N=N_T$ is a normal slice to $\Ver$ at $F_0$. Then the Euler characteristic of the complex link  $\C lk ((\Tan \cap N,F_0))$ is equal to the Euler obstruction of $\Tan$ at $F_0$.  This shows that $Eu_{\Tan} (F_0)= 1+ (-1)^{n}$.
\end{remark}

If $F_0 \in \sing (\cl O^{k-1}_{d})$ then $F_0$ is in the Zariski closure of the regular part of  $\cl O^{k-1}_{d}$ that, as we have just shown, is a subset of $\sing(\Okdn)$.  Therefore 
$F_0 \in \sing(\Okdn)$ because the singular locus is closed.

\medskip
 To parameterise $\Okdn \cap N$, where $N$ is the affine orthogonal complement to $F_0 + T_{F_0}(\cl O^{k-1}_{d})$ at $F_0$, we use again the parameterisation  \eqref{eq:psiT+psiR}.  The condition for 
$\Psk (K, H) \in  \Okdn \cap N$ gives the system of equations \eqref{eq:systemtriangE} with the equation for $s=k$ replaced by 
\begin{align}
H_{k} +(d-k+1) G_{k-1}K   -\sum_{i=1}^{k-1}  (c_{ki}  G_{k-i-1}K^{i+1}) \in (G_{k-1} \hat \H_{1, n-1} )^\perp  , 
\end{align}
where $(G_{k-1} \hat \H_{1, n-1} )^\perp$ is the orthogonal complement of $ G_{k-1} \hat \H_{1, n-1}$ in $\hat \H_{k, n-1}$.
Thus the solution of the system admits a parameter $\tilde H \in (G_{k-1} \hat \H_{1, n-1} )^\perp$ and the last equation of \eqref{eq:systemtriangE} is replaced by 
\begin{align}
H_{k} = -(d-k+1) G_{k-1}K   +\sum_{i=1}^{k-1}  (c_{ki}  \hat G_{k-i-1}K^{i+1})  
+ \tilde H = H_k (K) + \tilde H.  
\end{align}
This gives a parameterisation of $\Okdn \cap N$ that extends the parameterisation \eqref{eq:develop2A} of $\Okdn \cap N_T$, 
\begin{align}\label{eq:develop-diff}
&   \Psk (K, H(K) + \tilde H)  = \Psk(K, H(K)) +  (X_0+K) ^{d-k} \tilde H \\ 
 = F_0+  c \, G_{k-1} X_0^ {d-k-1} K^2 
 & +   X_0^{d-k} \tilde H + 
(d-k)   X_0^ {d-k-1} \tilde H K +  \binom {d-k} 2   X_0^ {d-k-1}\tilde H K^2 + (K^3)  .
\end{align}
If we denote  $F_1=   \Psk (K,H(K)+\tilde H )$ then   
\begin{align}\label{eq:distance}
\|F_1-F_0\|\simeq \|\tilde H, K^2\|. 
\end{align}  

\smallskip
\noindent
\textbf
{Computation of the multiplicity.} 
Let $\Phi: \H_{d}\to \C^{n-1}\times \C^M$, where $M = \dim (G_{k-1} \hat \H_{1, n-1} )^\perp$, so that $(n-1)+M$ is the codimension of $\cl O^{k-1}_{d}$ in $\Okdn$, 
be a generic linear map.  Then consider 
$\varphi  : (\H_{d} \times (G_{k-1} \hat \H_{1, n-1} )^\perp,0) \to (\C^{n-1}\times \C^M,0)$ defined by 
\begin{align*}
\varphi (K,\tilde H)= \Phi (\Psk (K, H(K) +\tilde H)  - F_0) . 
\end{align*}
The local degree of this map, that by \eqref{eq:develop-diff} equals $2^{n-1}$, is the multpilicity of 
$\Okdn$ at $F_0$.  

This concludes the proof of the proposition.
\end{proof}
This proposition generalizes \cite[Proposition 3.13]{BarrilliMourrainTaufer2025}, exploiting  different techniques, and computing the multiplicity of $\Okdn$ along $\smooth(\cl O^{k-1}_d)$.

\smallskip 

\begin{proposition}\label{prop:Whitney}
The $k$-th osculating variety $\Okdn$ is Whitney equisingular along the regular part of $\cl O^{k-1}_{d}$. 
\end{proposition}

\begin{proof}
We prove Whitney equisingularity of $\Okdn$ along $\cl O^{k-1}_{d}$, by verifying the Kuo-Verdier condition (w) at a smooth point $F_0$ of $\Odn^{k-1}$, (see \cite{Trotman2020} for a recent survey on Stratification Theory).  Thus, by Proposition \ref{prop:osc dim sing} we may write $F_0=X_0^{d-k} G$, with $G = \sum_{j=0}^{k} X_0^{k-j} G_{j}$, $G_j\in \hat H_{j,n-1}$ with $G_k=0$ and $G_{k-1}\ne 0$.

The Kuo-Verdier condition (w) at $F_0$  means the following.  There are constants $c'$ and $\varepsilon >0$ 
such that for every $F_0' \in \Odn^{k-1}$ and every $F_1\in  \smooth(\Odn^{k})$, such that $\norm{F'_0-F_0} < \varepsilon$, $\norm{F_1-F_0} < \varepsilon$, we have  
\begin{align}\label{eq:bound(w)}
\dist(T_{F'_0}(\Odn^{k-1}),T_{F_1}(\Okdn)) \le c' \norm{F_1-F'_0} 
\end{align}
where by $\dist (A,B)$ between two vector spaces of $\kk^n$ we mean the following
\begin{align}
\dist (A,B) = \sup_{v\in A^{\times}}\inf_{w\in B^{\times}} |\sin \theta (v,w)| = 
\sup_{v\in A^{\times}} \sup_{u\in (B^\perp)^{\times}} \frac {\|u,v\|} {\|u\| \|v\|} , 
\end{align}
where $\theta (v,w)$ is the angle between $v$ end $w$.  This distance is zero if and only if $A\subset B$. By Kuo and Verdier, for real analytic sets, the Kuo-Verdier condition (w) implies Whitney's conditions (a) and (b).  For complex analytic ones, by Teissier and Henry\&Merle,  it is equivalent to Whitney's conditions.  See Trotman's survey for exact references. 

To verify \eqref{eq:bound(w)} for $F_1$ it suffices to verify it for $F'_0$ the closest point to it.  Therefore it suffices to show \eqref{eq:bound(w)} only for $F'_0$ from a neighbourhood of $F_0$ and for $F_1\in N_{F'_0}$, where $N_{F'_0}$ is the affine orthogonal complement to $F'_0 + T_{F'_0}(\cl O^{k-1}_{d})$ at $F'_0$, provided the constants $c', \varepsilon $ can be chosen  independently of $F'_0$.  Thus in what follows we show  \eqref{eq:bound(w)} for $F'_0=F_0$ and $F_1\in N=N_{F_0}$, bearing in mind that the constants should be independent of the (local) choice of $F_0$.  

By \eqref{eq:tangentO_d} we may decompose 
\begin{align}
T_{F_0}(\Odn^{k-1}) = &  T_{0,G} \oplus T_{0,H}, \text { where }  \\ 
&  T_{0,G} := X_0^{d-k} G_{k-1}\, \hat \H_{1,n-1}  \text { and }  T_{0,H}:= X_0^{d-k+1} \, \H_{k-1} .
\end{align}

Let $F_1= (X_0+K_0)^{d-k} H \in \smooth(\Okdn)$ 
with $\norm{K_0} \leq \epsilon$, $\norm{H_1} \leq \epsilon$.
For $P_0 = X_0^{d-k+1} H_0  \in T_{0,H} $ with $H_0 \in \H_{k-1}$, 
let $P_1 = (X_0+K_0)^{d-k} (X_0 -(d-k) K_0) H_0 \in T_{F_1}(\Okdn)$.  
Then 
\begin{align}
P_1 & = (X_0+K_0)^{d-k} (X_0 -(d-k) K_0) H_0 \\ 
& = X_0^{d-k-1} (X_0 +(d-k) K_0)
(X_0 -(d-k) K_0) H_0 + (K_0^2)  \\
& = X_0^{d-k+1} H_0 + (K_0^2)  = P_0 +  (K_0^2) .
\end{align}
Therefore, then after \eqref{eq:distance}, if $F_1\in N$  then 
\begin{align}\label{eq:boundTH}
\dist(T_{0, H},T_{F_1}(\Okdn) ) \leq  c \, \norm{F_1-F_0} ,
\end{align}
where the constant $c$ depends on the norms of $H_0$ and $K_0$.  

It remains to show the inequality 
\begin{align}
\dist(T_{0,G} ,T_{F_1}(\Okdn)) \leq  c \, \norm{F_1-F_0} .
\end{align} 
Because of \eqref{eq:tangent_orbit1} this follows from the next lemma. Note that we do not need that $F_1\in N$.

\begin{lemma}
For every $F_0\in \Hdn, F_0\ne 0,$ there is  a neighborhood $U$ of $F_0$ in $\Hdn$ and a smooth map 
\begin{align}
T:U\to G(N,m) ,
\end{align}
where $m=\dim T_{F_0} Orb(F_0)$, such that $T(F_0) = T_{F_0} Orb(F_0) $ and $T(F)\subset T_{F} Orb(F) $ for every $F\in U$. 
\end{lemma}

\begin{proof}
Let $G: O(n) \times \Hdn \to \Hdn$ be the group action and denote by $D_g : TO(n ) \times \Hdn \to T\Hdn$ its derivative with respect to the first factor and for $F\in \Hdn$ we denote 
$D_g (F) : T_{Id} O(n ) \to T_F\Hdn = \Hdn$ its restriction over $F$. 

Let $V\subset T_{Id} O(n )$  be the vector subspace of dimension $m$ such that 
$T_{F_0} Orb(F_0) = D_g (F_0) (V)$.  
Then we may take as $T$ the map  $T(F): = D_g (F) (V)$ and as $U$ the set of $F$ for which $\rank D_g (F)|_V =m$.      
\end{proof}

This end the proof of Proposition \ref{prop:Whitney}.  
\end{proof}



\subsection{Formulas for the Euler characteristic of generic hypersurfaces and complete interesections}\label{sec:hirzebruch}

Let $X^n_{d_1, . . . , d_k} \subset \P^{n+k}$ be a smooth intersection of $k$ hypersurfaces of degrees $d_1, \ldots , d_k$ (so $\dim X^n_{d_1, . . . , d_k} =n$).  Denote 
$$
\chin{n}{d_1, \ldots , d_k} = \chi (X^n_{d_1, . . . , d_k}). 
$$
Then Hirzebruch's formula of \cite[Appendix I, Theorem 22.1.1]{Hirzebruch} gives a generating function for the topological Euler characteristic 
\begin{align}\label{eq:Hirzebruch}
\sum_{n\ge 0} \chin{n}{d_1, \ldots , d_k} z^n = 
\frac {d_1 \cdots  d_k} {(1-z)^2}  \prod_i \frac 1 {1+(d_i-1)z} .  
\end{align}

\begin{proposition}\label{prop:fromHirzebruch}
The following formulas follow from \eqref{eq:Hirzebruch}
\begin{enumerate}
\item 
If $k=1$ we have 
\begin{align}\label{eq:Hirzebruch1}
\chin{n}{d}  = 
\frac {(1-d)^{n+2} - 1} {d} + n+2 .  
\end{align}
\item 
If $k=2$ and $c\ne d$ we have 
\begin{align}\label{eq:Hirzebruch2}
\chin{n}{c,d}  = 
\frac { d^2((1-c)^{n+3} - 1) - c^2 ((1-d)^{n+3} - 1)} {cd(d-c)} + n+3 .  
\end{align}
and 
\begin{align}\label{eq:Hirzebruchdiff}
(-1)^{n} \chi(\P^{n}  \setminus (X^{n-1}_{c} \cup X^{n-1}_{d}))  = 
\frac {(d-1)^{n+1} - (c-1)^{n+1}} {d-c}.  
\end{align}
\end{enumerate}
\end{proposition} 

\begin{proof} 
First note that 
\begin{align*}
\sum_{n\ge 0} \chi(\P^n) z^n = \sum_{n\ge 0} (n+1) z^n =  (1+z+ z^2 +\cdots )^2 
=\frac 1 {(1-z)^2 }  .  
\end{align*}
and then, with the convention $\chin{n}{d}=0$ if $n<0$, 
\begin{align*}
 \sum_{n\ge 0} (\P^{n} \setminus X^{n-1}_{d}) z^{n} & =  \sum_{n\ge 0} (n+1) z^n - 
\sum_{n\ge 0} \chin{n-1}{d} z^n  
 = z \frac {d} {(1-z)^2}  \frac 1 {1+(d-1)z} - \frac 1 {(1-z)^2 } \\
& =   \frac {1+(d-1)z -dz} {(1-z)^2 (1+(d-1)z)} =  \frac {1} {(1-z) (1+(d-1)z)}  \\
& =  (1+z+ z^2 +\cdots ) (1+ (1-d)z+ (1-d)^2 z^2 +\cdots ) \\
& = \sum_{n\ge 0}   (1+ (1-d)+ (1-d)^2  +\cdots + (1-d)^n) z^n \\ 
&  = - \sum_{n\ge 0}  \frac {(1-d)^{n+1} - 1} {d} z^n .  
\end{align*}
That gives $\chin{n-1}{d}  = 
\frac {(1-d)^{n+1} - 1} {d} + n+1 $ as claimed.

\begin{align*}
& \sum_{n\ge 0} \chi(\P^n  \setminus (X^{n-1}_{c} \cup X^{n-1}_{d})) z^n = 
\sum_{n\ge 0} (\chi(\P^n) -  \chi(X^{n-1}_{c}) - \chi(X^{n-1}_{d}) + \chi (X^{n-2}_{c,d})) z^n \\ 
& = \frac {1} {(1-z)^2} - \frac {cz}{(1-z)^2(1+(c-1)z)} -\frac {dz}{(1-z)^2(1+(d-1)z)} \\ 
& \qquad  + \frac {cdz^2}{(1-z)^2(1+(c-1)z)(1+(d-1)z)} \\ 
& = \frac {(1+(c-1)z)(1+(d-1)z) - dz (1+(c-1)z)- cz (1+(d-1)z) + cdz^2  } 
{(1-z)^2 (1+(c-1)z)(1+(d-1)z)}\\ 
& = \frac {(1-z)^2 } 
{(1-z)^2 (1+(c-1)z)(1+(d-1)z)}\\ 
& =  (1+(1-c)z+(1-c)^2z^2+ \cdots )(1+(1-d)z+(1-d)^2z^2+ \cdots ).  
\end{align*}
that gives 
\begin{align}
\sum_{n\ge 0}  \chi(\P^{n}  \setminus (X^{n-1}_{c} \cup X^{n-1}_{d}))z^n   = \sum_{n\ge 0} 
\frac {(1-c)^{n+1} - (1-d)^{n+1}} {d-c} z^n  
\end{align}
that implies \eqref{eq:Hirzebruchdiff}.  
We use \eqref{eq:Hirzebruchdiff} to show  \eqref{eq:Hirzebruch2}. 
Because 
\begin{align*}
\chi(\P^n  \setminus (X^{n-1}_{c} \cup X^{n-1}_{d})) = 
(\chi(\P^n) -  \chi(X^{n-1}_{c})) + (\chi(\P^n)  - \chi(X^{n-1}_{d})) 
- (\chi(\P^n) -  \chi (X^{n-2}_{c,d}))  
\end{align*}
we have 
\begin{align*}
&  \chin{n-2}{c,d} - (n+1) = \frac {(1-c)^{n+1} - (1-d)^{n+1}} {d-c} 
+ \frac {(1-d)^{n+1} - 1} {d}
+ \frac {(1-c)^{n+1} - 1} {c} \\ 
& = \frac {1} {(d-c)cd} \Big ( cd ((1-c)^{n+1} - (1-d)^{n+1}) +
(d-c)c ((1-d)^{n+1} - 1) \\
& + (d-c)d ((1-c)^{n+1} - 1) \Big )   \\ 
& = \frac {1} {(d-c)cd} \Big ( -c^2 (1-d)^{n+1} + d^2 (1-c)^{n+1}  
-d^2 + c^2  \Big ).   
\end{align*}
\end{proof}

\begin{corollary}\label{cor:d=2}
For $d=2$ we have
\begin{align}\label{eq:Hirzebruch1_d=2}
- \chi(\P^{n+1} \setminus X^n_2) = 
\chin{n}{2} -(n+2)  = 
\frac {(1-2)^{n+2} - 1} {2} = 
\begin{cases}
0 \qquad \text{for $n$ even }\\
-1 \quad \text{ for $n$ odd }
\end{cases} .  
\end{align}
In particular, $((n+1) - \chin{n}{2})^2 + (-1)^{n+1}((n+1) - \chin{n}{2}) =0$.
\end{corollary}

\bibliographystyle{alpha}
\bibliography{references.bib}

\end{document}